\documentclass[11pt]{article}

\usepackage{hogwild}
\usepackage{fullpage}
\usepackage{hyperref}
\usepackage[numbers,square]{natbib}
\usepackage{xspace}
\usepackage{psfrag}

\newcommand{\wb}[1]{\overline{#1}}
\newcommand{\hogwild}{\textsc{Hogwild!}\xspace}
\newcommand{\noise}{\xi}
\newcommand{\noisier}{\zeta}
\newcommand{\errmat}{E}
\newcommand{\stepsize}{\alpha}
\newcommand{\steppow}{\beta}

\newcommand{\delay}{M}

\newcommand{\error}{\Delta}
\newcommand{\normal}{\mathsf{N}}

\newcommand{\opt}{^\star}

\newcommand{\statrv}{W}
\newcommand{\statval}{\omega}
\newcommand{\statdomain}{\Omega}
\newcommand{\resid}{R}
\newcommand{\lyap}{V}
\newcommand{\hess}{H}
\newcommand{\lipobj}{G}
\newcommand{\event}{\mc{E}}

\newcommand{\badratio}{\mathsf{K}}  
\newcommand{\badratioevent}{\mc{K}}  
\newcommand{\delmoment}{\tau}  
\newcommand{\npow}{\rho}  

\usepackage{color}

\definecolor{innerboxcolor}{rgb}{.9,.95,1}
\definecolor{outerlinecolor}{rgb}{.6,0,.2}

\renewcommand{\theassumption}{\Alph{assumption}}

\title{Asynchronous stochastic convex optimization}
\author{John C.\ Duchi \\ \texttt{jduchi@stanford.edu}
  \and Sorathan Chaturapruek \\
  \texttt{sorathan@stanford.edu} \and
  Christopher R\'e \\
  \texttt{chrismre@stanford.edu}
}

\begin{document}

\begin{center}
  {\LARGE Asynchronous stochastic convex optimization} \\
  \vspace{.4cm}
  {\large John C.\ Duchi$^1$ ~~~ Sorathan Chaturapruek$^2$ ~~~
    Christopher R\'e$^2$} \\
  \vspace{.3cm}
  {\large Stanford University} \\
  \vspace{.1cm}
  {\large Departments of $^1$Statistics, $^1$Electrical Engineering,
    and $^2$Computer Science} \\
  \vspace{.15cm}
  \texttt{\{jduchi,sorathan,chrismre\}@stanford.edu}
\end{center}

\begin{abstract}
  We show that asymptotically, completely asynchronous stochastic gradient
  procedures achieve optimal (even to constant factors) convergence rates
  for the solution of convex optimization problems under nearly the same
  conditions required for asymptotic optimality of standard stochastic
  gradient procedures. Roughly, the noise inherent to the stochastic
  approximation scheme dominates any noise from asynchrony. We also give
  empirical evidence demonstrating the strong performance of asynchronous,
  parallel stochastic optimization schemes, demonstrating that the
  robustness inherent to stochastic approximation problems allows
  substantially faster parallel and asynchronous solution methods.
\end{abstract}

\section{Introduction}

We study a natural asynchronous stochastic gradient method for the solution
of minimization problems of the form
\begin{equation}
  \label{eqn:convex-problem}
  \minimize ~ f(x) \defeq \E_P[F(x; \statrv)]
  = \int_{\statdomain} F(x; \statval) dP(\statval),
\end{equation}
where $x \mapsto F(x; \statval)$ is convex for each $\statval \in
\statdomain$, $P$ is a probability distribution on $\statdomain$, and the
vector $x \in \R^d$.  Stochastic gradient techniques for the solution of
problem~\eqref{eqn:convex-problem} have a long history in optimization,
starting from the early work of \citet{RobbinsMo51} and continuing on
through \citet{Ermoliev69} and \citet{PolyakJu92} and
\citet{NemirovskiJuLaSh09}. The latter two papers show how certain long
stepsizes and averaging techniques yield more robust and asymptotically
optimal optimization schemes, and we show how their results extend to
practical parallel and asynchronous optimization settings.

We consider an extension of previous stochastic gradient methods to a
natural family of \emph{asynchronous} gradient methods (see, e.g., the book
of~\citet{BertsekasTs89}), where multiple processors can draw samples from
the distribution $P$ and asynchronously perform updates to a centralized
parameter vector $x$. Our iterative scheme is based on the \hogwild
algorithm of \citet{NiuReReWrNi11}, which is designed to asynchronously
solve certain stochastic optimization problems in multi-core environments,
though our analysis and iterations are different. In particular,
we study the following procedure, where each processor runs asynchronously
and independently of the others, though they maintain a shared iteration
counter $k$; each processor performs the following:
\begin{enumerate}[(i)]
\item \label{item:x-read-convex} Processor reads current problem data $x$
  and counter $k$
\item \label{item:grad-comp-convex} Processor draws a random sample $\statrv
  \sim P$, computes $g = \nabla F(x; \statrv)$, and increments a centralized
  counter $k$
\item \label{item:update-step-convex} Processor updates $x \leftarrow x -
  \stepsize_k g$ via sequential updates $[x]_j = [x]_j - \stepsize_k [g]_j$ for
  each coordinate $j \in \{1, \ldots, d\}$.
\end{enumerate}
In the
iterations~\eqref{item:x-read-convex}--\eqref{item:update-step-convex}, the
scalars $\stepsize_k$ are a non-increasing stepsize sequence.

\subsection{Main results and outline}

The thrust of our results is that because of the noise inherent to the
sampling process for $W$, the errors introduced by asynchrony in the
iterations~\eqref{item:x-read-convex}--\eqref{item:update-step-convex} are
asymptotically negligible: they do not matter. Even more, we can efficiently
construct an $x$ from the asynchronous process possessing optimal
convergence properties and asymptotic variance.  This has consequences for
solving stochastic optimization problems on multi-core and multi-processor
systems; we can leverage parallel computing without performing any
synchronization, so that given a machine with $m$ processors, we can read
data and perform updates $m$ times more quickly than what is possible with a
single processor, and the error from reading stale information on $x$
becomes asymptotically negligible.  In Section~\ref{sec:main-results}, we
make these optimality claims formal, presenting our main convergence
theorems about the asynchronous
iterations~\eqref{item:x-read-convex}--\eqref{item:update-step-convex} for
solving the problem~\eqref{eqn:convex-problem}.  Our main result,
Theorem~\ref{theorem:convex-optimization}, gives explicit conditions under
which an asynchronous stochastic gradient procedure converges at the optimal
rate and with optimal asymptotic variance, and we give applications to
specific stochastic optimization problems in Section~\ref{sec:examples}.  We
also provide a more general result (Theorem~\ref{theorem:nonlinear}) on the
asynchronous solution of more general stochastic operator equations, again
demonstrating that asynchrony introduces asymptotically less noise than that
inherent in the stochastic problem itself.  While we give explicit
conditions under which our results hold, we note that roughly all we require
is a type of local strong convexity around the optimal point $x\opt =
\argmin_x f(x)$, that the Hessian of $f$ be positive definite near $x\opt$,
and a Lipschitz (smoothness) condition on the gradients $\nabla f(x)$.

In addition to theoretical results, in Section~\ref{sec:experiments}
we give empirical results on the power of parallelism and asynchrony in
the implementation of stochastic approximation procedures. Our experiments
demonstrate two results: first, even in non-asymptotic finite-sample settings,
asynchrony introduces little degradation in solution quality, regardless of
data sparsity (a common assumption in previous analyses); that is,
asynchronously-constructed estimates are statistically efficient. Second, we
show that there is some subtlety in implementation of these procedures in real
hardware; while increases in parallelism lead to concomitant linear
improvements in the speed with which we compute solutions to
problem~\eqref{eqn:convex-problem}, in some cases we require strategies to
reduce hardware resource competition between processors to achieve the full
benefits of asynchrony.

\subsection{Related work}

Several researchers have provided and analyzed asynchronous algorithms for
optimization. The seminal work of \citet{BertsekasTs89} provides a
comprehensive study both of models of asynchronous computation and analyses
of asynchronous numerical algorithms, including coordinate- and
gradient-descent methods.  The results of theirs relevant to our work are
roughly of two types. For non-stochastic problems, they show (roughly)
linear convergence of iterative methods assuming the iterations satisfy
certain contractive properties, which roughly correspond to variants of
diagonal dominance of the Hessian of $f$ (see, for example~\cite[Chapters
  6.3 and 7.5]{BertsekasTs89}). For stochastic problems~\cite[Chapter
  7.8]{BertsekasTs89}, they show results that have a similar flavor to ours:
errors due to asynchrony scale approximately quadratically in the stepsize
$\stepsize$, while gradient information scales linearly with $\stepsize$ so
that it dominates other errors. Bertsekas and Tsitsiklis use this error
scaling to show that stepsize choices of the form $\stepsize_k \approx 1/k$
guarantee asymptotic convergence under models of asynchrony with bounded
delay. They leave open, however, a few interesting questions, namely,
attainable rates of convergence for asynchronous stochastic procedures, the
effects of unbounded delays, and what optimality guarantees are possible
relative to synchronous executions.

Due to their simplicity and dimension-independent convergence properties,
stochastic and non-stochastic gradient methods have become extremeley
popular for large-scale data analysis and optimization problems
(e.g.~\cite{NemirovskiJuLaSh09, Nesterov09, DuchiHaSi11, DekelGiShXi12,
  DeanCoMoChDeMaRaSeTuYaNg12}).  Consequently, with the advent of multi-core
processing systems, there has been substantial work building on Bertsekas'
and Tsitsiklis's results. Much of this work shows that asynchrony introduces
negligible penalty in rates of convergence for optimization procedures under
suitable conditions, such as gradient sparsity, conditioning of the Hessian
of $f$, or allowable types of asynchrony (none, as we show, are essential).
\citet{NiuReReWrNi11} propose the \hogwild\ method and show that under
strong sparsity and smoothness assumptions on the data (essentially, that
the gradients $\nabla F(x; \statrv)$ have a vanishing fraction of non-zero
entries, that $f$ is strongly convex, and $\nabla F(x; \statval)$ is
Lipschitz for all $\statval$), convergence guarantees similar to the
synchronous case are possible. \citet{AgarwalDu11} showed under restrictive
ordering assumptions that some delayed gradient calculations have negligible
asymptotic effect. \citet{DuchiJoMc13_nips} extended
\citeauthor{NiuReReWrNi11}'s results to a dual averaging algorithm that
works for non-smooth, non strongly-convex problems, again so long as strong
gradient sparsity assumptions hold (roughly, that the probability of an
entry of $\nabla F(x; \statval)$ being non-zero is inversely proportional to
the maximum delay of any processor) and delays are bounded.  Researchers
have also investigated parallel coordinate descent solvers:
\citet{RichtarikTa15} and \citet{LiuWrReBiSr14} show how certain
``separability'' properties of an objective function $f$---meaning the
degree to which different coordinates of $x$ jointly affect $f(x)$ (rather
than $f(x)$ depending on each coordinate $x_j$ independently)---govern
convergence rate of parallel coordinate descent methods, the latter focusing
on asynchronous schemes. The conditions sufficient for fast or asynchronous
convergence of coordinate methods are similar to the diagonal dominance
conditions used by \citet[Chapter 6.3.2]{BertsekasTs89}.  Yet, as we show,
large-scale stochastic optimization renders many of these problem
assumptions unnecessary.  In particular, the asynchronous
iterations~\eqref{item:x-read-convex}--\eqref{item:update-step-convex}
retain all optimality properties of synchronous (correct) gradient
procedures, even in the face of nearly unbounded delays, and enjoy
optimal rates of convergence, even to constant pre-factors.

\paragraph{Notation}
We say a sequence of random variables or vectors $X_n$ converges in
distribution to a random variable $Z$, denoted $X_n \cd Z$, if for all
bounded continuous functions $\E[f(X_n)] \to \E[f(Z)]$.  We say that
a sequence of (finite-dimensional) random vectors $X_n$ converges in $L_p$
to a random vector $Z$, denoted $X_n \clp{p} Z$, if there is a norm
$\norm{\cdot}$ such that $\E[\norm{X_n - Z}^p] \to 0$ as $n \to
\infty$. This convergence is equivalent for any choice of the norm
$\norm{\cdot}$. We let $X_n \cp Z$ denote that $X_n$ converges in
probability to $Z$, meaning that $\P(\norm{X_n - Z} > \epsilon) \to 0$ as $n
\to \infty$ for any $\epsilon > 0$, and $X_n \cas c$ denotes almost sure
convergence, meaning that $\P(\lim_n X_n \neq c) = 0$. The notation
$\normal(\mu, \Sigma)$ denotes the multivariate Gaussian with mean $\mu$ and
covariance $\Sigma$. We let $I_{d \times d}$ denote the identity matrix
in $\R^{d \times d}$, using $I$ when the dimension is clear from context.

We use standard big-$O$ notation. For (nonnegative) sequences $a_n$ and
$b_n$, we let $a_n \lesssim b_n$ mean there exists a constant $C < \infty$
such that $a_n \le C b_n$ for all $n$, and $a_n \asymp b_n$ means that there
exist constants $0 < c \le C < \infty$ such that $c \le \liminf_n
\frac{a_n}{b_n} \le \limsup_n \frac{a_n}{b_n} \le C$. For random vectors
$X_n, Z_n$, we say $X_n = O_P(Z_n)$ if for all $\epsilon > 0$, there
exists $C < \infty$ such that $\sup_n \P(\norm{X_n} \ge C \norm{Z_n})
\le \epsilon$, while $X_n = o_P(Z_n)$ means that for all $c > 0$, $\limsup_n
\P(\norm{X_n} \ge c \norm{Z_n}) = 0$.

\section{Main results}
\label{sec:main-results}

Our main results repose on a few standard assumptions often used for the
analysis of stochastic optimization procedures, which we now detail, along
with a few necessary definitions.  We let $k$ denote the iteration counter
used throughout the asynchronous gradient procedure. Given that we compute $g
= \nabla F(x; W)$ with counter value $k$ in the
iterations~\eqref{item:x-read-convex}--\eqref{item:update-step-convex}, we let
$x_k$ denote \emph{the} (possibly inconsistent) particular $x$ used to compute
$g$, and likewise say that $g = g_k$, noting that the update to $x$ is then
performed using $\stepsize_k$.

With the update \eqref{item:update-step-convex}, we can give a more explicit
formula for $x_k$ as a function of time $k$ with a small amount of
additional notation. In particular, let $\errmat^{ki} \in \{0, 1\}^{d \times
  d}$ be a diagonal matrix whose $j$th diagonal entry is $1$ if the $i$th
gradient (i.e.\ that computed when the iteration counter is $i$)
has been incorporated into iterate $x_k$ and is $0$ otherwise. Then
the iteration~\eqref{item:x-read-convex}--\eqref{item:update-step-convex}
and index assigments imply that
\begin{equation}
  \label{eqn:iterate-formula}
  x_k = -\sum_{i=1}^{k-1} \stepsize_i \errmat^{ki} g_i.
\end{equation}
With this definition of the update matrices $\errmat^{ki}$, we then
associate a delay value $\delay_k$ for each $k$, defined by
\begin{equation*}
  \delay_k \defeq \min \{l - k : \errmat^{lk} = I_{d \times d}, l \ge k\},
\end{equation*}
or the amount of time required for all updates from the $k$th gradient to be
incorporated into the central $x$ vector.  Rather than assuming a uniform
bound on the delay, throughout, we make the following assumption on
the moments of the random variables $\delay_k$.
\begin{assumption}
  \label{assumption:delay-moments}
  There exists $\delmoment > 2$ and a constant $\delay < \infty$ such that
  \begin{equation*}
    \sup_k \E\left[\delay_k^\delmoment\right]^\frac{1}{\delmoment}
    \le \delay.
  \end{equation*}
\end{assumption}
\noindent
Assumption~\ref{assumption:delay-moments} places our asynchronous
iterations~\eqref{item:x-read-convex}--\eqref{item:update-step-convex}
somewhere between Bertsekas' and Tsitsiklis's classification of
\emph{totally asynchronous} algorithms~\cite[Chapter 6]{BertsekasTs89},
which require only that each processor performs its updates eventually, and
\emph{partially asynchronous} algorithms~\cite[Chapter 7]{BertsekasTs89},
which specify a uniform bound on any processor's delay; roughly, we have
a quantitative version of total asynchrony.
For example, if we know that processors have bounded delays, we may take
$\delmoment = \infty$ and assume that $\delay = \sup_k \delay_k < \infty$.
In more general cases, however, we can allow infrequent longer delays with,
as we shall see, negligible effect on our results except that the allowable
stepsizes $\stepsize_i$ are more restricted.

\subsection{Asynchronous convex optimization}

We now present our main theoretical results for solving the stochastic
convex problem~\eqref{eqn:convex-problem}, giving the necessary assumptions
on $f$ and $F(\cdot; \statrv)$ for our results.  Our first assumption
roughly states that $f$ has a unique minimizer $x\opt$, that $f$ has a
quadratic expansion near the point $x\opt$, and is smooth (similar
assumptions are common~\cite[e.g.][]{PolyakJu92, BertsekasTs89} and
are satisfied in our applications).
\begin{assumption}
  \label{assumption:main-assumption}
  The function $f$ has unique minimizer $x\opt$ and is twice continuously
  differentiable in the neighborhood of $x\opt$ with positive definite
  Hessian $\hess = \nabla^2 f(x\opt) \succ 0$. There is a covariance matrix
  $\Sigma \succ 0$ such that
  \begin{equation*}
    \E[\nabla F(x\opt; \statrv) \nabla F(x\opt; \statrv)^\top]
    = \Sigma.
  \end{equation*}
  Additionally, there exists a constant $C < \infty$
  such that the gradients $\nabla F(x; \statrv)$ satisfy
  \begin{equation}
    \E[\norm{\nabla F(x; \statrv) - \nabla F(x\opt; \statrv)}^2]
    \le C \norm{x - x\opt}^2,
    ~~~ \mbox{all}~ x \in \R^d.
    \label{eqn:smooth-random-gradients}
  \end{equation}
  Lastly, $f$ has $L$-Lipschitz continuous gradient, meaning
  $\norm{\nabla f(x) - \nabla f(y)} \le L \norm{x - y}$ for $x, y \in
  \R^d$.
\end{assumption}
\noindent
Assumption~\ref{assumption:main-assumption} guarantees the uniqueness of the
vector $x^\star$ minimizing $f(x)$ over $\R^d$ and ensures that $f$ is
well-behaved enough for our asynchronous iteration procedure to introduce
negligible noise over a non-asynchronous procedure.  In addition to
Assumption~\ref{assumption:main-assumption}, we make one of two additional
assumptions.  In the first case, we assume that $f$ is strongly convex:
\newcounter{saveassumption}
\setcounter{saveassumption}{\value{assumption}}
\begin{assumption}
  \label{assumption:strong-convex}
  The function $f$ is $\lambda$-strongly convex over all of $\R^d$ for
  some $\lambda > 0$, that is,
  \begin{equation}
    \label{eqn:strongly-convex}
    f(y) \ge f(x) + \<\nabla f(x), y - x\> + \frac{\lambda}{2}
    \norm{x - y}^2
    ~~~ \mbox{for~} x, y \in \R^d.
  \end{equation}
\end{assumption}
\noindent
Our alternate assumption is a Lipschitz assumption on $f$ itself, made
by virtue of a second moment bound on $\nabla F(x; \statrv)$.
\setcounter{assumption}{\value{saveassumption}}  
\renewcommand{\theassumption}{\Alph{assumption}'}  
\begin{assumption}
  \label{assumption:lipschitz}
  There exists a constant $\lipobj < \infty$ such that for all $x \in \R^d$,
  \begin{equation}
    \label{eqn:lipschitz}
    \E[\norm{\nabla F(x; \statrv)}^2] \le \lipobj^2.
  \end{equation}
\end{assumption}
\renewcommand{\theassumption}{\Alph{assumption}}
\noindent
In Section~\ref{sec:examples} to come, we give examples in which all of
these assumptions are satisfied, showing that they are not too restrictive.

With our assumptions in place, we obtain our main theorem.
\begin{theorem}
  \label{theorem:convex-optimization}
  Let Assumptions~\ref{assumption:delay-moments} with moment $\delmoment >
  2$ and~\ref{assumption:main-assumption} hold.  Let the iterates $x_k$ be
  generated by the asynchronous process~\eqref{item:x-read-convex},
  \eqref{item:grad-comp-convex}, \eqref{item:update-step-convex} with
  stepsize choice $\stepsize_k = \stepsize k^{-\steppow}$, where $\steppow
  \in (\frac{1}{2} + \frac{1}{\delmoment - 1}, 1)$ and $\stepsize > 0$.
  Then if either of
  Assumptions~\ref{assumption:strong-convex} or~\ref{assumption:lipschitz}
  holds, we have $x_n \cas x\opt$ and
  \begin{equation*}
    \frac{1}{\sqrt{n}} \sum_{k=1}^n (x_k - x\opt)
    \cd \normal\left(0, \hess^{-1} \Sigma \hess^{-1}\right)
    = \normal\left(0, (\nabla^2 f(x\opt))^{-1} \Sigma (\nabla^2 f(x\opt))^{-1}
    \right).
  \end{equation*}
\end{theorem}

Before moving to example applications of
Theorem~\ref{theorem:convex-optimization}, we make a few additional remarks
on the theorem, its consequences, and its associated conditions. Let
$\wb{x}_n \defeq \frac{1}{n} \sum_{k=1}^n x_k$ for shorthand. First,
using the delta method~\cite[e.g.][Theorem 1.8.12]{LehmannCa98}, we can give
convergence rates for the function values $f(\wb{x}_n)$ to
$f(x\opt)$. Specifically, they converge at the optimal rate of $1/n$, and we
can give explicit constants.
\begin{corollary}
  Let the conditions of Theorem~\ref{theorem:convex-optimization} hold.
  Then
  \begin{equation*}
    n \left(f(\wb{x}_n) - f(x\opt)\right)
    \cd \half \tr\left[\hess^{-1} \Sigma \right] \cdot \chi_1^2,
  \end{equation*}
  where $\chi_1^2$ denotes a chi-squared random variable with $1$ degree of
  freedom, and $\hess = \nabla^2 f(x\opt)$ and $\Sigma = \E[\nabla F(x\opt;
    \statrv) \nabla F(x\opt; \statrv)^\top]$.
\end{corollary}
\begin{proof}
  Theorem~\ref{theorem:convex-optimization} implies $\wb{x}_n \cas x\opt$.
  By a Taylor expansion, we have
  \begin{equation*}
    n \left(f(\wb{x}_n) - f(x\opt)\right)
    = n \left[\<\nabla f(x\opt), \wb{x}_n - x\opt\>
      + \half \<\wb{x}_n - x\opt, \nabla^2 f(x\opt)
      (\wb{x}_n - x\opt)\>
      + E_3(\wb{x}_n - x\opt) \right],
  \end{equation*}
  where the error term $E_3$ satisfies
  $|E_3(x - x\opt)| \le r(x) \ltwo{x - x\opt}^2$ for a function $r$
  satisfying
  $r(x) \to 0$ as $x \to x\opt$. As
  $\nabla f(x\opt) = 0$, we see that for a remainder
  $|r_n| \le r(\wb{x}_n) \cdot n \ltwo{\wb{x}_n - x\opt}^2 \cp 0$, we have
  \begin{equation*}
    n \left(f(\wb{x}_n) - f(x\opt)\right)
    = \half \<n^{-\half} (\wb{x}_n - x\opt),
    \nabla^2 f(x\opt) n^{-\half} (\wb{x}_n - x\opt)\>
    + r_n.
  \end{equation*}
  By Theorem~\ref{theorem:convex-optimization}, the first term is
  asymptotically distributed as $Z^\top \hess Z$ for $Z \sim \normal(0,
  \hess^{-1} \Sigma \hess^{-1})$, and applying Slutsky's
  theorem~\cite[Theorem 2.7]{VanDerVaart98} and the continuous mapping
  theorem gives the result.
\end{proof}

Moreover, the convergence guarantee in
Theorem~\ref{theorem:convex-optimization} is generally unimprovable even
by numerical constants. Recall that we have
\begin{equation*}
  \sqrt{n} (\wb{x}_n - x\opt) \cd
  \normal\left(0, \hess^{-1} \Sigma \hess^{-1}\right).
\end{equation*}
Standard results in asymptotic statistics imply that
the rate of convergence and covariance $\hess^{-1} \Sigma \hess^{-1}$
are optimal.
Indeed, the Le Cam-H\'ajek local minimax theorem~\cite{LeCamYa00}
implies that in standard statistical models, if we define the balls
$B(x, t) = \{x' \in \R^d : \norm{x - x'} \le t\}$, then
\begin{equation*}
  \lim_{t \to \infty} \liminf_{n \to \infty}
  \sup_F \left\{
  n \E[\ltwo{\what{x}_n - x}^2]
  ~ : ~ x = \argmin_{x'} \E[F(x'; \statrv)] \in B(x\opt, t / \sqrt{n})
  \right\}
  \ge \tr\left[\hess^{-1} \Sigma \hess^{-1}\right],
\end{equation*}
where the supremum is taken over loss functions $F$ satisfying
Assumptions~\ref{assumption:main-assumption} and
\ref{assumption:strong-convex} or~\ref{assumption:lipschitz}, and
$\what{x}_n$ is \emph{any} sequence of estimators based on observing a
sample $\statrv_1, \ldots, \statrv_n$. The Le Cam-H\'ajek convolution
theorem~\cite{LeCamYa00} and classical calculations with Bahadur efficiency
(cf.~\citet[Chapter 8]{VanDerVaart98}) also show that the asymptotic
covariance $\hess^{-1} \Sigma \hess^{-1}$ is also generally optimal, meaning
that no estimator can converge faster than $\sqrt{n}$, and the asymptotic
covariance $S$ of essentially any estimator converging at the rate
$\sqrt{n}$ must satisfy $S \succeq \hess^{-1} \Sigma \hess^{-1}$.

More concisely, in spite of the asynchrony we allow in the
iterations~\eqref{item:x-read-convex}--\eqref{item:update-step-convex}, we
attain the best possible convergence rate. For the decision variables $x$,
the rate $n^{-\half}$ is unimprovable, as is---at least generally---the
asymptotic covariance $H^{-1} \Sigma H^{-1}$. We also have $f(\wb{x}_n) -
f(x\opt) = O_P(n^{-1})$, which is information-theoretically
optimal~\cite{NemirovskiYu83,AgarwalBaRaWa12}. So we see that quite
literally, the noise inherent to the sampling in stochastic gradient
procedures swamps any noise introduced by asynchrony.

\subsection{Examples}
\label{sec:examples}

We now give two classical statistical optimization problems to illustrate
Theorem~\ref{theorem:convex-optimization}. We verify that the conditions of
the theorem hold for each of the examples, and these show that the
conditions of Assumptions~\ref{assumption:main-assumption} and
\ref{assumption:strong-convex} or~\ref{assumption:lipschitz} are not overly
restrictive.

\paragraph{Linear regression}
Standard linear regression problems satisfies the conditions of
Assumption~\ref{assumption:strong-convex} under an additional fourth
moment condition. In this case, the data $\statval
= (a, b) \in \R^d \times \R$ and the objective $F(x; \statval) = \half (\<a,
x\> - b)^2$. If we have moment bounds
$\E[\ltwo{a}^4] < \infty$, $\E[b^2] < \infty$ and
$\hess = \E[aa^\top] \succ 0$, we have $\nabla^2 f(x\opt) =
\hess$, and
\begin{equation*}
  \E\left[\ltwo{\nabla F(x; W) - \nabla F(x\opt; W)}^2\right]
  = \E\left[\ltwo{a}^2(a^\top(x - x^\star))^2\right]
  \le \E\left[\ltwo{a}^4\right]\ltwo{x - x\opt}^2,
\end{equation*}
whence the assumptions of Theorem~\ref{theorem:convex-optimization} are
satisfied. We can give more explicit calculations if we make
standard modeling assumptions, for example, that $b = \<a, x\opt\> +
\varepsilon$, where $\varepsilon$ is an independent mean-zero noise sequence
with $\E[\varepsilon^2] = \sigma^2$. In this case, the minimizer of $f(x) =
\E[F(x; \statrv)]$ is $x\opt$, we have $\<a, x\opt\> - b = -\varepsilon$, and
\begin{equation*}
  \E[\nabla F(x\opt; \statrv) \nabla F(x\opt; \statrv)^\top]
  = \E[(\<a, x\opt\> - b)aa^\top(\<a, x\opt\> - b)]
  = \E[aa^\top \varepsilon^2] = \sigma^2 \E[aa^\top] = \sigma^2 \hess.
\end{equation*}
In particular, the asynchronous iterates satisfy
\begin{equation*}
  n^{-\half} \sum_{k=1}^n (x_k - x\opt)
  \cd \normal(0, \sigma^2 \hess^{-1})
  = \normal\left(0, \sigma^2 \E[aa^\top]^{-1}\right).
\end{equation*}
This has the asymptotic variance of the ordinary least squares estimate of
$x\opt$, which is minimax optimal~\cite[Chapter 5]{LehmannCa98}.

\paragraph{Logistic regression}
As long as the data has finite second moment, logistic regression problems
satisfy all the conditions of Assumption~\ref{assumption:lipschitz} in
Theorem~\ref{theorem:convex-optimization}. In this case we
have $\statval = (a, b) \in \R^d \times \{-1, 1\}$ and instantaneous objective
$F(x; \statval) = \log(1 + \exp(-b \<a, x\>))$.
For fixed $\statval$, this function is Lipschitz continuous and has
gradient and Hessian
\begin{equation*}
  \nabla F(x; \statval) = -\frac{1}{1 + \exp(b \<a, x\>)} ba
  ~~~ \mbox{and} ~~~
  \nabla^2 F(x; \statval)
  = \frac{e^{b \<a, x\>}}{(1 + e^{b \<a, x\>})^2}
  aa^\top,
\end{equation*}
where $\nabla F(x; \statval)$ is Lipschitz continuous as $\norm{\nabla^2
  F(x; (a, b))} \le \frac{1}{4} \ltwo{a}^2$. Thus, so long as
$\E[\ltwo{a}^2] < \infty$ and $\E[\nabla^2 F(x\opt; \statrv)] \succ 0$ (the
latter occurs if $\E[aa^\top]$ is full rank), logistic regression satisfies
the conditions of Theorem~\ref{theorem:convex-optimization}. In particular,
the asynchronous stochastic gradient method achieves optimal convergence
guarantees.

\subsection{Extension to nonlinear problems and variational inequalities}
\label{sec:nonlinear-problems}

We prove Theorem~\ref{theorem:convex-optimization} by way of a more general
result on finding the zeros of a residual operator $\resid : \R^d \to \R^d$,
where we only observe noisy views of $\resid(x)$, and there is unique
$x\opt$ such that $\resid(x\opt) = 0$. Such situations arise, for example,
in the solution of stochastic monotone operator problems
(cf.~\citet*{JuditskyNeTa11} or~\citet{BertsekasTs89}), including finding
equilibria in stochastic convex Nash games
(e.g.~\cite[Sec.~2.1]{JuditskyNeTa11}, \cite[Ex.~3.5.1(d)]{BertsekasTs89}),
general saddle-point problems, or multi-user routing
problems~\cite[Ex.~3.5.1(c)]{BertsekasTs89}. In this more general setting,
we consider the following stochastic and asynchronous iterative process,
which extends that for the convex case outlined previously. Each processor
performs the following asynchronously and independently:
\begin{enumerate}[(i)]
\item Processor reads current problem data $x$ and counter $k$
\item \label{item:grad-comp-nl} Processor receives vector $g = \resid(x) +
  \noise$, where $\noise$ is a random (conditionally) mean-zero noise
  vector, and increments a centralized counter $k$
\item \label{item:update-step-nl} Processor updates $x \leftarrow x -
  \stepsize_k g$ via sequential updates $[x]_j = [x]_j - \stepsize_k [g]_j$ for
  each coordinate $j \in \{1, \ldots, d\}$
\end{enumerate}
As in the convex case, we associate vectors $x_k$ and $g_k$ with the update
performed using $\stepsize_k$, and we let $\noise_k$ denote the noise vector
used to construct $g_k$. As before, these iterates and assignment of indices
again imply that $x_k$ has the form~\eqref{eqn:iterate-formula},
that is, $x_k = -\sum_{i=1}^{k-1} \stepsize_i \errmat^{ki} g_i$ for
diagonal matrices $\errmat^{ki}$ capturing the updates that have been
performed at time $k$.

For this iterative process, we define the increasing sequence of
$\sigma$-fields $\mc{F}_k$ by
\begin{equation}
  \label{eqn:filtration}
  \mc{F}_k = \sigma\left(
  \noise_1, \ldots, \noise_k,
  \left\{\errmat^{ij} : i \le k + 1, j \le i\right\}
  \right),
\end{equation}
that is, the noise variables $\noise_k$ are adapted to the filtration
$\mc{F}_k$, and these $\sigma$-fields are the smallest containing both the
noise and all index updates that have occurred and that will occur to
compute $x_{k + 1}$. Thus we have $x_{k + 1} \in \mc{F}_k$, and our
mean-zero assumption on the noise $\noise$ is that
\begin{equation*}
  \E[\noise_k \mid \mc{F}_{k-1}] = 0.
\end{equation*}

Our analysis builds off of \citeauthor{PolyakJu92}'s study~\cite{PolyakJu92}
of averaging in stochastic approximation, and we model our requirements for
the convergence of the preceding iteration on those they use for the
solution of the nonlinear equality $\resid(x\opt)$.  First, we assume there
is a Lyapunov function $\lyap$ that functions (essentially) as a squared
norm, which satisfies $\lyap(x) \ge \lambda \norm{x}^2$ for all $x \in
\R^d$, $\norm{\nabla \lyap(x) - \nabla \lyap(y)} \le L \norm{x - y}$ for all
$x, y$, that $\nabla \lyap(0) = 0$, and $\lyap(0) = 0$. Note in particular
that this implies
\begin{equation}
  \label{eqn:lyapunov-is-norm-squared}
  \lambda \norm{x}^2 \le
  \lyap(x) \le \lyap(0) + \<\nabla \lyap(0), x - 0\>
  + \frac{L}{2} \norm{x}^2
  = \frac{L}{2} \norm{x}^2
\end{equation}
and that $\norm{\nabla \lyap(x)}^2 \le L^2 \norm{x}^2 \le (L^2 / \lambda)
\lyap(x)$.  In addition, we make the following assumptions on the residual
function (cf.~\cite[Assumption 3.2]{PolyakJu92}).
\begin{assumption}
  \label{assumption:residual-quadratic}
  There exists a matrix $\hess \in \R^{d \times d}$ with $\hess \succ 0$,
  a parameter $0 < \gamma \le 1$, constant $C < \infty$,
  and some $\epsilon > 0$
  such that if $x$ satisfies $\norm{x - x\opt} \le \epsilon$, then
  \begin{equation*}
    \norm{\resid(x) - \hess(x - x\opt)} \le C \norm{x - x\opt}^{1 + \gamma}.
  \end{equation*}
\end{assumption}
\noindent
Assumption~\ref{assumption:residual-quadratic} essentially requires that
$\resid$ is differentiable at $x\opt$ with derivative matrix $\hess \succ
0$.  We also make a few assumptions on the noise process $\noise$
paralleling Assumption 3.3 of \citet{PolyakJu92}; specifically, we assume
$\noise$ implicitly depends on $x \in \R^d$ (so that we may write $\noise_k
= \noise(x_k)$), and that the following assumption holds.
\begin{assumption}
  \label{assumption:noise-fun}
  The noise vector $\noise(x)$ decomposes as $\noise(x) = \noise(0) +
  \noisier(x)$, where $\noise(0)$ is a process satisfying $\E[\noise_k(0)
    \noise_k(0)^\top \mid \mc{F}_{k-1}] \cp \Sigma \succ 0$ for some matrix
  $\Sigma \in \R^{d \times d}$, the boundedness condition $\E[\sup_k
    \E[\norm{\noise_k(0)}^2 \mid \mc{F}_{k-1}]] < \infty$, and
  \begin{equation*}
    \E[\norm{\noisier_k(x)}^2 \mid \mc{F}_{k-1}]
    \le C \norm{x - x\opt}^2
  \end{equation*}
  for some constant $C < \infty$ and all $x \in \R^d$.
\end{assumption}

As in the convex case, we make one of two additional assumptions,
which should be compared with Assumptions~\ref{assumption:strong-convex}
and~\ref{assumption:lipschitz}. The first is that
$\resid$ gives globally strong information about $x\opt$.
\setcounter{saveassumption}{\value{assumption}}
\begin{assumption}[Strongly convex residuals]
  \label{assumption:resid-strong-convex}
  There exists a constant $\lambda_0 > 0$ such that for all $x \in \R^d$, 
  $\<\nabla \lyap(x - x\opt), \resid(x)\> \ge \lambda_0 \lyap(x - x\opt)$.
\end{assumption}
\noindent
Alternatively, we may make an assumption on the boundedness of
$\resid$, which we shall see suffices for proving our main results.
\setcounter{assumption}{\value{saveassumption}}  
\renewcommand{\theassumption}{\Alph{assumption}'}  
\begin{assumption}[Bounded residuals]
  \label{assumption:resid-lipschitz}
  There exist $\lambda_0 > 0$ and $\epsilon > 0$ such that
  \begin{equation*}
    \inf_{0 < \norm{x - x\opt} \le \epsilon}
    \frac{\<\nabla \lyap(x - x\opt), \resid(x)\>}{
      \lyap(x - x\opt)} \ge \lambda_0
    ~~~ \mbox{and} ~~~
    \inf_{\epsilon < \norm{x - x\opt}}
    \<\nabla \lyap(x - x\opt), \resid(x)\> > 0.
  \end{equation*}
  There also exists some $C <\infty$ such that $\norm{\resid(x)} \le C$ and
  $\E[\norms{\noise_k}^2 \mid \mc{F}_{k-1}] \le C^2$ for all $k$ and $x$.
\end{assumption}
\renewcommand{\theassumption}{\Alph{assumption}}

With these assumptions in place, we obtain the following more general
version of Theorem~\ref{theorem:convex-optimization};
indeed, we show
in the sequel how Theorem~\ref{theorem:convex-optimization} follows
from this result.
\begin{theorem}
  \label{theorem:nonlinear}
  Let $\lyap$ be a function satisfying
  inequality~\eqref{eqn:lyapunov-is-norm-squared}, and let
  Assumptions~\ref{assumption:residual-quadratic},
  \ref{assumption:noise-fun}, and~\ref{assumption:delay-moments} hold. Let
  the stepsizes $\stepsize_k = \stepsize k^{-\steppow}$, where
  $\frac{1}{\delmoment - 1} + \frac{1}{1 + \gamma} < \steppow < 1$.  Let one
  of Assumptions~\ref{assumption:resid-strong-convex}
  or~\ref{assumption:resid-lipschitz} hold. Then
  $x_n \cas x\opt$ and
  \begin{equation*}
    \frac{1}{\sqrt{n}}
    \sum_{k=1}^n (x_k - x\opt)
    \cd \normal\left(0, \hess^{-1} \Sigma \hess^{-1} \right).
  \end{equation*}
\end{theorem}

We may compare this result to \citeauthor{PolyakJu92}'s Theorem 2, which
gives identical covariance matrix and asymptotic convergence guarantees, but
with weaker conditions on the function $\lyap$ and stepsize sequence
$\stepsize_k$. Our mildly stronger assumptions---namely,
Assumptions~\ref{assumption:resid-strong-convex}
and~\ref{assumption:resid-lipschitz} are stronger versions of Assumption 3.1
of \citet{PolyakJu92}, which requires only the conditions on $\lyap$ and
$\resid$ of Assumption~\ref{assumption:resid-lipschitz}---allow our result
to apply even in the asynchronous settings considered in this paper.

\section{Experimental results}
\label{sec:experiments}


\newcommand{\density}{p_{\rm{nz}}}
\newcommand{\sparseproj}{\Pi_{\density}}
\newcommand{\batchsize}{\mathsf{B}}
\newcommand{\littleb}{b}

We provide empirical results studying the performance of asynchronous
stochastic approximation schemes on several simulated and real-world datasets.
Our theoretical results suggest that asynchrony should introduce little
degradation in solution quality; we also
investigate the engineering techniques necessary to truly leverage the power
of asynchronous stochastic procedures.
In our experiments, we focus on linear and logistic regression, the examples
given in Section~\ref{sec:examples}; that is, we have data $(a_i, b_i) \in
\R^d \times \R$ (for linear regression) or $(a_i, b_i) \in \R^d \times \{-1,
1\}$ (for logistic regression), for $i = 1, \ldots, N$, and objectives
\begin{equation}
  \label{eqn:empirical-expectation}
  f(x) = \frac{1}{2N} \sum_{i=1}^N (\<a_i, x\> - b_i)^2
  ~~~ \mbox{and} ~~~
  f(x) = \frac{1}{N} \sum_{i=1}^N \log\big(1 + \exp(-b_i \<a_i, x\>)
  \big).
\end{equation}

We perform each of our experiments using a 48-core Intel Xeon machine with 1
terabyte of RAM, and have put code and binaries to replicate our experiments
on \texttt{CodaLab}~\cite{DuchiChRe15code}. The Xeon architecture puts each
core onto one of four sockets, where each socket has its own memory. To
limit the impact of communication overhead in our experiments, we limit all
experiments to at most 12 cores, all on the same socket. Within an
experiment---based on the empirical
expectations~\eqref{eqn:empirical-expectation}---we iterate in
\emph{epochs}, so that our stochastic gradient procedure loops through
random permutations of all examples, touching each example exactly once per
epoch in a different random order within each epoch
(cf.~\cite{RechtRe12}).\footnote{Strictly speaking, this violates the
  stochastic gradient assumption, but it allows direct comparison with the
  original \hogwild code and implementation~\cite{NiuReReWrNi11}.}  We use
the following two schemes for the stepsize $\stepsize_k$.
\begin{description}
\item[Decreasing stepsizes.] We set $\steppow = 0.55$ and let $\stepsize_k =
  k^{-\steppow}$. The value written on the shared iteration counter $k$ by
  one processor may be overwritten by other processors.
\item[Exponential backoff stepsizes.] We use a fixed stepsize $\stepsize$,
  decreasing the stepsize by a factor of $0.95$ between each epoch (this
  matches the experimental protocol of \citet{NiuReReWrNi11} and follows
  \citet{HazanKa11} and \citet{GhadimiLa12}).
\end{description}
To address issues of hardware resource contention
(see Section~\ref{sec:hardware-fun} for more on this), in
some cases we use a \emph{mini-batching} strategy. Abstractly, in the
formulation of the basic problem~\eqref{eqn:convex-problem}, this means that
in each calculation of a stochastic gradient $g$ we draw $\batchsize \ge 1$
samples $\statrv_1, \ldots, \statrv_\batchsize$ i.i.d.\ according to $P$,
then set
\begin{equation}
  \label{eqn:mini-batching}
  g(x) = \frac{1}{\batchsize} \sum_{\littleb=1}^\batchsize \nabla F(x;
  \statrv_\littleb).
\end{equation}
The mini-batching strategy~\eqref{eqn:mini-batching} does not change the
(asymptotic) convergence guarantees of asynchronous stochastic gradient
descent, as the covariance matrix $\Sigma = \E[g(x\opt) g(x\opt)^\top]$
satisfies $\Sigma = \frac{1}{\batchsize} \E[\nabla F(x\opt; \statrv) \nabla
  F(x\opt; \statrv)^\top]$, while the total iteration count is reduced by the
a factor $\batchsize$. Lastly, we measure the performance of
optimization schemes via \emph{speedup}, defined as
\begin{equation}
  \label{eqn:speedup}
  \text{speedup}
  =
  \frac{\text{average epoch runtime on a single core using stochastic
      gradient descent}}{
    \text{average epoch runtime of asynchronous method on $m$ cores}}.
\end{equation}
In our experiments, we see that increasing the number $m$ of cores does not
change the gap in optimality $f(x_k) - f(x\opt)$ after each epoch, so speedup
is equivalent to the ratio of the time required to obtain an
$\epsilon$-accurate solution using a single processor/core to that required
to obtain $\epsilon$-accurate solution using $m$ processors/cores.

\subsection{Efficiency and sparsity}

For our first set of experiments, we study the effect that data sparsity has
on the convergence behavior of asynchronous methods using the linear
regression objective~\eqref{eqn:empirical-expectation}. Sparsity has been an
essential part of the analysis of many asynchronous and parallel
optimization schemes~\cite{NiuReReWrNi11,DuchiJoMc13_nips,RichtarikTa15},
while our theoretical results suggest it should be unimportant, so
understanding these effects is important.  We generate synthetic linear
regression problems with $N = 10^6$ examples in $d = 10^3$ dimensions via
the following procedure.  Let $\density \in \openleft{0}{1}$ be the desired
fraction of non-zero gradient entries, and let $\sparseproj$ be a random
projection operator that zeros out all but a fraction $\density$ of the
elements of its argument, meaning that for $a \in \R^d$, $\sparseproj(a)$
uniformly at random chooses $\density d$ elements of $a$ and leaves them
identical, zeroing the remaining elements.  We generate data for our linear
regression by drawing a random vector $u\opt \sim \normal(0, I_{d \times
  d})$, then constructing
\begin{equation}
  \label{eqn:simulated-regression}
  b_i = \<a_i, u\opt\> + \varepsilon_i,
  ~~ \mbox{where} ~~
  \varepsilon_i \simiid \normal(0, 1),
  ~~ \wt{a}_i \simiid \normal(0, I_{d \times d}),
  ~~ \mbox{and} ~~ a_i = \sparseproj(\wt{a}_i)
\end{equation}
for $i = 1, \ldots, N$, where $\sparseproj(\wt{a}_i)$ denotes an independent
random sparse projection of $\wt{a}_i$. To measure optimality gap, we
directly compute $x\opt = (A^TA)^{-1} A^Tb$, where $A = [a_1 ~ a_2 ~ \cdots
  ~ a_N]^\top \in \R^{N \times d}$.

\begin{figure}[t]
  \begin{center}
    \begin{tabular}{cccc}
      \hspace{-.3cm}
      \psfrag{optimalitygap}[bl][bl][.6]{$f(x_k) - f(x\opt)$}
      \psfrag{statisticaldensity}{}
      \psfrag{\#epochs}[bl][bl][.6]{Epochs}
      \includegraphics[width=.24\columnwidth,trim={20 0 0 0},clip=true]{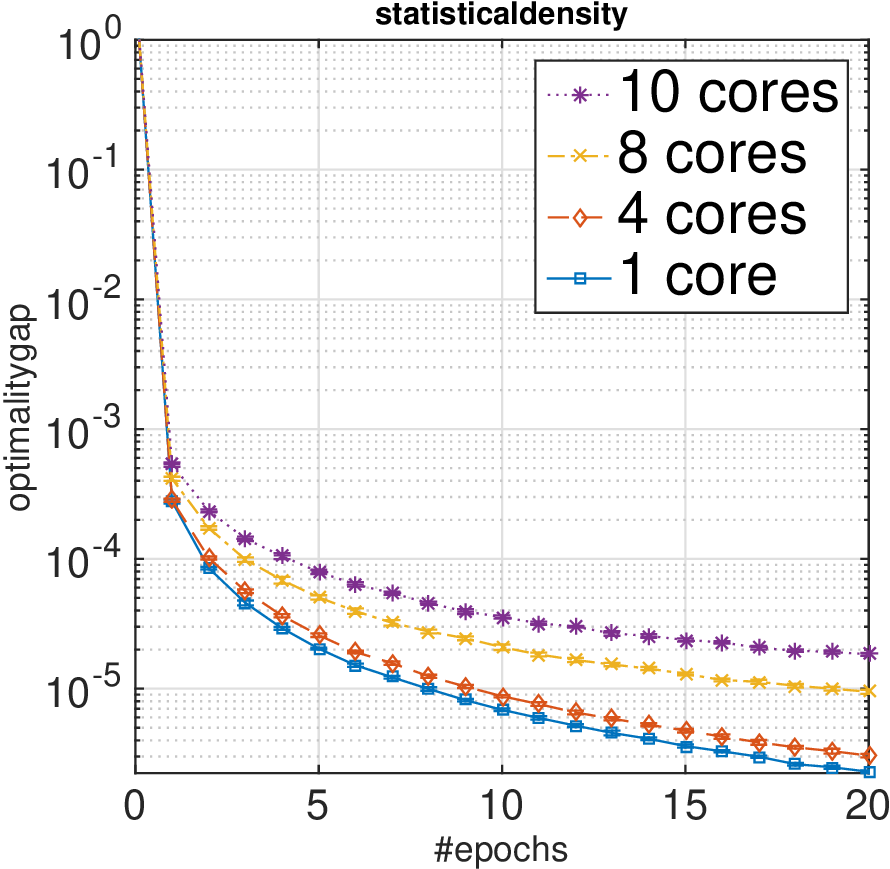} &
      \hspace{-.3cm}
      \psfrag{statisticaldensity}{}
      \psfrag{optimalitygap}[bl][bl][.6]{}
      \psfrag{\#epochs}[bl][bl][.6]{Epochs}
      \includegraphics[width=.24\columnwidth,trim={20 0 0 0},clip=true]{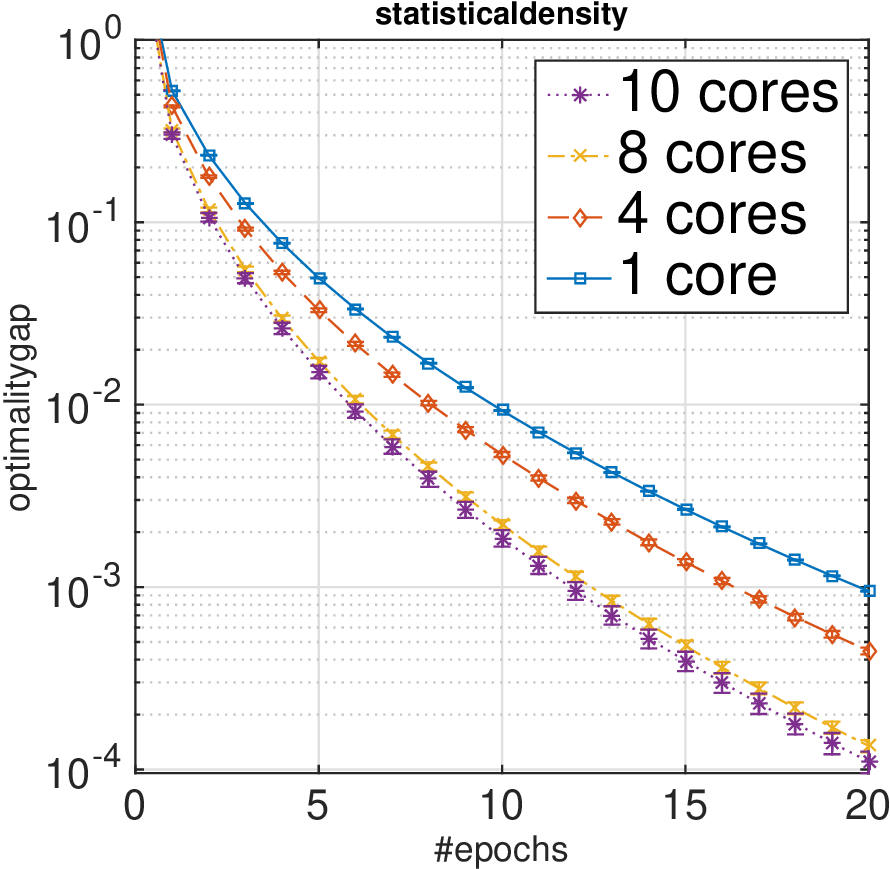} &
      \hspace{-.3cm}
      \psfrag{statisticaldensity}{}
      \psfrag{\#epochs}[bl][bl][.6]{Epochs}
      \psfrag{optimalitygap}[bl][bl][.6]{}
      \includegraphics[width=.24\columnwidth,trim={20 0 0 0},clip=true]{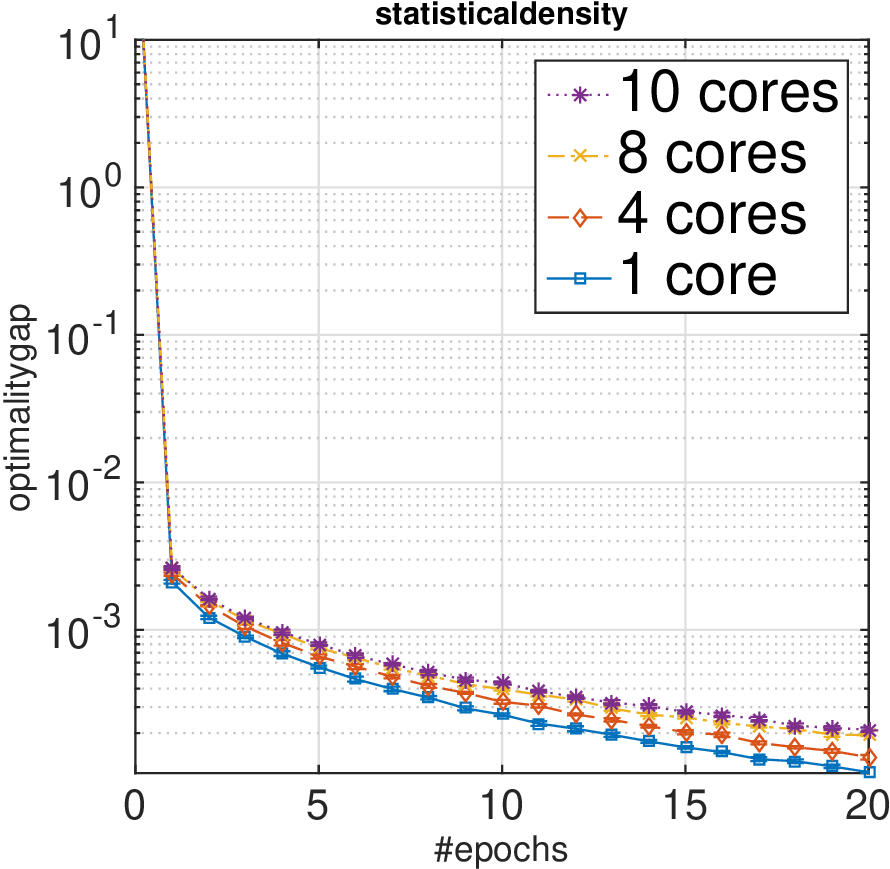} &
      \hspace{-.3cm}
      \psfrag{statisticaldensity}{}
      \psfrag{\#epochs}[bl][bl][.6]{Epochs}
      \psfrag{optimalitygap}[bl][bl][.6]{}
      \includegraphics[width=.24\columnwidth,trim={20 0 0 0},clip=true]{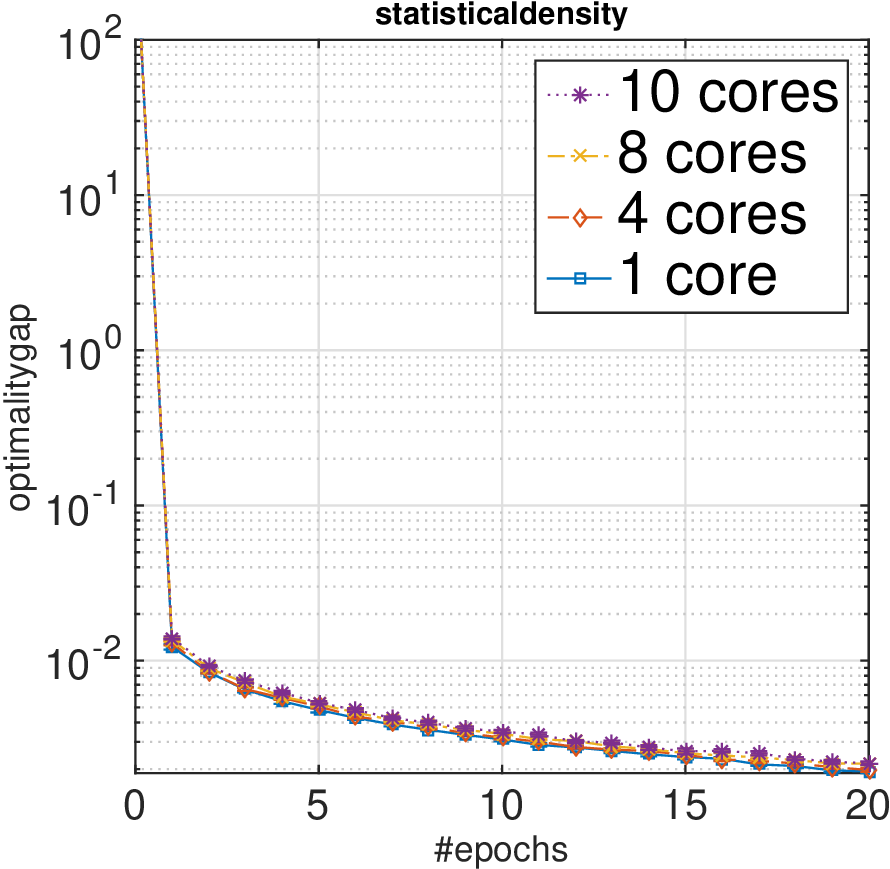} \\
      (a) $\density = .005$ & (b) $\density = .01$
      & (c) $\density = .2$ & (d) $\density = 1$
    \end{tabular}
    \caption{\label{fig:synthetic-sparsity-dec} Decreasing stepsizes:
      Optimality gaps for synthetic linear regression experiments showing
      effects of data sparsity and asynchrony on $f(x_k) - f(x\opt)$
      with stepsize $\stepsize_k = k^{-\steppow}$. A
      fraction $\density$ of each vector $a_i \in \R^d$ is non-zero.}
  \end{center}
\end{figure}
\begin{figure}[t]
  \begin{center}
    \begin{tabular}{cccc}
      \hspace{-.3cm}
      \psfrag{optimalitygap}[bl][bl][.6]{$f(x_k) - f(x\opt)$}
      \psfrag{statisticaldensity}{}
      \psfrag{\#epochs}[bl][bl][.6]{Epochs}
      \includegraphics[width=.24\columnwidth,trim={20 0 0 0},clip=true]{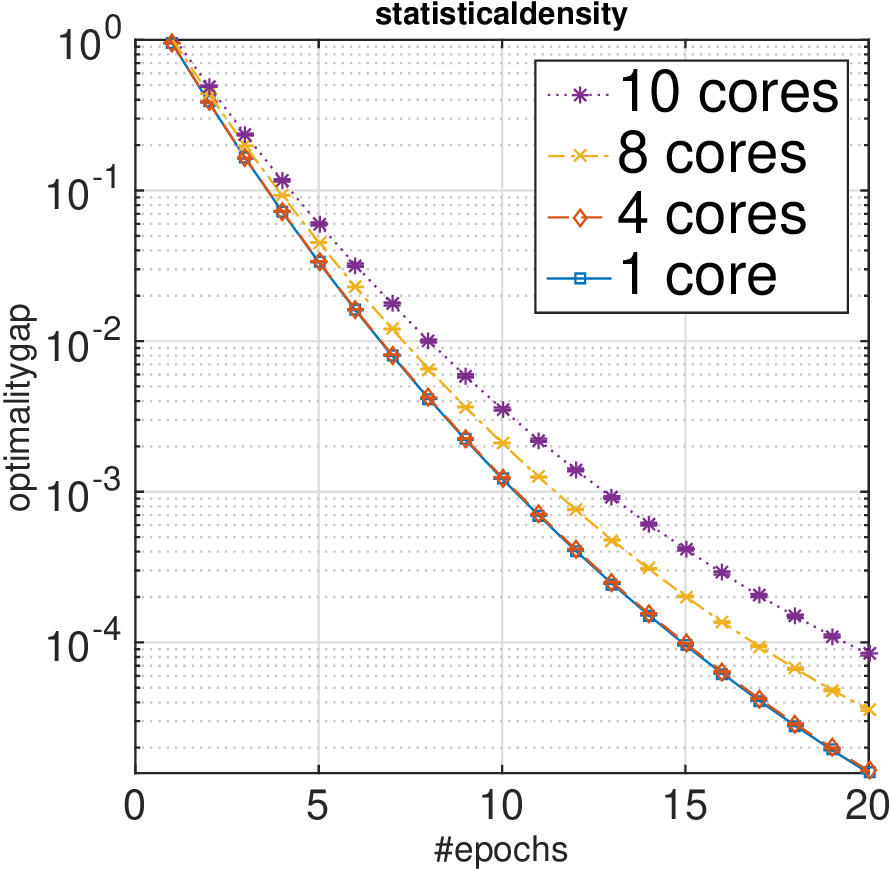} &
      \hspace{-.3cm}
      \psfrag{statisticaldensity}{}
      \psfrag{optimalitygap}[bl][bl][.6]{}
      \psfrag{\#epochs}[bl][bl][.6]{Epochs}
      \includegraphics[width=.24\columnwidth,trim={20 0 0 0},clip=true]{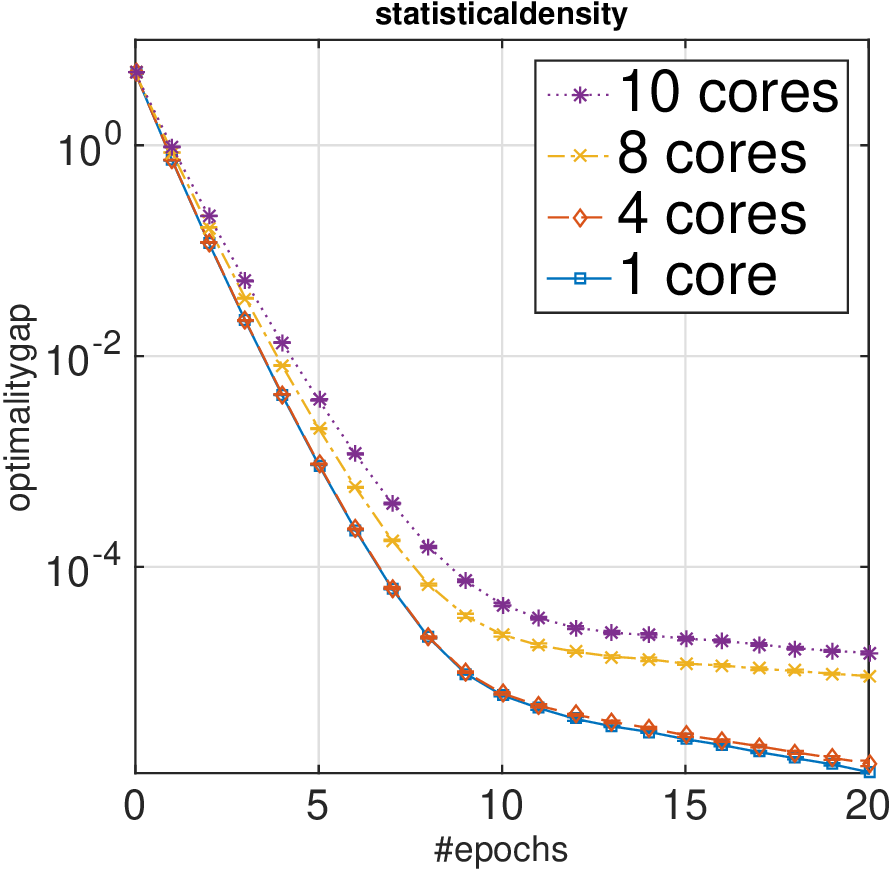} &
      \hspace{-.3cm}
      \psfrag{statisticaldensity}{}
      \psfrag{\#epochs}[bl][bl][.6]{Epochs}
      \psfrag{optimalitygap}[bl][bl][.6]{}
      \includegraphics[width=.24\columnwidth,trim={20 0 0 0},clip=true]{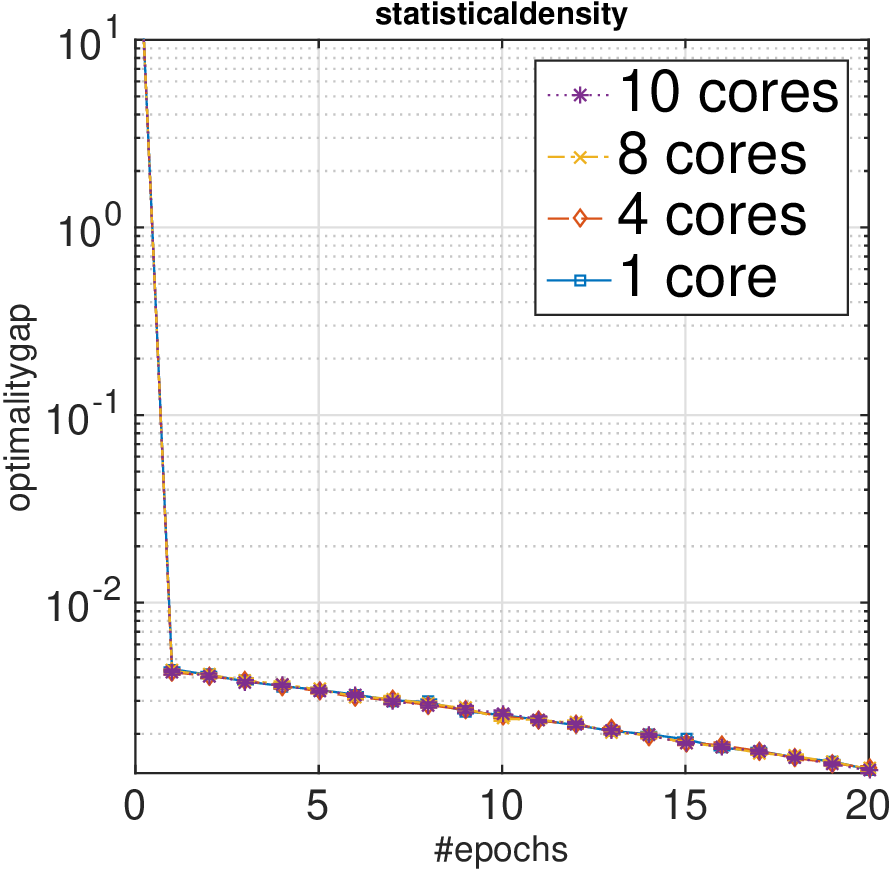} &
      \hspace{-.3cm}
      \psfrag{statisticaldensity}{}
      \psfrag{\#epochs}[bl][bl][.6]{Epochs}
      \psfrag{optimalitygap}[bl][bl][.6]{}
      \includegraphics[width=.24\columnwidth,trim={20 0 0 0},clip=true]{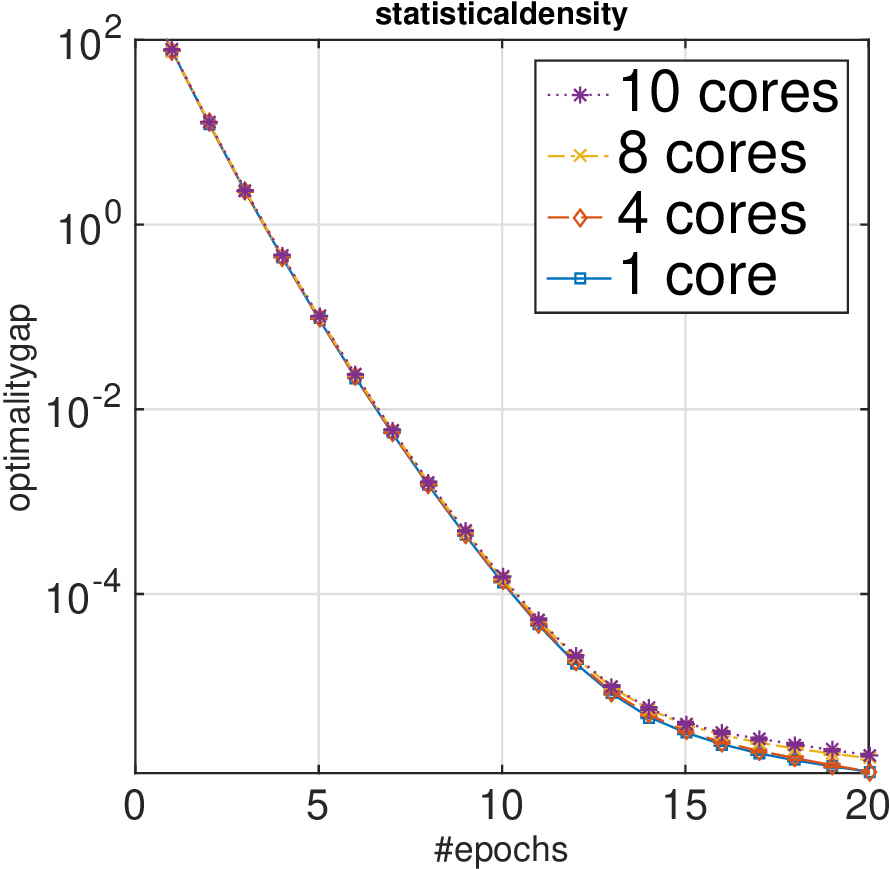} \\
      (a) $\density = .005$ & (b) $\density = .01$
      & (c) $\density = .2$ & (d) $\density = 1$
    \end{tabular}
    \caption{\label{fig:synthetic-sparsity} Exponentially decreasing stepsizes:
      Optimality gaps for synthetic linear regression experiments showing
      effects of data sparsity and asynchrony on $f(x_k) - f(x\opt)$ with
      epoch-based stepsizes $\stepsize_{{\rm epoch}~k} = .95^k$. A fraction
      $\density$ of each vector $a_i \in \R^d$ is non-zero.}
  \end{center}
\end{figure}

In Figures~\ref{fig:synthetic-sparsity-dec} and~\ref{fig:synthetic-sparsity}
we plot the results of simulations using densities $\density \in \{.005,
.01, .2, 1\}$ and mini-batch size $\batchsize = 10$, showing the gap $f(x_k)
- f(x\opt)$ as a function of the number of epochs for each of the given
sparsity levels. Figure~\ref{fig:synthetic-sparsity-dec} gives results for
our simulated data~\eqref{eqn:simulated-regression} using the decreasing
stepsize scheme $\stepsize_k = k^{-\steppow}$ with $\steppow = .55$, while
Figure~\ref{fig:synthetic-sparsity} gives results using the exponential
backoff scheme of~\cite{HazanKa11,GhadimiLa12,NiuReReWrNi11}, where
stepsizes are chosen per epoch as $\stepsize_{{\rm epoch}~k} = .95^k$.  Each
plot includes error bars with standard errors over 10 random experiments
using different random seeds (the errors are generally too small to see in
the plots). We give results using 1, 4, 8, and 10 processor cores
(increasing degrees of asynchrony). From the plots, we see that regardless
of the number of cores, the convergence behavior is nearly identical, with
minor degradations in performance for the sparsest data. (We plot the gaps
$f(x_k) - f(x\opt)$ on a logarithmic axis.)  Moreover, as the data becomes
denser, the more asynchronous methods---larger number of cores---achieve
performance essentially identical to the fully synchronous method in terms
of convergence versus number of epochs.

\begin{figure}[ht]
  \begin{center}
    \begin{tabular}{cccc}
      \hspace{-.1cm}
      \psfrag{optimalitygap}[bl][bl][.6]{$f(x_k) - f(x\opt)$}
      \psfrag{hardwaredensity}{}
      \psfrag{\#cores}[bl][bl][.6]{Cores}
      \includegraphics[width=.24\columnwidth,trim={20 0 0 0},clip=true]{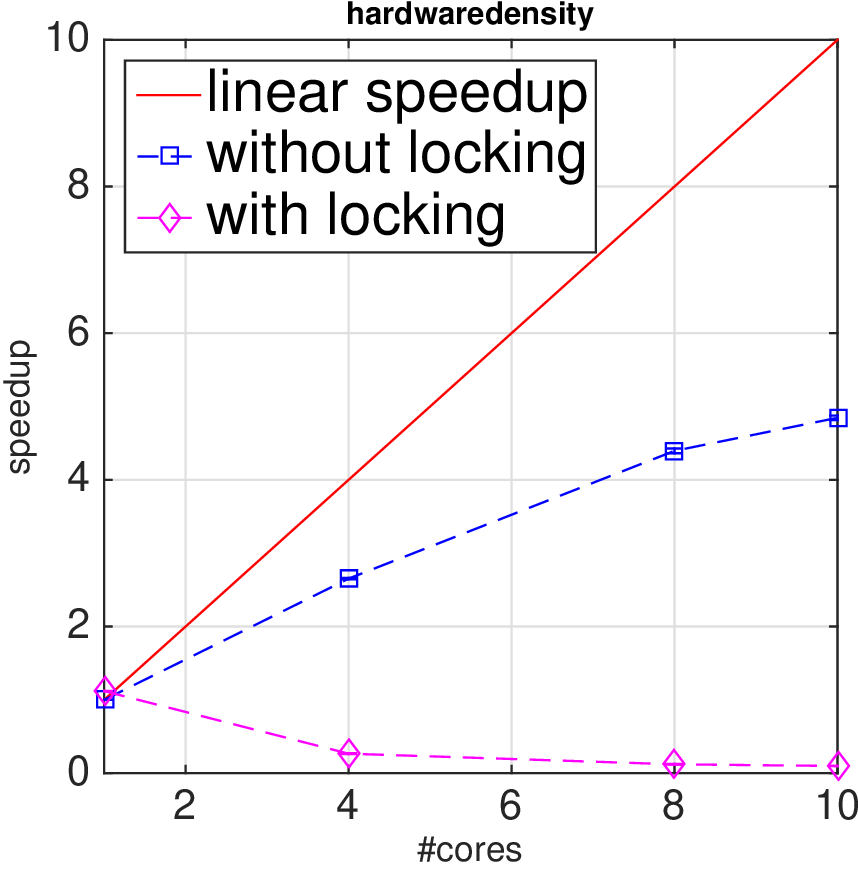} &
      \hspace{-.3cm}
      \psfrag{hardwaredensity}{}
      \psfrag{optimalitygap}[bl][bl][.6]{}
      \psfrag{\#cores}[bl][bl][.6]{Cores}
      \includegraphics[width=.24\columnwidth,trim={20 0 0 0},clip=true]{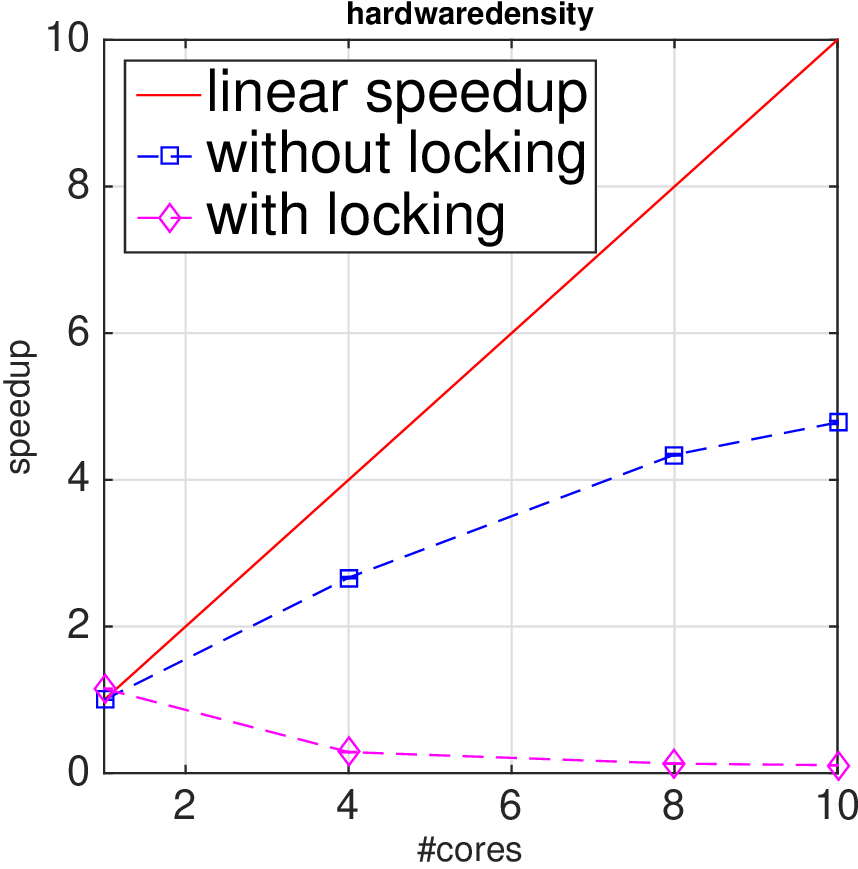} &
      \hspace{-.3cm}
      \psfrag{hardwaredensity}{}
      \psfrag{\#cores}[bl][bl][.6]{Cores}
      \psfrag{optimalitygap}[bl][bl][.6]{}
      \includegraphics[width=.24\columnwidth,trim={20 0 0 0},clip=true]{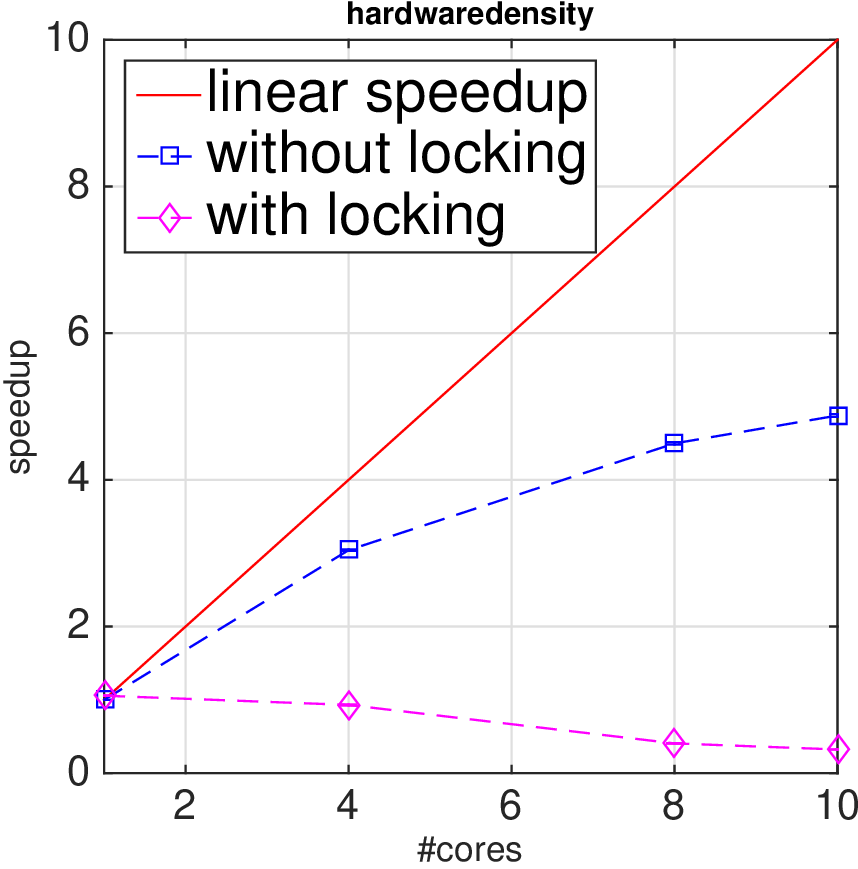} &
      \hspace{-.3cm}
      \psfrag{hardwaredensity}{}
      \psfrag{\#cores}[bl][bl][.6]{Cores}
      \psfrag{optimalitygap}[bl][bl][.6]{}
      \includegraphics[width=.24\columnwidth,trim={20 0 0 0},clip=true]{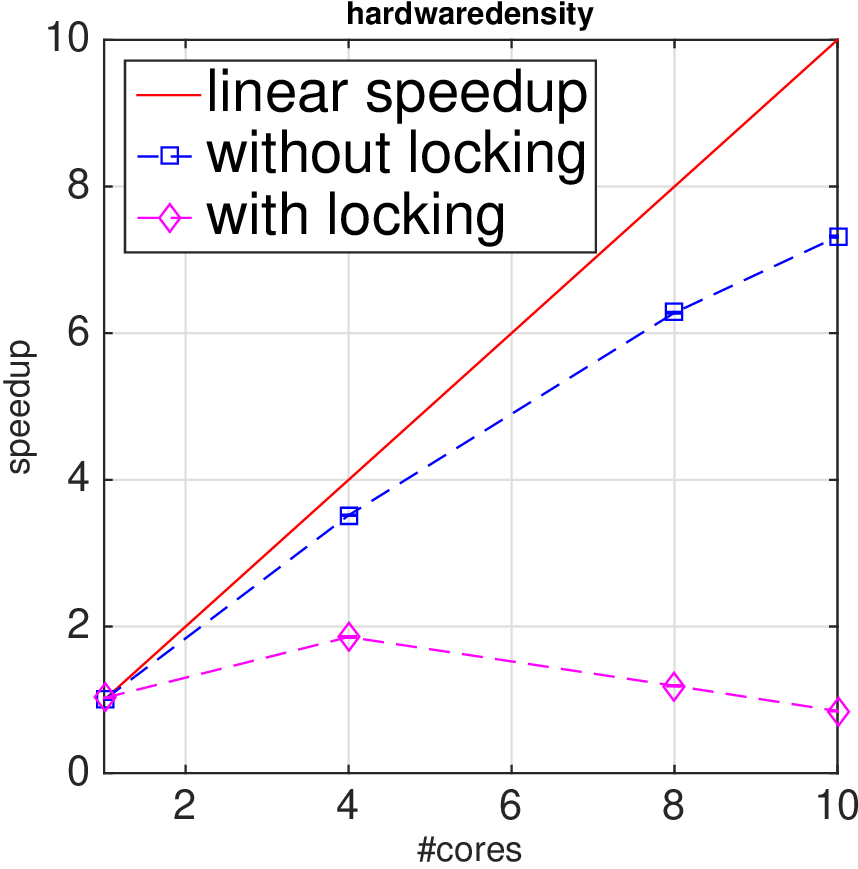} \\
      (a) $\density = .005$ & (b) $\density = .01$
      & (c) $\density = .2$ & (d) $\density = 1$
    \end{tabular}
    \caption{\label{fig:synthetic-sparsity-hardware-dec} Decreasing
      stepsizes: Speedups for synthetic linear regression experiments
      showing effects of data sparsity on speedup~\eqref{eqn:speedup}
      with stepsize $\stepsize_k = k^{-\steppow}$. A
      fraction $\density$ of each vector $a_i \in \R^d$ is non-zero.}
  \end{center}
\end{figure}
\begin{figure}[ht]
  \begin{center}
    \begin{tabular}{cccc}
      \hspace{-.1cm}
      \psfrag{optimalitygap}[bl][bl][.6]{$f(x_k) - f(x\opt)$}
      \psfrag{hardwaredensity}{}
      \psfrag{\#cores}[bl][bl][.6]{Cores}
      \includegraphics[width=.24\columnwidth,trim={20 0 0 0},clip=true]{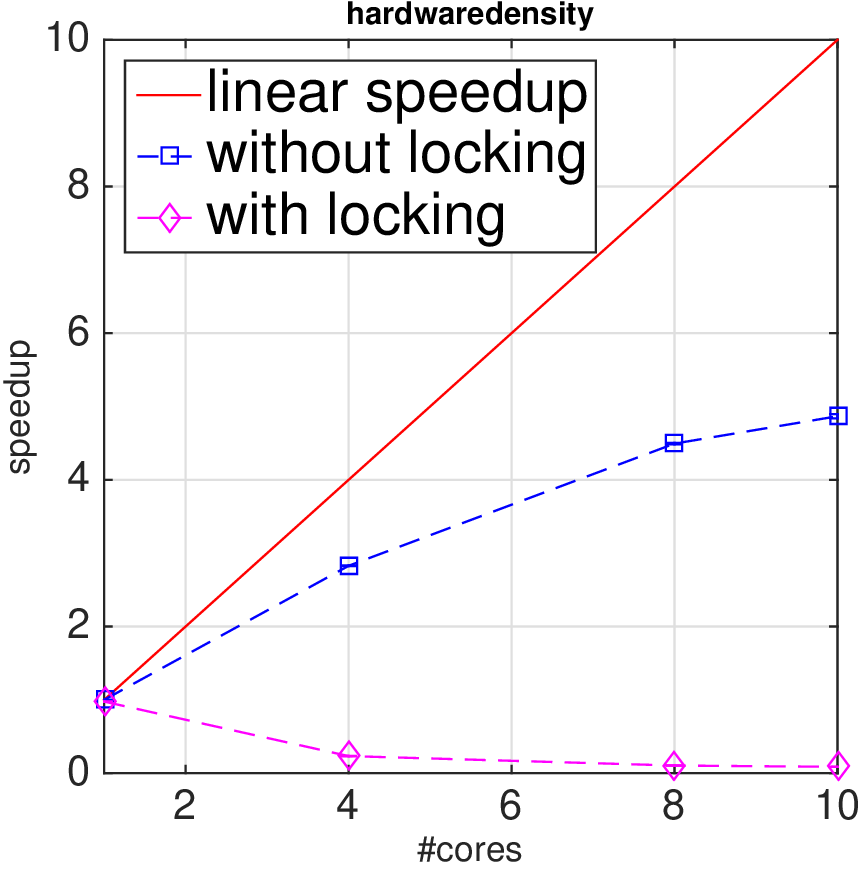} &
      \hspace{-.3cm}
      \psfrag{hardwaredensity}{}
      \psfrag{optimalitygap}[bl][bl][.6]{}
      \psfrag{\#cores}[bl][bl][.6]{Cores}
      \includegraphics[width=.24\columnwidth,trim={20 0 0 0},clip=true]{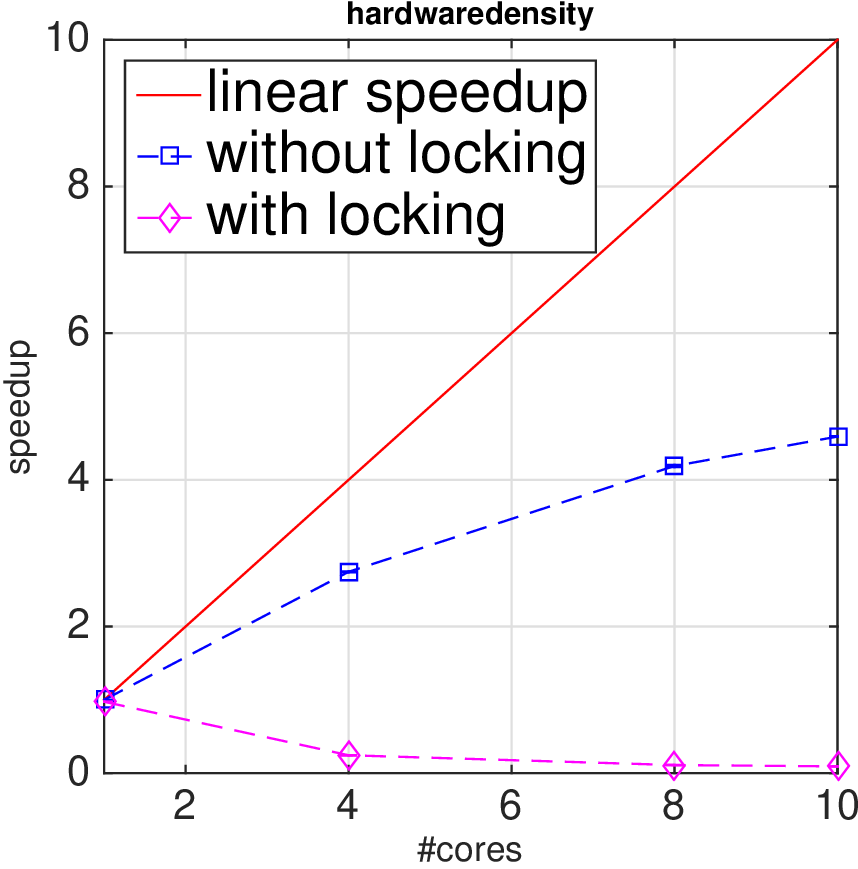} &
      \hspace{-.3cm}
      \psfrag{hardwaredensity}{}
      \psfrag{\#cores}[bl][bl][.6]{Cores}
      \psfrag{optimalitygap}[bl][bl][.6]{}
      \includegraphics[width=.24\columnwidth,trim={20 0 0 0},clip=true]{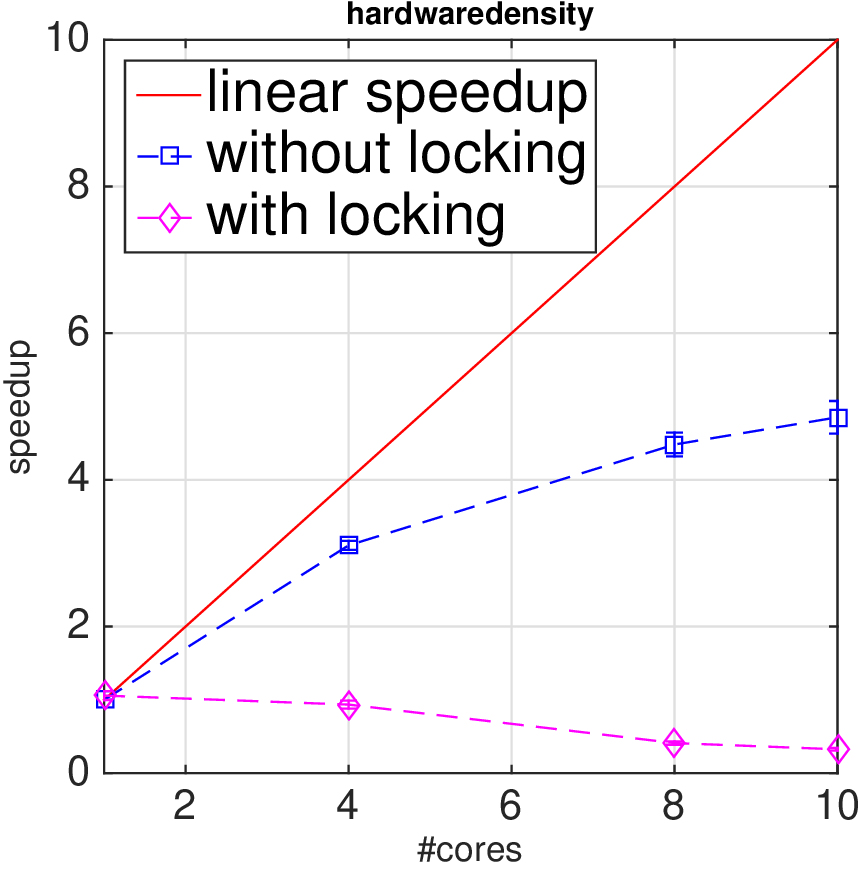} &
      \hspace{-.3cm}
      \psfrag{hardwaredensity}{}
      \psfrag{\#cores}[bl][bl][.6]{Cores}
      \psfrag{optimalitygap}[bl][bl][.6]{}
      \includegraphics[width=.24\columnwidth,trim={20 0 0 0},clip=true]{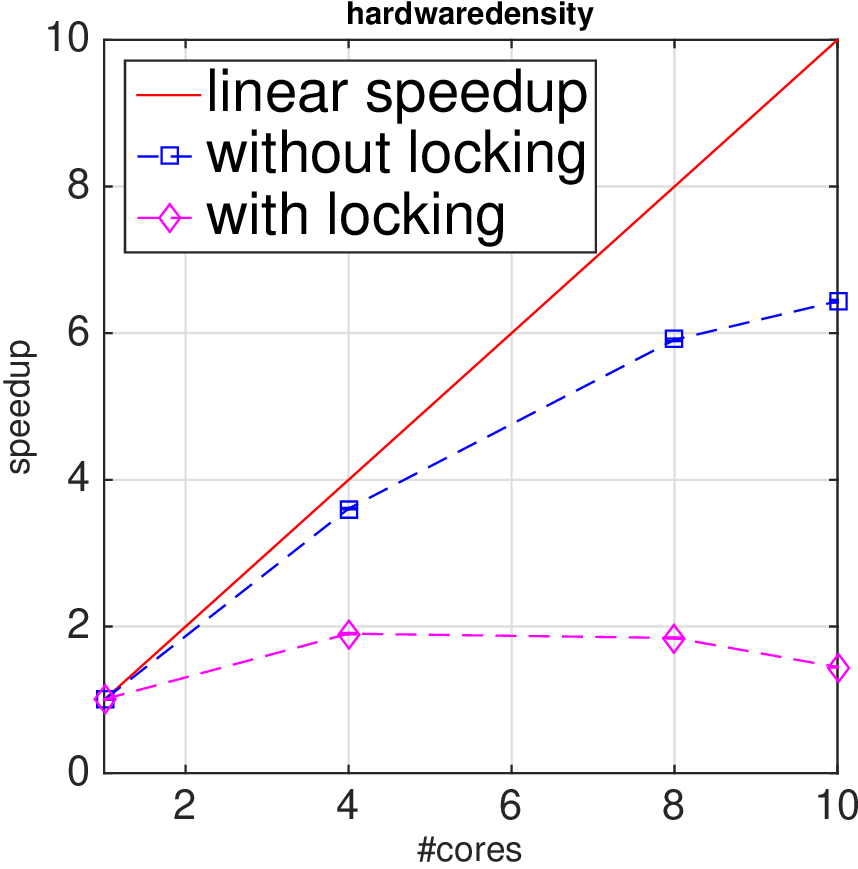} \\
      (a) $\density = .005$ & (b) $\density = .01$
      & (c) $\density = .2$ & (d) $\density = 1$
    \end{tabular}
    \caption{\label{fig:synthetic-sparsity-hardware} Exponentially
      decreasing stepsizes: Speedups for synthetic linear regression
      experiments showing effects of data sparsity on
      speedup~\eqref{eqn:speedup} with epoch-based stepsizes
      $\stepsize_{{\rm epoch}~k} = .95^k$. A fraction $\density$ of each
      vector $a_i \in \R^d$ is non-zero.}
  \end{center}
\end{figure}

In Figures~\ref{fig:synthetic-sparsity-hardware-dec}
and~\ref{fig:synthetic-sparsity-hardware}, we plot the speedup achieved for
the synthetic regression problem with data~\eqref{eqn:simulated-regression}
using different numbers of cores for the experiments in
Figures~\ref{fig:synthetic-sparsity-dec} and~\ref{fig:synthetic-sparsity}
(as before, Fig.~\ref{fig:synthetic-sparsity-hardware-dec} uses stepsizes
$\stepsize_k = k^{-\steppow}$ and Fig.~\ref{fig:synthetic-sparsity-hardware}
uses stepsizes exponentially decreasing between epochs). As a point of
comparison, we also implement a synchronized method, which uses multiple
cores to compute gradients independently $g^i$ on each core $i = 1, \ldots,
c$, then computes the average $\bar{g} = \frac{1}{c} \sum_{i=1}^c g^i$ and
uses that to perform a standard stochastic gradient update; this requires
explicit synchronization (locking) of the updates, though with sufficiently
large batch sizes $\batchsize$ computation may theoretically overwhelm the
communication and locking overhead~\cite{DekelGiShXi12}.
For comparison, we use the same batch size $\batchsize = 10$ for each of the
synchronous and asynchronous procedures, and we see see that the performance
of the naive locking strategy is worse than than the asynchronous gradient
method across all data densities $\density$. At higher densities, more
computation is necessary within each gradient computation, so that the
communication overhead causes less performance degradation for the
synchronous method, yet the asynchronous method also benefits and attains
better relative performance.  We see clearly that data sparsity is not
necessary for the asynchronous gradient method to enjoy substantial
perfromance benefits.

\begin{table}[ht!]
  \begin{center}
    \begin{tabular}{|r|c|c|c|c|}
      \hline 
      \multicolumn{5}{|c|}{No batching ($\batchsize = 1$)} \\
      \hline
      Number of cores & 1 & 4 & 8 & 10 \tabularnewline
      \hline
      fraction of L1 misses & 0.0021 $\pm$ 0.0001 & 0.0061 $\pm$ 0.0001 & 0.0096 $\pm$ 0.0001 & 0.0102 $\pm$ 0.0001 \tabularnewline
      \hline
      fraction of L2 misses & 0.50 $\pm$ 0.01 & 0.63 $\pm$ 0.01 & 0.76 $\pm$ 0.01 & 0.78 $\pm$ 0.01\tabularnewline
      \hline  
      fraction of L3 misses & 0.41 $\pm$ 0.01 & 0.25 $\pm$ 0.01 & 0.24 $\pm$ 0.01 & 0.25 $\pm$ 0.01\tabularnewline
      \hline 
      epoch average time (s) & 4.55 & 1.85 & 1.61 & 1.47\tabularnewline
      \hline 
      \textbf{speedup} & \textbf{1.00} & \textbf{2.46 $\pm$ 0.01} & \textbf{2.83 $\pm$ 0.01} & \textbf{3.09 $\pm$ 0.01} \\
      \hline
      \multicolumn{5}{c}{\vspace{-.1cm}}\\ 
      \hline
      \multicolumn{5}{|c|}{Batch size $\batchsize = 10$} \\
      \hline
      Number of cores & 1 & 4 & 8 & 10 \tabularnewline
      \hline
      fraction of L1 misses & 0.0027 $\pm$ 0.0002 & 0.0033 $\pm$ 0.0001 & 0.0043 $\pm$ 0.0001 & 0.0046 $\pm$ 0.0001\tabularnewline
      \hline
      fraction of L2 misses & 0.44 $\pm$ 0.01 & 0.50 $\pm$ 0.01 & 0.60 $\pm$ 0.01 & 0.63 $\pm$ 0.01\tabularnewline
      \hline 
      fraction of L3 misses & 0.35 $\pm$ 0.03 & 0.33 $\pm$ 0.01 & 0.33 $\pm$ 0.01 & 0.33 $\pm$ 0.01\tabularnewline
      \hline 
      epoch average time (s) & 2.97 & 0.87 & 0.58 & 0.51\tabularnewline
      \hline 
      \textbf{speedup} & \textbf{1.00} & \textbf{3.42 $\pm$ 0.01} & \textbf{5.16 $\pm$ 0.02} & \textbf{5.80 $\pm$ 0.03}\tabularnewline
      \hline
    \end{tabular}
    \caption{\label{table:cache-fun} Memory traffic for batched
      updates~\eqref{eqn:mini-batching} versus non-batched updates
      ($\batchsize = 1$) for a dense linear regression problem in $d = 10^3$
      dimensions with a sample of size $N = 10^6$. Cache misses are
      substantially higher with $\batchsize = 1$.}
  \end{center}
  \vspace{-.3cm}
\end{table}

\subsection{Hardware issues and cache locality}
\label{sec:hardware-fun}

We now detail a set of experiments investigating hardware issues that arise
even in the implementation of asynchronous gradient methods.  The Intel x86
architecture (as with essentially every processor architecture) organizes
memory in a hierarchy, going from level 1 to level 3 (L1 to L3) caches of
increasing sizes. An important aspect of the speed of different optimization
schemes is the relative fraction of memory \emph{hits}, meaning accesses to
memory that is cached locally (in order of decreasing speed, L1, L2, or L3
cache). In Table~\ref{table:cache-fun}, we show the proportion of cache
misses at each level of the memory hierarchy for our synthetic regression
experiment with fully dense data ($\density = 1$) over the execution of 20
epochs, averaged over 10 different experiments. We compare memory contention
when the batch size $\batchsize$ used to compute the local asynchronous
gradients~\eqref{eqn:mini-batching} is 1 and 10. We see that the proportion
of misses for the fastest two levels---1 and 2---of the cache for
$\batchsize = 1$ increase significantly with the number of cores, while
increasing the batch size to $\batchsize = 10$ substantially mitigates cache
incoherence.  In particular, we maintain (near) linear increases in
iteration speed with little degradation in solution quality (the gap
$f(\what{x}) - f(x\opt)$ output by each of the procedures with and without
batching is identical to within $10^{-3}$;
cf.\ Figure~\ref{fig:synthetic-sparsity}(d)).

\begin{figure}[ht]
  \begin{center}
    \begin{tabular}{ccc}
      \hspace{-.1cm}
      \psfrag{optimalitygap}[bl][bl][.6]{$f(x_k) - f(x\opt)$}
      \psfrag{statisticalefficiency}{}
      \psfrag{\#epochs}[bl][bl][.6]{Epochs}
      \includegraphics[width=.32\columnwidth,trim={20 0 0 0},clip=true]{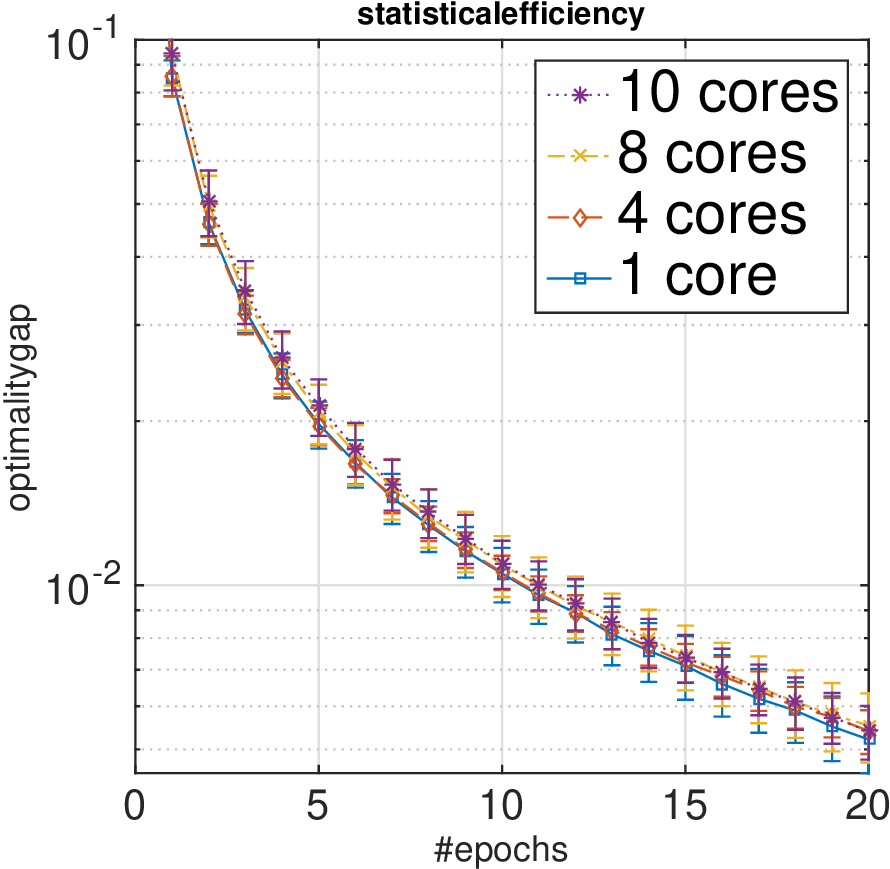} &
      \hspace{-.3cm}
      \psfrag{statisticalefficiency}{}
      \psfrag{optimalitygap}[bl][bl][.6]{}
      \psfrag{\#epochs}[bl][bl][.6]{Epochs}
      \includegraphics[width=.32\columnwidth,trim={20 0 0 0},clip=true]{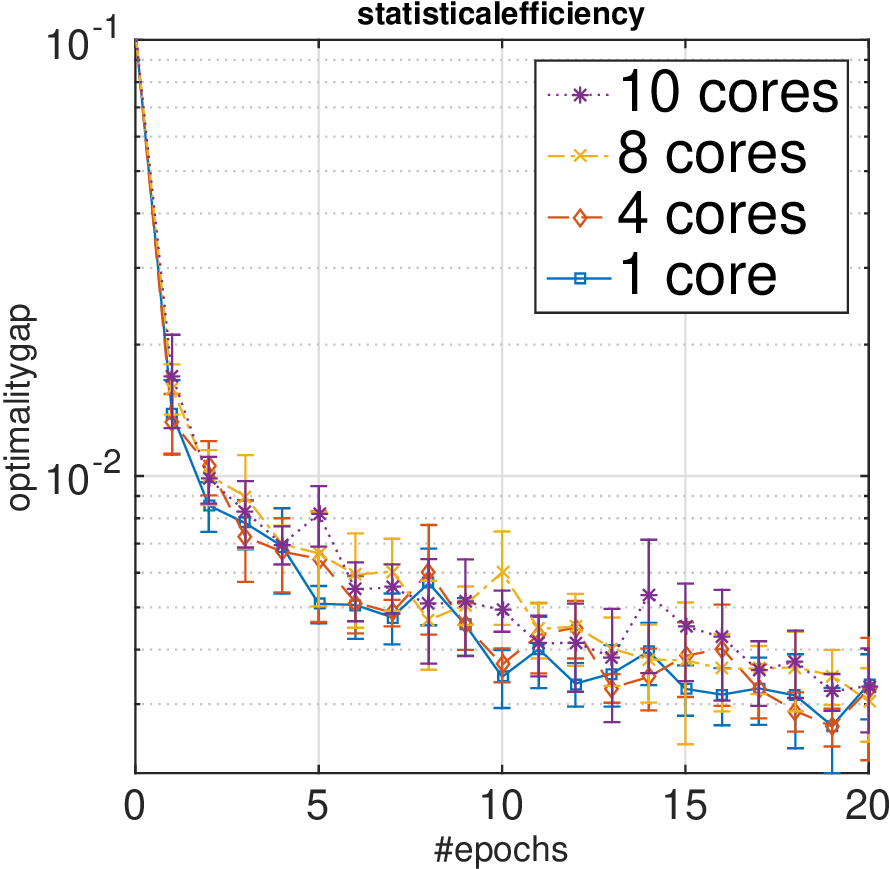} &
      \hspace{-.3cm}
      \psfrag{statisticalefficiency}{}
      \psfrag{\#epochs}[bl][bl][.6]{Epochs}
      \psfrag{optimalitygap}[bl][bl][.6]{}
      \includegraphics[width=.32\columnwidth,trim={20 0 0 0},clip=true]{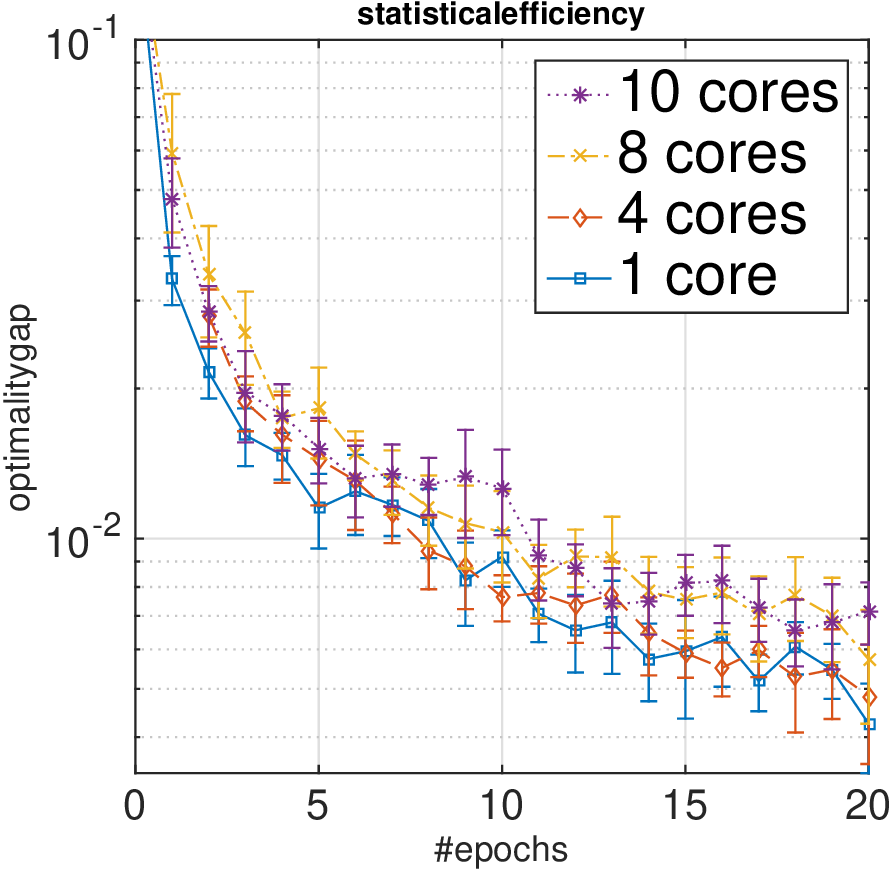}
      \\
      (a) RCV1 ($\density = .003$) & (b) Higgs ($\density = 1$)
      & (c) Forest ($\density = 1$)
    \end{tabular}
    \caption{\label{fig:real-statistical-efficiency-dec} Decreasing
      stepsizes: Optimality gaps $f(x_k) - f(x\opt)$ on the (a) RCV1, (b)
      Higgs, and (c) Forest Cover datasets using stepsize $\stepsize_k =
      k^{-\steppow}$ with $\steppow = .55$.}
  \end{center}
\end{figure}
\begin{figure}[th]
  \begin{center}
    \begin{tabular}{ccc}
      \hspace{-.1cm}
      \psfrag{optimalitygap}[bl][bl][.8]{$f(x_k) - f(x\opt)$}
      \psfrag{statisticalefficiency}{}
      \psfrag{\#epochs}[bl][bl][.8]{Epochs}
      \includegraphics[width=.32\columnwidth,trim={20 0 0 0},clip=true]{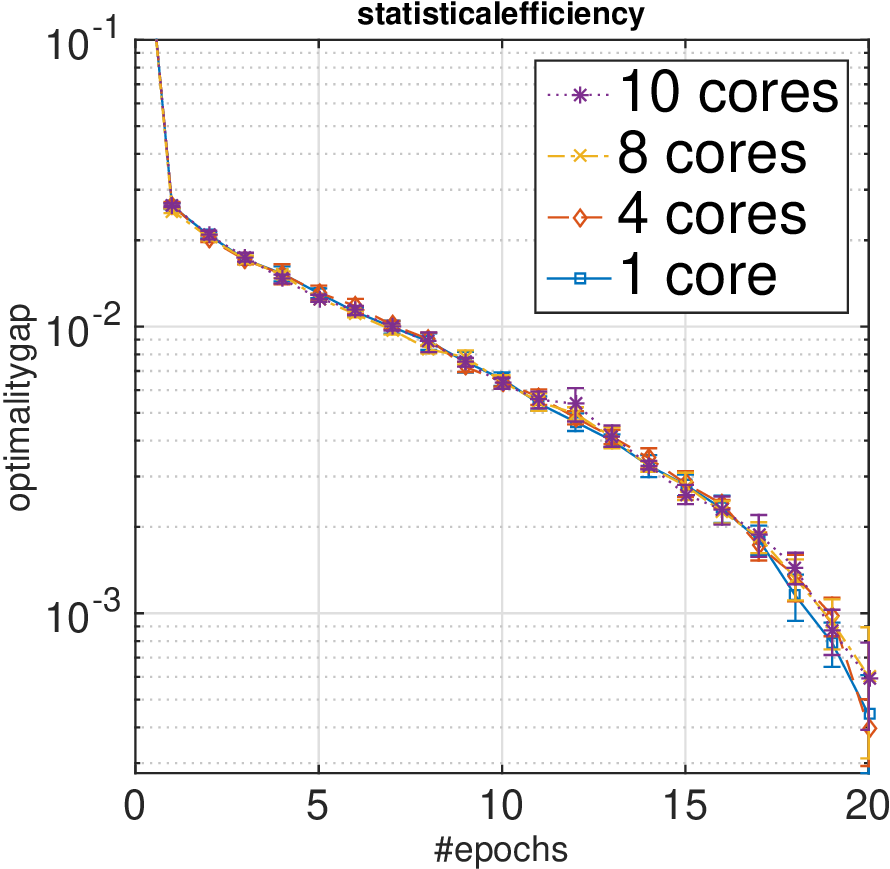} &
      \hspace{-.3cm}
      \psfrag{statisticalefficiency}{}
      \psfrag{optimalitygap}[bl][bl][.6]{}
      \psfrag{\#epochs}[bl][bl][.8]{Epochs}
      \includegraphics[width=.32\columnwidth,trim={20 0 0 0},clip=true]{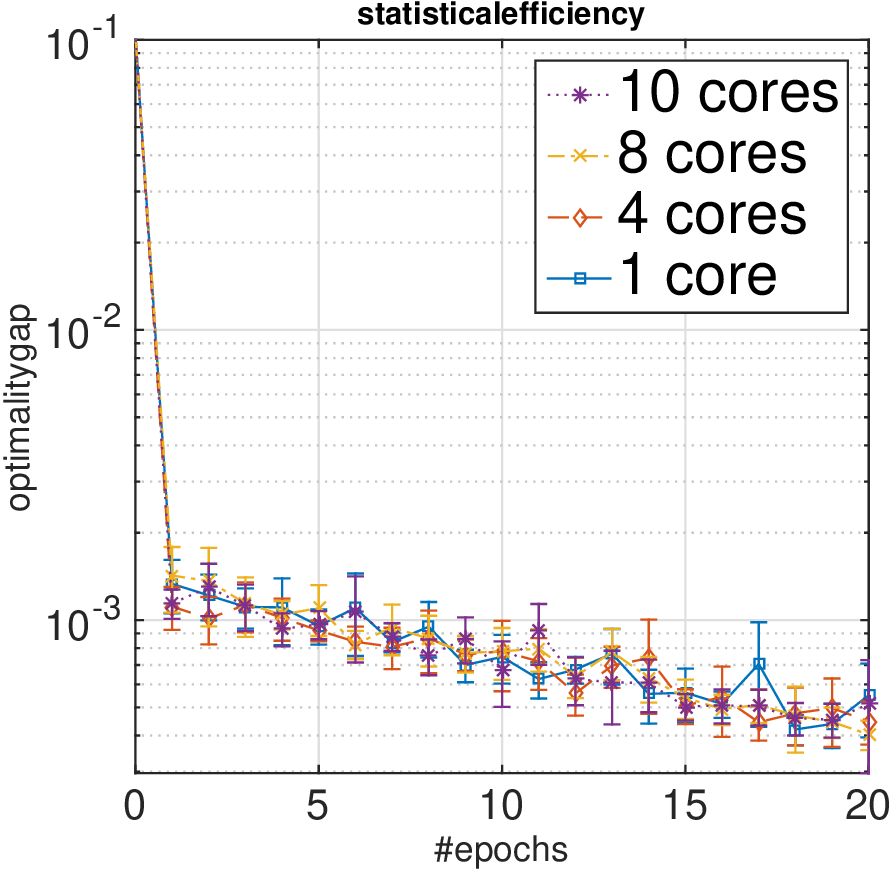} &
      \hspace{-.3cm}
      \psfrag{statisticalefficiency}{}
      \psfrag{\#epochs}[bl][bl][.8]{Epochs}
      \psfrag{optimalitygap}[bl][bl][.6]{}
      \includegraphics[width=.32\columnwidth,trim={20 0 0 0},clip=true]{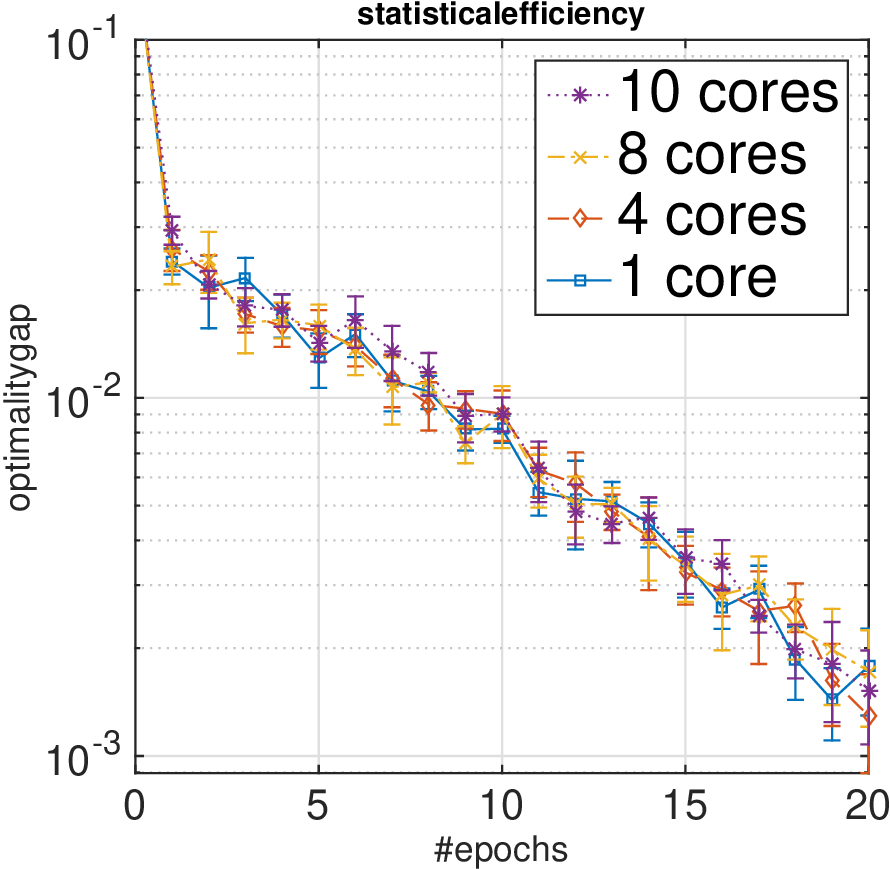}
      \\
      (a) RCV1 ($\density = .003$) & (b) Higgs ($\density = 1$)
      & (c) Forest ($\density = 1$)
    \end{tabular}
    \caption{\label{fig:real-statistical-efficiency} Exponentially decreasing
      stepsizes: Optimality gaps
      $f(x_k) - f(x\opt)$ on the (a) RCV1, (b) Higgs, and (c) Forest
      Cover datasets with epoch-based stepsizes
      $\stepsize_{{\rm epoch}~k} = .95^k$.}
  \end{center}
\end{figure}

\subsection{Real datasets}

We perform experiments using three different real-world datasets: the
Reuters RCV1 corpus~\cite{LewisYaRoLi04}, the Higgs detection
dataset~\cite{BaldiSaWh14}, and the Forest Cover dataset~\cite{Lichman13}.
Each represents a binary classification problem, which we formulate using
logistic regression (recall Sec.~\ref{sec:examples}). We briefly detail
relevant statistics of each:
\begin{enumerate}[(1)]
\item The Reuters RCV1 dataset consists of $N \approx 7.81 \cdot 10^5$ data
  vectors (documents) $a_i \in \{0, 1\}^d$ with $d \approx 5 \cdot 10^4$
  dimensions; each vector has sparsity approximately $\density = 3 \cdot
  10^{-3}$. Our task is to classify each document as being about corporate
  industrial topics (CCAT) or not.
\item The Higgs detection dataset consists of $N = 10^6$ data vectors
  $\wt{a}_i \in \R^{d_0}$, with $d_0 = 28$. We quantize each coordinate into
  5 bins containing equal fraction of the coordinate values and encode each
  vector $\wt{a}_i$ as a vector $a_i \in \{0, 1\}^{d}$ with $d = 5d_0$ whose
  non-zero entries correspond to quantiles into which coordinates fall. The
  task is to detect (simulated) emissions from a linear accelerator.
\item The Forest Cover dataset consists of $N \approx 5.7 \cdot 10^5$ data
  vectors $a_i \in \{-1, 1\}^d$ with $d = 54$, and the task is to predict
  forest growth types.
\end{enumerate}
Thus, each dataset gives a different flavor of optimization problem: the
first is very sparse and high-dimensional, the second is somewhat sparse and
of moderate dimension, while the forest dataset is dense but of relatively
small dimension. These allow a broader picture of the performance of the
asynchronous gradient method.

\begin{figure}[ht]
  \begin{center}
    \begin{tabular}{ccc}
      \hspace{-.1cm}
      \psfrag{optimalitygap}[bl][bl][.8]{Speedup}
      \psfrag{hardwareefficiency}{}
      \psfrag{\#cores}[bl][bl][.6]{Cores}
      \includegraphics[width=.32\columnwidth,trim={20 0 0 0},clip=true]{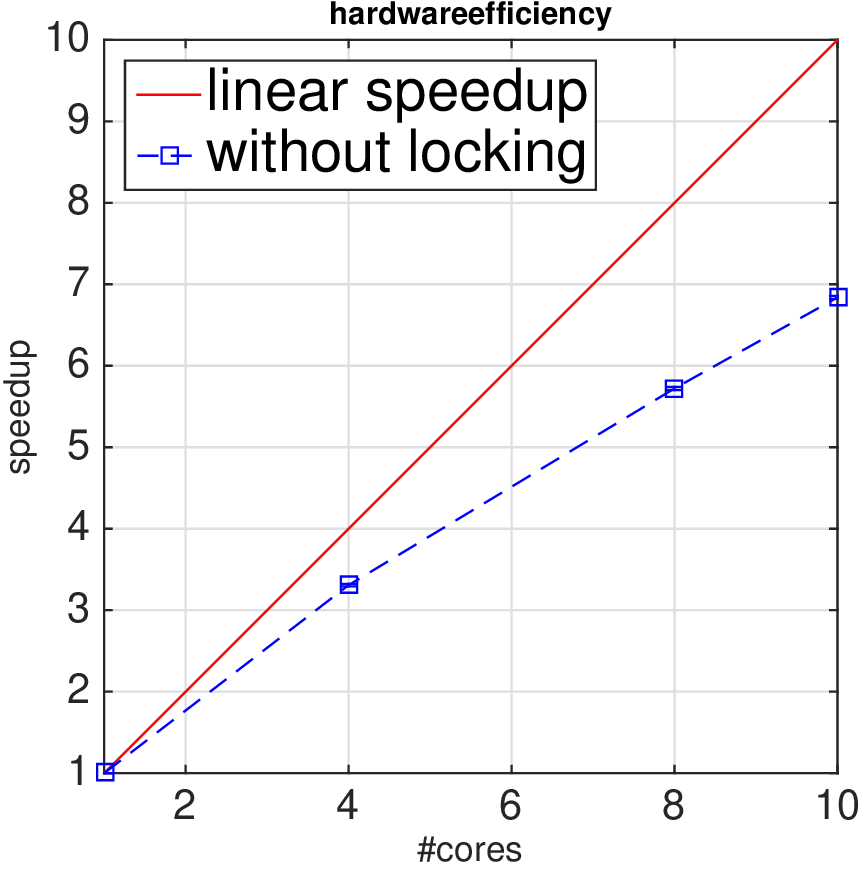} &
      \hspace{-.1cm}
      \psfrag{hardwareefficiency}{}
      \psfrag{\#cores}[bl][bl][.6]{Cores}
      \psfrag{optimalitygap}[bl][bl][.6]{}
      \includegraphics[width=.32\columnwidth,trim={20 0 0 0},clip=true]{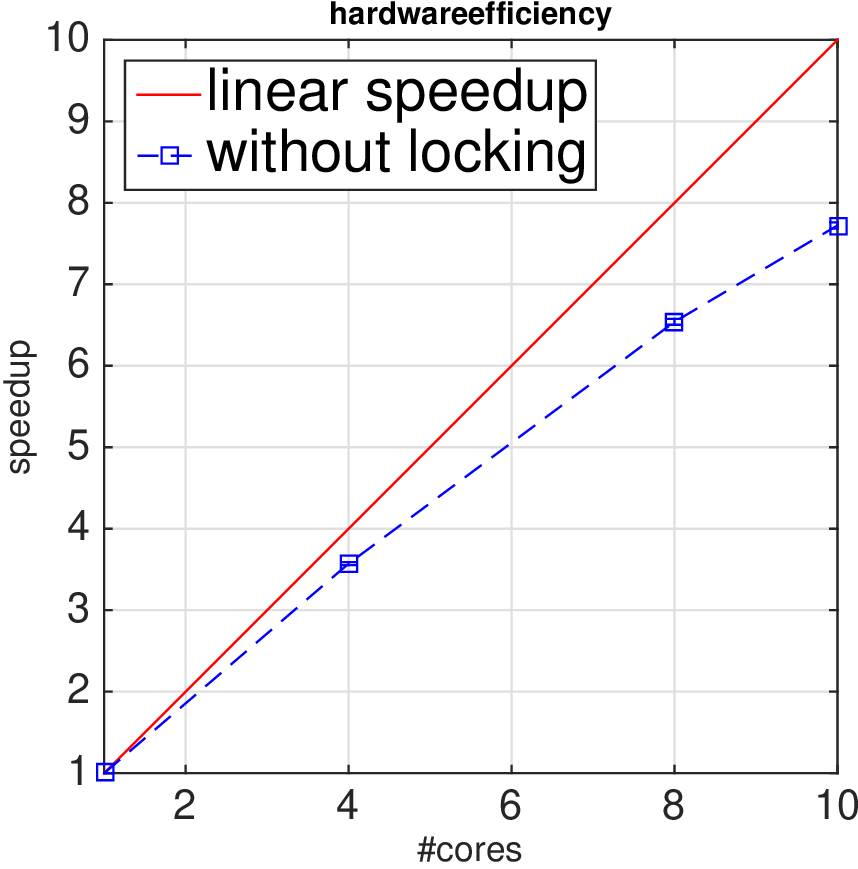} &
      \hspace{-.1cm}
      \psfrag{hardwareefficiency}{}
      \psfrag{optimalitygap}[bl][bl][.6]{}
      \psfrag{\#cores}[bl][bl][.6]{Cores}
      \includegraphics[width=.32\columnwidth,trim={20 0 0 0},clip=true]{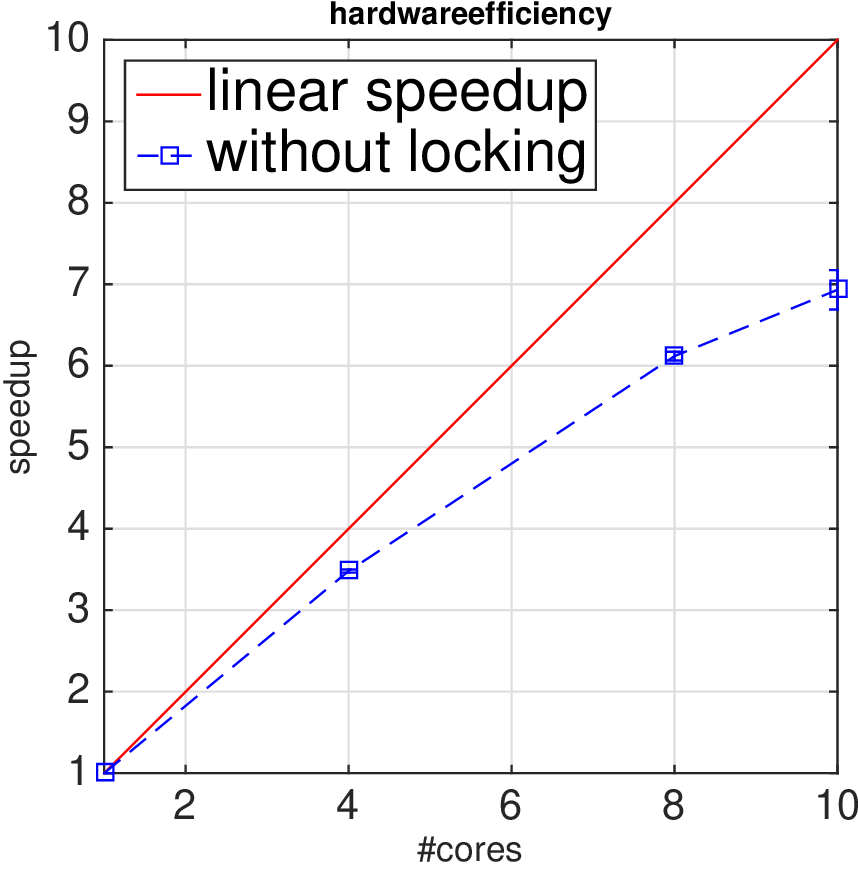}
      \\
      (a) RCV1 ($\density = .003$) & (b) Higgs ($\density = 1$)
      & (c) Forest ($\density = 1$)
    \end{tabular}
    \caption{\label{fig:real-hardware-efficiency-dec}
      (Decreasing stepsizes) Logistic regression experiments showing
      speedup~\eqref{eqn:speedup} on the (a) RCV1, (b) Higgs,
      and (c) Forest Cover datasets.}
  \end{center}
\end{figure}

We follow the same experimental protocol as in our simulated data
experiments.  That is, we perform 10 experiments for each dataset using 1,
4, 8, and 10 cores, where each experiment consists of running the
asynchronous gradient method for 20 epochs, within each of which examples
are accessed according to a new random permutation. We use a batch size
$\batchsize = 10$ for each experiment, and collect standard errors for the
(estimated) optimality gaps. In these experiments, as a proxy for the
optimal value $f(x\opt)$ we run a synchronous gradient method for 100
epochs, using its best objective value as $f(x\opt)$.  In
Figures~\ref{fig:real-statistical-efficiency-dec}
and~\ref{fig:real-statistical-efficiency}, we plot the gap $f(x_k) -
f(x\opt)$ as a function of epochs, giving standard error intervals, for each
of the three datasets. The figures show there is essentially no degradation
in objective value when using different numbers of processors, that is,
asynchrony appears to have negligble effect for each problem, whether we use
stepsizes $\stepsize_k = k^{-\steppow}$ (with $\steppow = .55$) that we
analyze or the epoch-based exponentially decreasing stepsize scheme used by
\citet{NiuReReWrNi11}, also analyzed by \citet{HazanKa11} and
\citet{GhadimiLa12}.  In Figures~\ref{fig:real-hardware-efficiency-dec}
and~\ref{fig:real-hardware-efficiency}, we plot speedup achieved for these
same experiments. The asynchronous gradient method iteration achieves nearly
linear speedup of between 6$\times$ and 8$\times$ on each of the datasets
using 10 cores.

\begin{figure}[ht]
  \begin{center}
    \begin{tabular}{ccc}
      \hspace{-.1cm}
      \psfrag{optimalitygap}[bl][bl][.8]{Speedup}
      \psfrag{hardwareefficiency}{}
      \psfrag{\#cores}[bl][bl][.8]{Cores}
      \includegraphics[width=.32\columnwidth,trim={20 0 0 0},clip=true]{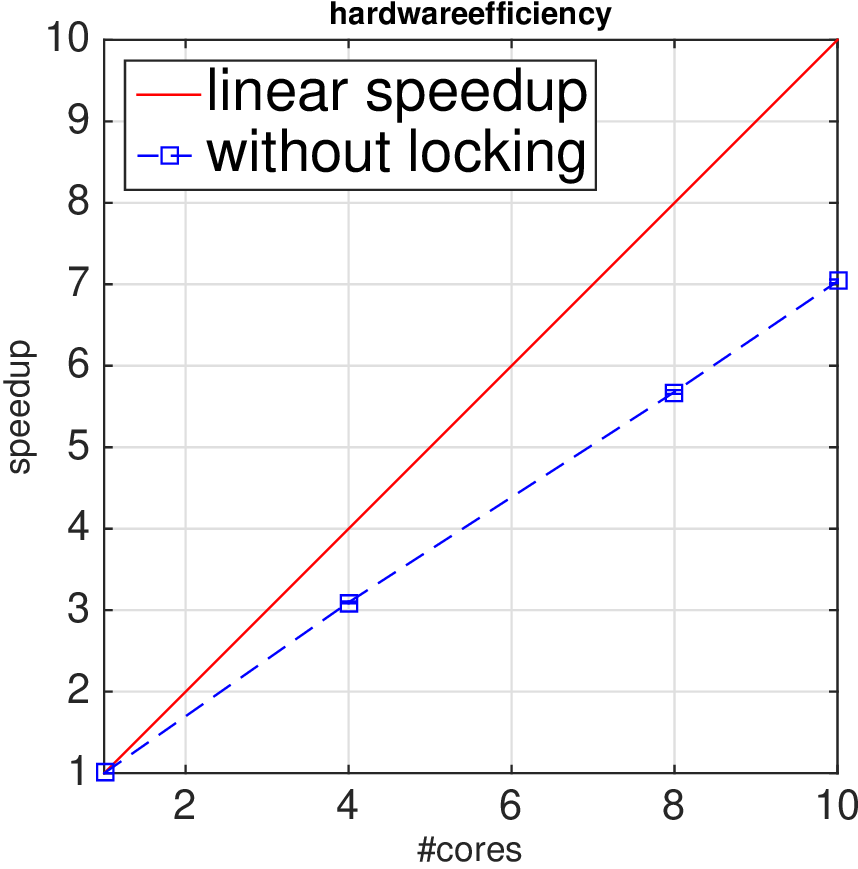} &
      \hspace{-.1cm}
      \psfrag{hardwareefficiency}{}
      \psfrag{\#cores}[bl][bl][.8]{Cores}
      \psfrag{optimalitygap}[bl][bl][.8]{}
      \includegraphics[width=.32\columnwidth,trim={20 0 0 0},clip=true]{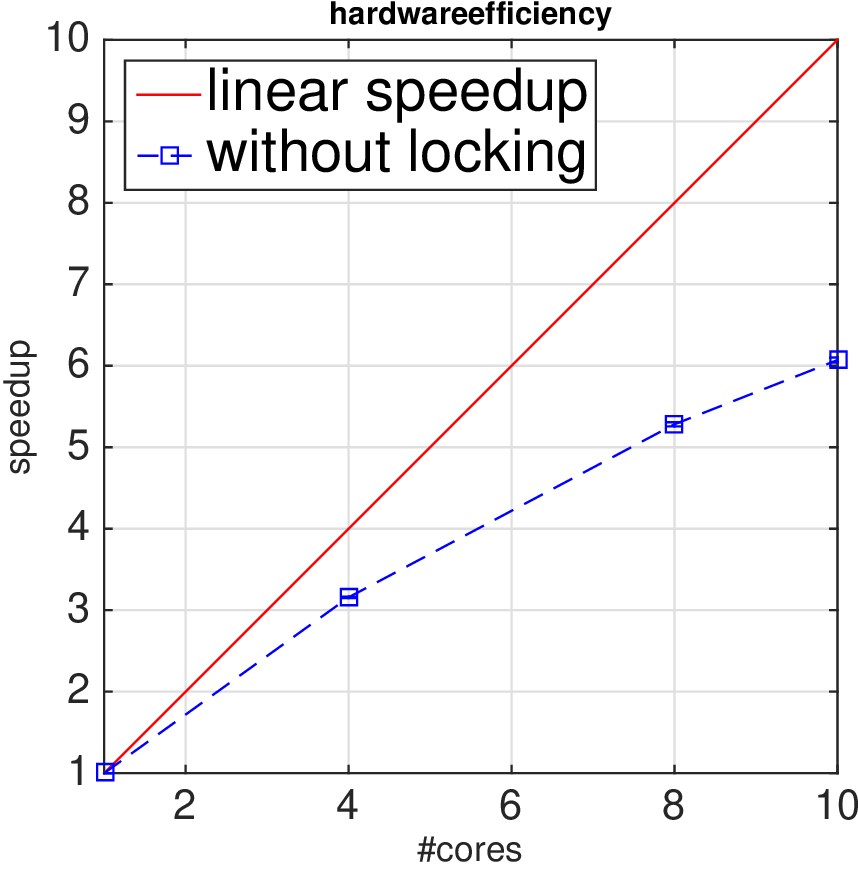} &
      \hspace{-.1cm}
      \psfrag{hardwareefficiency}{}
      \psfrag{optimalitygap}[bl][bl][.8]{}
      \psfrag{\#cores}[bl][bl][.8]{Cores}
      \includegraphics[width=.32\columnwidth,trim={20 0 0 0},clip=true]{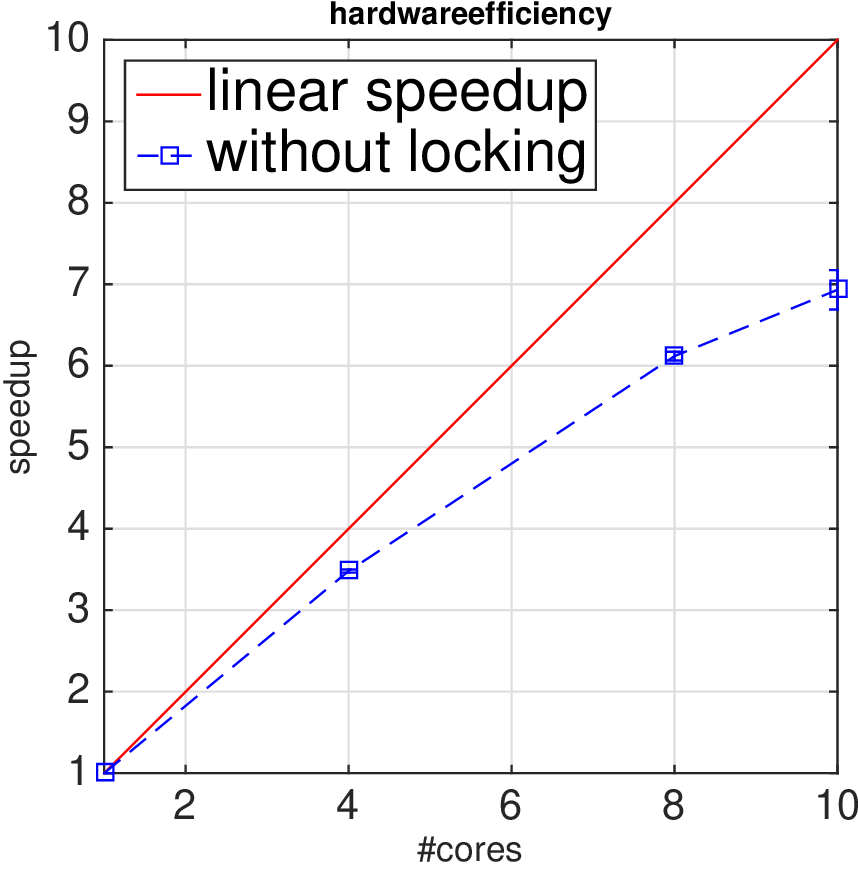}
      \\
      (a) RCV1 ($\density = .003$) & (b) Higgs ($\density = 1$)
      & (c) Forest ($\density = 1$)
    \end{tabular}
    \caption{\label{fig:real-hardware-efficiency} Exponentially decreasing
      stepsizes: Logistic regression experiments showing
      speedup~\eqref{eqn:speedup} on the (a) RCV1, (b) Higgs, and (c) Forest
      Cover datasets with epoch-based stepsize $\stepsize_{{\rm epoch}~k} =
      .95^k$.}
  \end{center}
\end{figure}

\section{Proofs}
\label{sec:proofs}

In this section, we present proofs of our two main theorems, deferring
the proofs of technical lemmas to subsequent appendices.

\subsection{Proof of Theorem~\ref{theorem:convex-optimization}}
\label{sec:proof-convex-optimization}

We prove Theorem~\ref{theorem:convex-optimization} by a reduction to
Theorem~\ref{theorem:nonlinear}.  For both settings of
Theorem~\ref{theorem:convex-optimization}, we will use $\lyap(x) = \half
\norm{x}^2$ and $\resid(x) = \nabla f(x)$, then apply
Theorem~\ref{theorem:nonlinear}. With these choices,
Assumption~\ref{assumption:residual-quadratic} is satisfied with the Hessian
$\hess = \nabla^2 f(x\opt) \succ 0$ and $\gamma = 1$ by a Taylor expansion
of $f$, which is assumed continuously twice differentiable near $x\opt$.

Let us also verify that Assumption~\ref{assumption:noise-fun} holds for
an appropriate noise sequence $\noise_k$ in the stochastic
convex optimization setting. Throughout, we
also use the sequence of $\sigma$-fields $\mc{F}_k$ defined by
expression~\eqref{eqn:filtration}, so that $x_k \in \mc{F}_{k-1}$, and the
counter $k$ implicitly gives a gradient $g_k$, point $x_k$, sample
$\statrv_k$, and noise error
\begin{equation}
  \noise_k = g_k - \resid(x_k) = \nabla F(x_k; \statrv_k) - \nabla f(x_k).
  \label{eqn:convex-noise}
\end{equation}
That is, we have $g_k = \resid(x_k) + \noise_k$ as in the nonlinear setting
of Theorem~\ref{theorem:nonlinear}.  We first verify that $\noise_k$ is a
martingale difference sequence: we have $\E[\noise_k \mid \mc{F}_{k-1}] =
\nabla f(x_k) - \nabla f(x_k) = 0$ because $\statrv_k$ is independent of
$x_k$. We also have the decomposition
\begin{equation*}
  \noise_k
  = \underbrace{\nabla F(x\opt; \statrv_k)}_{
    \noise_k(0)} + \underbrace{\nabla F(x_k; \statrv_k)
    - \nabla F(x\opt; \statrv_k) - \nabla f(x_k)}_{
    \noisier_k(x_k)},
\end{equation*}
and that $\norm{\nabla f(x_k)} = \norm{\nabla f(x_k) - \nabla f(x\opt)}
\le L \norm{x_k - x\opt}$, so
\begin{equation*}
  \E[\norms{\noisier_k(x_k)}^2 \mid \mc{F}_{k-1}]
  = \E[\norms{\nabla F(x_k; \statrv_k) - \nabla F(x\opt; \statrv_k)}^2
    \mid \mc{F}_{k-1}]
  + \norm{\nabla f(x_k)}^2
  \le C \norm{x_k - x\opt}^2
\end{equation*}
by inequality~\eqref{eqn:smooth-random-gradients}. Moreover, we have
\begin{equation*}
  \E[\noise_k(0) \noise_k(0)^\top \mid \mc{F}_{k-1}] = \E[\nabla F(x\opt;
    \statrv_k) \nabla F(x\opt; \statrv_k)^\top \mid \mc{F}_{k-1}] = \Sigma,
\end{equation*}
as the random variable $\statrv_k$ is independent of $\mc{F}_{k-1}$.
That is, Assumption~\ref{assumption:noise-fun} holds.

Now we show that Assumption~\ref{assumption:resid-strong-convex} holds
whenever Assumption~\ref{assumption:strong-convex} holds.
If $f$ is strongly convex, then
taking $\resid(x) = \nabla f(x)$ and $\lyap(x) = \half \norm{x}^2$
we have for any $x, y \in \R^d$ that
\begin{equation*}
  f(y) \ge f(x) + \<\nabla f(x), y - x\> + \frac{\lambda}{2} \norm{x - y}^2
  ~~~ \mbox{and} ~~~
  f(x) \ge f(y) + \<\nabla f(y), x - y\> + \frac{\lambda}{2} \norm{x - y}^2.
\end{equation*}
Taking $y = x\opt$ in the preceding expression while noting that $\nabla
f(x\opt) = \nabla f(y) = 0$, we have
\begin{align*}
  \<\nabla \lyap(x - x\opt), \resid(x)\>
  & = \<x - x\opt, \nabla f(x)\>
  = \<\nabla f(x) - \nabla f(x\opt), x - x\opt\>
  \ge \lambda \norm{x - x\opt}^2
  = 2\lambda \lyap(x - x\opt).
\end{align*}
Clearly, $\lyap$ has $1$-Lipschitz gradient and satisfies
$\lyap(x - x\opt) \ge \half \norm{x - x\opt}^2$.

Now we consider conditions on the convex function $f$ under which
Assumption~\ref{assumption:resid-lipschitz} is satisfied.  In particular,
let $f$ be a differentiable Lipschitz-continuous convex function defined on
$\R^d$, and assume that $f$ is locally strongly convex near
$x\opt$, meaning that there exist $\lambda > 0$ and $\epsilon > 0$ such that
\begin{equation*}
  f(y) \ge f(x) + \<\nabla f(x), y - x\> + \frac{\lambda}{2} \norm{x - y}^2
  ~~ \mbox{for~} x, y ~\mbox{s.t.} ~ \norm{x - x\opt} \le \epsilon,
  \norm{y - x\opt} \le \epsilon.
\end{equation*}
In particular, as $\nabla f(x\opt) = 0$, we have $f(x) \ge f(x\opt) +
\frac{\lambda}{2} \norm{x - x\opt}^2$ for all $x$ such that $\norm{x - x\opt}
\le \epsilon$.  We have the following lemma on the growth of such
functions.
\begin{lemma}
  \label{lemma:convex-grows}
  Let $f$ satisfy the conditions in the preceding paragraph.  Then
  \begin{equation}
    \label{eqn:convex-grows}
    f(x) \ge f(x\opt) + \frac{\lambda}{2}
    \min\{\norm{x - x\opt}^2, \epsilon \norm{x - x\opt}\}.
  \end{equation}
\end{lemma}

Deferring proof of Lemma~\ref{lemma:convex-grows}, we show how it implies
that the conditions of Assumption~\ref{assumption:resid-lipschitz} are
satisfied with the Lyapunov function $\lyap(x) = \half \norm{x}^2$ and
residual operator $\resid(x) = \nabla f(x)$. Indeed, by applying the strong
convexity inequality with $y = x\opt$, we have for $x$ such that $\norm{x -
  x\opt} \le \epsilon$ that
\begin{equation*}
  \<\nabla \lyap(x - x\opt), \resid(x)\>
  = \<x - x\opt, \nabla f(x)\>
  = \<\nabla f(x) - \nabla f(x\opt), x - x\opt\>
  \ge \lambda \norm{x - x\opt}^2.
\end{equation*}
Now we claim that
\begin{equation}
  \label{eqn:lipschitz-lyap}
  \inf_{x : \norm{x - x\opt} > \epsilon}
  \<\nabla \lyap(x - x\opt), \resid(x)\> =
  \inf_{x : \norm{x - x\opt} > \epsilon}
  \<x - x\opt, \nabla f(x)\> > 0.
\end{equation}
To see this, note that
by claim~\eqref{eqn:convex-grows}, for $x$ such that
$\norm{x - x\opt} \ge \epsilon$, we have for some constant $c > 0$ that
$f(x) - f(x\opt) \ge c \norm{x - x\opt}$,
while we have $f(x\opt) \ge f(x) + \<\nabla f(x), x\opt - x\>$, so that
\begin{equation*}
  \<\nabla f(x), x - x\opt\>
  \ge f(x) - f(x\opt) \ge c \norm{x - x\opt} > c \hspace{.05em}\epsilon
  ~~ \mbox{for~all~} x ~\mbox{s.t.}~ \norm{x - x\opt} > \epsilon.
\end{equation*}
Additionally, whenever Assumption~\ref{assumption:lipschitz}
holds, we have
$\E[\norm{\nabla F(x; \statrv)}^2] \le \lipobj^2$,
so that all the conditions of Assumption~\ref{assumption:resid-lipschitz}
are satisfied. Except for the proof of Lemma~\ref{lemma:convex-grows},
this completes the proof of Theorem~\ref{theorem:convex-optimization}.

\begin{proof-of-lemma}[\ref{lemma:convex-grows}]
  Fix $y \in \R^d$ and let $h_y(t) = f(x\opt + t y / \norm{y}) - f(x\opt)$.
  Notably, $h_y$ is a one-dimensional convex function, and
  $h_y(t) \ge (\lambda/2) t^2$ for $|t| \le \epsilon$.
  As the slopes of convex functions are non-decreasing
  (cf.~\citet[Chapter I]{HiriartUrrutyLe93}), we have
  \begin{equation*}
    h_y'(\epsilon) = \lim_{\delta \to 0}
    \frac{h_y(\epsilon + \delta) - h_y(\epsilon)}{\delta}
    \ge \frac{h_y(\epsilon) - h_y(0)}{\epsilon}
    \ge \frac{\lambda \epsilon^2}{2 \epsilon}
    = \frac{\lambda \epsilon}{2}.
  \end{equation*}
  This inequality implies that for any $t \ge \epsilon$, we have
  \begin{equation*}
    \frac{f(x\opt + t \frac{y}{\norm{y}}) - f(x\opt)}{t}
    = \frac{h_y(t) - h_y(0)}{t}
    \ge \frac{h_y(\epsilon) - h_y(0)}{\epsilon}
    \ge \frac{\lambda \epsilon}{2},
  \end{equation*}
  while for $0 < t < \epsilon$, we use that $t \mapsto (h_y(t) - h_y(0))/t$ is
  non-decreasing in $t$ to obtain
  \begin{equation*}
    \frac{f(x\opt + t \frac{y}{\norm{y}}) - f(x\opt)}{t}
    = \frac{h_y(t) - h_y(0)}{t}
    \ge \frac{\lambda t^2}{2t} = \frac{\lambda t}{2}.
  \end{equation*}
  Combining the two preceding displays, we have
  $f(x\opt + t \frac{y}{\norm{y}})
  - f(x\opt) \ge \frac{\lambda t}{2} \min\left\{\epsilon, t\right\}$,
  which is equivalent to inequality~\eqref{eqn:convex-grows}.
\end{proof-of-lemma}

\subsection{Proof of Theorem~\ref{theorem:nonlinear}}
\label{sec:proof-nonlinear}

Before beginning the proof of the theorem proper, we state a martingale
convergence lemma necessary for our development, then give an outline of
the proof to come.
\begin{lemma}[Robbins and Siegmund~\cite{RobbinsSi71}]
  \label{lemma:robbins-siegmund}
  Let $\mc{F}_1 \subset \mc{F}_2 \subset \cdots$ be a filtration and
  $V_n, \beta_n, \kappa_n, \varepsilon_n$ be non-negative
  $\mc{F}_n$-measurable random variables such that
  \begin{equation*}
    \E[V_{n+1} \mid \mc{F}_n] \le (1 + \beta_n) V_n + \kappa_n - \varepsilon_n.
  \end{equation*}
  On the event that $\sum_{n=1}^\infty \beta_n < \infty$
  and $\sum_{n = 1}^\infty \kappa_n < \infty$, we have
  $V_n \cas V$ for a non-negative random variable $V$ with
  $V < \infty$ almost surely, and $\sum_{n=1}^\infty \varepsilon_n < \infty$
  a.s.
\end{lemma}

We prove Theorem~\ref{theorem:nonlinear} by relating the sequence $x_k$ from
expression~\eqref{eqn:iterate-formula} to a sequence whose performance is
somewhat easier to analyze, and which has values more closely approximating
a ``correct'' stochastic gradient iteration: we define
\begin{equation}
  \label{eqn:ideal-iteration}
  \wt{x}_k \defeq -\sum_{i=1}^{k-1} \stepsize_i g_i
  ~~~ \mbox{and} ~~~
  \wt{\error}_k \defeq \wt{x}_k - x\opt.
\end{equation}
With this iteration, we have that $\wt{x}_k \in \mc{F}_{k-1}$, where
$\mc{F}_k$ is the $\sigma$-field defined in
expression~\eqref{eqn:filtration}, and (we show) it is close enough to the
correct iterates $x_k$ to give our desired results. The idea of analyzing a
corrected sequence for distributed iterations builds out of \citet[Chapter
  7.8]{BertsekasTs89} and has been
used~\cite[e.g.][]{DuchiAgWa12} for distributed and parallel optimization
problems.

Because delays may grow unboundedly under
Assumption~\ref{assumption:delay-moments}, we consider the effects of delay
increasing polynomially with $n$. With this in mind, for the proof and all
internal lemmas, we let $\npow$ be a fixed constant satisfying
\begin{equation}
  \label{eqn:delay-moment-interval}
  \frac{1}{\delmoment - 1} < \npow < \steppow - \frac{1}{1 + \gamma},
\end{equation}
where the stepsizes $\stepsize_k = \stepsize k^{-\steppow}$, $\delmoment$ is
the moment in Assumption~\ref{assumption:delay-moments}, and $\gamma \in
\openleft{0}{1}$ is the power in
Assumption~\ref{assumption:residual-quadratic}. The
interval~\eqref{eqn:delay-moment-interval} is assumed to be non-empty by the
conditions of the theorem. Note that this implies that $\npow \in
(\frac{1}{\delmoment - 1}, \steppow - \half)$. We will show that $n^\npow$
functions as a bound on the delays in incorporating gradient information.

\paragraph{Outline of proof}
We provide a brief outline before giving the remainder of the proof.
First, we show that there is (asymptotically) a finite bound such that the
delays $\delay_n$ on asynchronous updates are at most of order $n^\npow$ by
Assumption~\ref{assumption:delay-moments} (Lemmas~\ref{lemma:delays-stop}
and~\ref{lemma:crazy-ratios}).  Then we show that the ``corrected'' sequence
$\wt{x}_k$ (and $\wt{\error}_k$) converges appropriately
(Lemma~\ref{lemma:p1-convergence}), assuming that the true errors $\error_k$
do not diverge, giving almost sure convergence of $\wt{\error}_k$ using the
Robbins-Siegmund martingale convergence theorem
(Lemma~\ref{lemma:robbins-siegmund}).  We use these results to show that
$\error_k$, $\wt{\error}_k$, and $\wt{\error}_k'$, where $\wt{\error}_k'$ is
defined by the simpler linear matrix iteration $v_{k + 1} = (I - \stepsize_k
\hess) v_k - \stepsize_k \noise_k$ with $\wt{\error}_k' = v_k - x\opt$, are
all asymptotically equivalent in probability
(Lemmas~\ref{lemma:wrong-is-right}, \ref{lemma:corrected-errors-matrices},
and~\ref{lemma:corrected-errors-matrices-lipschitz}), as long as the errors
$\error_k$ are assumed to stay bounded. In particular, the differences
$\norms{\error_k - \wt{\error}_k}^2 = O_P(\stepsize_k^2 k^{2 \npow})$, that
is, they scale quadratically in the stepsize $\stepsize_k$ and with some
penalty for delays, and the errors tend to zero as long as $\npow$ is not
too large; asynchrony is dominated by the magnitude of observed
gradient noise.  Asymptotic normality
(Lemma~\ref{lemma:corrected-sequence-normal}) of the equivalent
sequences $\error_k, \wt{\error}_k, \wt{\error}_k'$ then follows from
results of \citet{PolyakJu92}, which guarantee a central limit theorem for
the sequence $\wt{\error}_k'$, because the error bounds on $\noisier$ of
Assumption~\ref{assumption:noise-fun} guarantee that $\noise_k$ eventually
behaves like an i.i.d.\ sequence.  Lastly (in
Lemma~\ref{lemma:smaller-than-past}), we show that our overarching
assumption---that the true errors $\error_k$ did not diverge---in fact holds
under the assumptions of the theorem.

We now turn to the proof of Theorem~\ref{theorem:nonlinear} proper.

\begin{lemma}
  \label{lemma:delays-stop}
  Let $\npow > \frac{1}{\delmoment - 1}$ and
  $\event_n$ be the event that $\errmat^{nk} \neq I_{d \times d}$ for some
  $k \le n - n^\npow$. Then
  \begin{equation*}
    \P(\event_n ~ \mbox{occurs~infinitely~often}) = 0.
  \end{equation*}
\end{lemma}
\begin{proof}
  We have that
  $\errmat^{nk} \neq I$ if and only if $\delay_k \ge n - k + 1$, so that
  \begin{equation*}
    \P(I \neq \errmat^{nk})
    \le \P(\delay_k \ge n - k + 1)
    \le \frac{\E[\delay_k^\delmoment]}{(n - k + 1)^\delmoment}
    \le \frac{\delay^\delmoment}{(n - k + 1)^\delmoment}.
  \end{equation*}
  Letting $\event_n$ be the event that
  $I \neq \errmat^{nk}$ for some $k \le n - n^\npow$ as in the
  statement of the lemma,
  \begin{align*}
    \P(\event_n) & = \P(M_k \ge n - k + 1 ~ \mbox{for~some~} k \le n-n^\npow)
    \le \sum_{k=1}^{n - n^\npow} \frac{\delay^\delmoment}{
      (n - k + 1)^\delmoment} \\
    & = \sum_{k = n^\npow + 1}^n \frac{\delay^\delmoment}{k^\delmoment}
    \lesssim \int_{n^\npow}^n t^{-\delmoment} dt
    \lesssim (n^\npow)^{1 - \delmoment}
    = n^{\npow(1 - \delmoment)}.
  \end{align*}
  Thus we find that
  \begin{equation*}
    \sum_{n = 1}^\infty \P(\event_n)
    \lesssim \sum_{n = 1}^\infty \frac{1}{n^{\npow(\delmoment - 1)}}
    \stackrel{(i)}{<} \infty,
  \end{equation*}
  where inequality~(i) holds if and only if
  $\npow (\delmoment - 1) > 1$, or
  $\npow > \frac{1}{\delmoment - 1}$. Applying the Borel-Cantelli lemma
  gives the result.
\end{proof}

As an immediate consequence of this lemma, we obtain the following.
\begin{lemma}
  \label{lemma:crazy-ratios}
  Let $\stepsize_k = \stepsize k^{-\steppow}$, where $\steppow \in
  \openright{0}{1}$. For any $\npow \in (\frac{1}{\delmoment - 1}, 1)$, 
  with probability $1$ we have
  \begin{equation*}
    \limsup_n \sup_{z \in \R_+^n} \frac{\sum_{k=1}^n
      \norm{I - \errmat^{nk}} z_k}{
      \sum_{k = n - n^\npow}^n z_k}
    \le 1
    ~~ \mbox{and} ~~
    \sup_n \sup_{z \in \R^n_+} \frac{\sum_{k=1}^n \stepsize_k
      \norm{I - \errmat^{nk}} z_k}{\stepsize_n \sum_{k = n - n^\npow}^n
      z_k} < \infty,
  \end{equation*}
  where we treat $0 / 0 = 1$.
\end{lemma}
\begin{proof}
  The first statement of the lemma follows from
  Lemma~\ref{lemma:delays-stop}, as with probability $1$ over
  the delays $\delay_k$ and delay matrices $\errmat^{nk}$, there exists
  some (random) $N$ such that $n \ge N$ implies that
  $\errmat^{nk} = I$ for all $k \le n - n^\npow$, so that
  for $n \ge N$, we have for all nonnegative sequences $z_1, z_2, \ldots$
  that
  \begin{equation*}
    \frac{\sum_{k = 1}^n \norm{I - \errmat^{nk}} z_k}{
      \sum_{k = n - n^\npow}^n z_k}
    \le \frac{\sum_{k = n - n^\npow}^n z_k}{
      \sum_{k = n - n^\npow}^n z_k} = 1.
  \end{equation*}
  The second follows from the first once we note that
  for $k \in [n - n^\npow, n]$, we have
  \begin{equation*}
    1 \ge \frac{\stepsize_n}{\stepsize_k}
    \ge n^{-\steppow} (n - n^\npow)^\steppow
    = (1 - n^{\npow - 1})^\steppow
    = \exp(-\steppow n^{\npow - 1})(1 + o(1)) \to 1
  \end{equation*}
  as $n \to \infty$. The limit supremum is finite, so the supremum
  must likewise be finite.
\end{proof}

In particular, Lemma~\ref{lemma:crazy-ratios} implies that any sequence
$\stepsize_n \sum_{k=1}^n \stepsize_k \norm{I - \errmat^{nk}} Z_k$ cannot
diverge more quickly than $\stepsize_n^2 \sum_{k = n - n^\npow}^n Z_k$.  We
will use this fact frequently. For the remainder of the proof of
Theorem~\ref{theorem:nonlinear}, we define the random variable (implicitly
depending on the power $\npow$ chosen in the
interval~\eqref{eqn:delay-moment-interval})
\begin{equation}
  \label{eqn:ratio-def}
  \badratio_n \defeq \max_{m \le n}
  \sup_{z \in \R^n_+} \frac{\sum_{k = 1}^m \stepsize_k
    \norm{I - \errmat^{nk}} z_k}{\stepsize_m \sum_{k = m - m^\npow}^m
    z_k}
  ~~ \mbox{and} ~~
  \badratio_{\infty} \defeq \limsup_n \badratio_n.
\end{equation}
As $\badratio_n$ are non-decreasing, we have $\badratio_\infty = \lim_n
\badratio_n = \sup_n \badratio_n$, and by Lemma~\ref{lemma:crazy-ratios}, we
see that with probability $1$ over the delay process, we have
$\badratio_\infty < \infty$, and moreover, we have $\badratio_n \in
\mc{F}_{n-1}$ by definition~\eqref{eqn:filtration} of the $\sigma$-fields
$\mc{F}_k$.  For $t \in \R$ define
\begin{equation}
  \label{eqn:ratio-event-def}
  \badratioevent_{n,t} \defeq \{\badratio_n \le t\}
  ~~ \mbox{and} ~~
  \badratioevent_{\infty, t} \defeq \bigcap_n \badratioevent_{n,t}
  = \left\{\sup_n \badratio_n \le t \right\}
\end{equation}
be the events that $\badratio_n$ and $\badratio_\infty$ are bounded by $t$,
respectively, noting that $\badratioevent_{n,t} \in \mc{F}_{n-1}$ as
$\badratio_n \in \mc{F}_{n-1}$ as before. Then
Lemma~\ref{lemma:crazy-ratios}
implies
\begin{equation*}
  \lim_{t \to \infty} \P(\cap_{n \ge 1} \badratioevent_{n,t})
  = \lim_{t \to \infty} \P(\badratioevent_{\infty,t})
  = \lim_{t \to \infty} \P\left(\sup_n \badratio_n \le t\right) = 1.
\end{equation*}

Our first lemma, whose proof we provide in
Sec.~\ref{sec:proof-p1-convergence}, builds off of the Robbins-Siegmund
martingale convergence theorem (Lemma~\ref{lemma:robbins-siegmund})
to give an almost sure convergence result for the corrected
sequence $\wt{\error}_n$.
\begin{lemma}
  \label{lemma:p1-convergence}
  Let Assumptions~\ref{assumption:residual-quadratic}
  and~\ref{assumption:delay-moments} hold and the
  stepsizes $\stepsize_k = \stepsize k^{-\steppow}$.
  \begin{enumerate}[(a)]
  \item If Assumption~\ref{assumption:resid-strong-convex} holds,
    let $t < \infty$ and assume additionally that
    $\sup_n \E[\indic{\badratioevent_{n,t}} \norms{\error_n}^2] < \infty$.
    Then there is a finite random variable $\lyap_t$
    such that $\indic{\badratioevent_{n,t}}
    \lyap(\wt{\error}_n) \cas \lyap_t$
    and
    \begin{equation*}
      \sum_{n = 1}^\infty \stepsize_n \indic{\badratioevent_{n,t}}
      \<\nabla \lyap(\wt{\error}_n), \resid(\wt{x}_n)\> < \infty.
    \end{equation*}
  \item
    If Assumption~\ref{assumption:resid-lipschitz} holds, there is a finite
    random variable $\lyap$ such that $\lyap(\wt{\error}_n) \cas \lyap$, and
    \begin{equation*}
      \sum_{n = 1}^\infty \stepsize_n
      \<\nabla \lyap(\wt{\error}_n), \resid(\wt{x}_n)\>
      < \infty.
    \end{equation*}
  \end{enumerate}
\end{lemma}

We can now verify that $\wt{\error}_n \cas 0$ under the conditions of
Lemma~\ref{lemma:p1-convergence}. First, let
Assumption~\ref{assumption:resid-lipschitz} hold, and let $\epsilon > 0$ be
the radius for which $\<\nabla \lyap(x - x\opt), \resid(x)\> \ge \lambda_0
\lyap(x - x\opt)$ for $\norm{x - x\opt} \le \epsilon$.  With $c_0 \defeq
\inf_{x:\norm{x - x\opt} > \epsilon} \< \nabla \lyap(x - x\opt), \resid(x)\> >
0$, we have
\begin{equation*}
  \sum_{n = 1}^\infty 
  \stepsize_n \min\{\lambda_0 \lyap(\wt{\error}_n), c_0\}
  \le \sum_{n=1}^\infty \stepsize_n
  \<\nabla \lyap(\wt{\error}_n), \resid(\wt{x}_n)\>
  < \infty,
\end{equation*}
and $\lyap(\wt{\error}_n) \cas \lyap$ for some random variable $\lyap$.
If $\P(\lyap > 0) > 0$, there exist realizations of the randomness in the
problem such that $\lyap > 0$, and for such realizations there must
exist $\epsilon_0 > 0$ such that $\lyap(\wt{\error}_n) \ge \epsilon_0$
for all sufficiently large $n$ as $\lyap(\wt{\error}_n) \cas \lyap$;
this contradicts $\sum_{n=1}^\infty \stepsize_n = \infty$, so we must
have $\lyap = 0$ a.s.\ under Assumption~\ref{assumption:resid-lipschitz}.

Under the alternate Assumption~\ref{assumption:resid-strong-convex} and that
$\sup_n \E[\indic{\badratioevent_{n,t}} \norm{\error_n}^2] < \infty$ for all
$t$, we have
\begin{equation*}
  \sum_{n=1}^\infty \indic{\badratioevent_{n,t}}
  \stepsize_n \lambda_0 \lyap(\wt{\error}_n)
  \le \sum_{n=1}^\infty \indic{\badratioevent_{n,t}} \stepsize_n
  \<\nabla \lyap(\wt{\error}_n), \resid(\wt{x}_n)\>
  < \infty.
\end{equation*}
As $\sum_n \stepsize_n = \infty$, we must then have $\lyap(\wt{\error}_n)
\indic{\badratioevent_{n,t}} \cas 0$, as we know it converges to
something by Lemma~\ref{lemma:p1-convergence}.
Using that $\lim_{t \to \infty} \P(\badratioevent_{\infty, t})
= 1$ and $\badratioevent_{n,t} \supset \badratioevent_{n+1,t}$ for all $n$,
we find
\begin{equation*}
  \P(\lyap(\wt{\error}_n) \not \to 0)
  = \lim_{t \to \infty} \P(\badratioevent_{\infty,t}
  ~ \mbox{and} ~ \lyap(\wt{\error}_n) \not\to 0)
  \le \limsup_{t \to \infty} \P\left(\indic{\badratioevent_{n,t}}
  \lyap(\wt{\error}_n) \not\to 0\right) = 0
\end{equation*}
by the preceding discussion. In particular, we have
\begin{equation}
  \label{eqn:resid-cas-error}
  \lyap(\wt{\error}_n) \cas 0
\end{equation}
whenever the conditions of Lemma~\ref{lemma:p1-convergence} hold.

Now we show that the averages of $\wt{\error}_n$ and $\error_n$ are
asymptotically equivalent in distribution, and we have quantitative control
over this equivalence.  (See Section~\ref{sec:proof-wrong-is-right} for a
proof of this lemma.)
\begin{lemma}
  \label{lemma:wrong-is-right}
  In addition to the conditions of Lemma~\ref{lemma:p1-convergence}, assume
  that $\steppow > \frac{1}{\tau - 1} + \half$ and either (a)
  Assumption~\ref{assumption:resid-strong-convex} holds and the sequence
  $C_{\error,n,t}^2 = \max_{k \le n} \E[\indic{\badratioevent_{k,t}}
    \norm{\error_k}^2]$ satisfies $\sup_n C_{\error,n,t} < \infty$ for all
  $t \in \R$ or (b) Assumption~\ref{assumption:resid-lipschitz} holds.  In
  case (a), there
  exists a universal constant $C$ such that
  \begin{equation}
    \label{eqn:expectation-single-step-bound}
    \E\left[\indic{\badratioevent_{n,t}} \norms{\wt{\error}_n - \error_n}^2
      \right]
    \le C \, \stepsize_n^2 n^{2 \npow} t^2 (C_{\error,n-1,t}^2 + 1).
  \end{equation}
  Additionally,
  \begin{equation}
    \label{eqn:wrong-is-right}
    \sqrt{n}\left(\wb{\wt{\error}}_n - \wb{\error}_n\right)
    = \frac{1}{\sqrt{n}}
    \sum_{k=1}^n (\wt{\error}_k - \error_k)
    \cp 0,
    ~~~~
    \frac{1}{\sqrt{n}} \sum_{k=1}^n \norms{\wt{\error}_k - \error_k}
    \cp 0,
  \end{equation}
  and
  \begin{equation}
    \label{eqn:resid-cas-error-true}
    \lyap(\error_n) \cas 0.
  \end{equation}
\end{lemma}

Thus, any distributional convergence
results we are able to show on $n^{-\half} \sum_{k=1}^n \wt{\error}_k$ will
also hold for the uncorrected sequence $n^{-\half} \sum_{k=1}^n \error_k$ as
long as the conditions of Lemma~\ref{lemma:wrong-is-right} hold.  With this
in mind, we give an additional equivalence result showing that
$\wt{\error}_k$ is equivalent to an easier to analyze sequence of errors
generated from a simpler matrix iteration. Let the noise sequence
$\{\noise_k\}$ be generated as in the
iterations~\eqref{item:grad-comp-nl}--\eqref{item:update-step-nl}, and
consider the two iterations
\begin{equation}
  \label{eqn:corrected-iterations}
  \wt{\error}_{k + 1}
  = \wt{\error}_k - \stepsize_k \left(\resid(x_k) + \noise_k \right)
  ~~~ \mbox{and} ~~~
  \wt{\error}_{k+1}'
  = (I - \stepsize_k \hess) \wt{\error}_k'
  - \stepsize_k \noise_k.
\end{equation}
We have the following two results, which show (under slightly different
conditions) that the differences between the
iterations~\eqref{eqn:corrected-iterations} tends to zero.  The proofs of
the lemmas are quite similar, so we put material relevant for the proof of
both in Section~\ref{sec:proof-corrected-errors-matrices}, specializing
to each of the lemmas in sections~\ref{sec:proof-corrected-errors-matrices-a}
and~\ref{sec:proof-corrected-errors-matrices-b}, respectively.

\begin{lemma}
  \label{lemma:corrected-errors-matrices}
  Let Assumption~\ref{assumption:residual-quadratic} hold with some $\gamma
  \in \openleft{0}{1}$. Let the stepsizes $\stepsize_k = \stepsize
  k^{-\steppow}$, where $\frac{1}{1 + \gamma} + \frac{1}{\delmoment - 1} <
  \steppow < 1$. Let Assumption~\ref{assumption:resid-strong-convex} hold
  and assume that for each $t \in \R$ there is a constant $C_{\error,t} <
  \infty$ such that $\E[\indic{\badratioevent_{n,t}}\ltwo{\error_n}^2] \le
  C_{\error,t}^2$ for all $n$. Then
  \begin{equation*}
    n^{-\half} \sum_{k=1}^n (\wt{\error}_k - \wt{\error}_k')
    \cp 0.
  \end{equation*}
\end{lemma}
\begin{lemma}
  \label{lemma:corrected-errors-matrices-lipschitz}
  Let Assumption~\ref{assumption:residual-quadratic} hold with some $\gamma
  \in \openleft{0}{1}$. Let the stepsizes $\stepsize_k = \stepsize
  k^{-\steppow}$, where $\frac{1}{1 + \gamma} + \frac{1}{\delmoment - 1} <
  \steppow < 1$. Let Assumption~\ref{assumption:resid-lipschitz} hold. Then
  \begin{equation*}
    n^{-\half} \sum_{k=1}^n (\wt{\error}_k - \wt{\error}_k')
    \cp 0.
  \end{equation*}
\end{lemma}

With the preceding lemmas in place, we require two additional claims that
immediately yield our desired convergence guarantee. The first is the
asymptotic normality of the matrix product sequence $\wt{\error}_k'$, that
is, that $n^{-\half} \sum_{k=1}^n \wt{\error}'_k$ is asymptotically
normal. The second is that under
Assumption~\ref{assumption:resid-strong-convex}, we have $\sup_k
\E[\norm{\error_k}^2 \indic{\badratioevent_{k,t}}] < \infty$ for all $t \in
\R$.  We begin with the first result.
\begin{lemma}[Polyak and Juditsky~\cite{PolyakJu92}, Theorem 1]
  \label{lemma:classical-matrix-convergence}
  Let $\{\noise_k\}$ be a martingale difference sequence adapted to the
  filtration $\mc{F}_k$, so that $\E[\noise_k \mid \mc{F}_{k-1}] = 0$ and
  $\sup_k \E[\norm{\noise_k}^2 \mid \mc{F}_{k-1}] < \infty$ with probability
  1. Assume additionally that
  \begin{equation*}
    \lim_{c \to \infty} \limsup_{k \to \infty}
    \E\left[\norms{\noise_k}^2 \indic{\norms{\noise_k} > c}
      \mid \mc{F}_{k-1}\right] = 0 ~ \mbox{in~probability}
  \end{equation*}
  and that as $k \to \infty$ we have
  \begin{equation*}
    \cov(\noise_k \mid \mc{F}_{k-1}) \cp \Sigma \succ 0,
  \end{equation*}
  where $\cov(\noise \mid \mc{F}) = \E[\noise \noise^\top \mid \mc{F}]$.
  Then if $\hess \succ 0$, the iteration
  \begin{equation*}
    \wt{\error}_{k+1}' = (I - \stepsize_k \hess) \wt{\error}_k'
    - \stepsize_k \noise_k
  \end{equation*}
  with $\stepsize_k = \stepsize k^{-\steppow}$ satisfies
  \begin{equation*}
    \frac{1}{\sqrt{n}} \sum_{k = 1}^n \wt{\error}_k'
    \cd
    \normal\left(0, \hess^{-1} \Sigma \hess^{-1}\right).
  \end{equation*}
\end{lemma}

We verify that the conditions of
Lemma~\ref{lemma:classical-matrix-convergence} hold in the two settings
captured by Theorem~\ref{theorem:nonlinear}, that is, under
Assumption~\ref{assumption:resid-strong-convex}
and~\ref{assumption:resid-lipschitz}.  For
Assumption~\ref{assumption:resid-lipschitz}, we immediately have each
condition on $\noise_k$ except that $\cov(\noise_k \mid \mc{F}_{k-1}) \cp
\Sigma$ for some positive definite matrix $\Sigma$. But we have $\error_k
\cas 0$ as in expression~\eqref{eqn:resid-cas-error-true}, so
that by Assumption~\ref{assumption:noise-fun}
\begin{equation*}
  \cov(\noise_k \mid \mc{F}_{k-1})
  = \E[\noise_k(0) \noise_k(0)^\top
    + \noisier(x_k) \noisier(x_k)^\top \mid \mc{F}_{k-1}]
  = \Sigma + o_P(1) + O_P(\norm{x_k - x\opt}^2)
  \cp \Sigma.
\end{equation*}
Thus, the conditions of Lemma~\ref{lemma:classical-matrix-convergence} hold
under Assumption~\ref{assumption:resid-lipschitz}.  We now argue that the
conditions of the lemma hold under
Assumption~\ref{assumption:resid-strong-convex} and the additional condition
that $\sup_k \E[\norm{\error_k}^2 \indic{\badratioevent_{k,t}}] < \infty$
for all $t \in \R$.  In this case, we still know that $\E[\noise_k \mid
  \mc{F}_{k-1}] = 0$, and by Assumption~\ref{assumption:noise-fun}, we have
\begin{align*}
  \sup_k \E\left[\norms{\noise_k}^2 \mid \mc{F}_{k-1}\right]
  & \le 2 \sup_k \E\left[\norms{\noise_k(0)}^2 +
    \norms{\noisier(x_k)}^2 \mid \mc{F}_{k-1}\right] \\
  & \le 2 \sup_k \E[\norms{\noise_k(0)}^2
    \mid \mc{F}_{k-1}] + C \sup_k \norm{\error_k}^2
  < \infty,
\end{align*}
as $\norm{\error_k} \cas 0$ by the
result~\eqref{eqn:resid-cas-error-true}. The second condition on the limits
of $\E[\norms{\noise_k}^2 \indic{\norms{\noise_k} > c} \mid \mc{F}_{k-1}]$
holds similarly, as we have $\lim_{c \to \infty} \P(\sup_k \norm{\error_k} >
c) = 0$ because $\norm{\error_k} \cas 0$, again using the
guarantee~\eqref{eqn:resid-cas-error-true} if $\sup_k
\E[\indic{\badratioevent_{k,t}} \norms{\error_k}^2] < \infty$. The
covariance condition follows identically to the argument under
Assumption~\ref{assumption:resid-lipschitz}.  Summarizing, we have shown the
following result.
\begin{lemma}
  \label{lemma:corrected-sequence-normal}
  Let either (i) Assumption~\ref{assumption:resid-strong-convex} hold and
  assume that $\sup_k \E[\norms{\error}_k^2 \indic{\badratioevent_{k,t}}] <
  \infty$ for each $t$ or (ii) Assumption~\ref{assumption:resid-lipschitz}
  hold. Then the sequence $\wt{\error}'_k$ defined by the
  iteration~\eqref{eqn:corrected-iterations} satisfies
  \begin{equation*}
    \frac{1}{\sqrt{n}} \sum_{k = 1}^n \wt{\error}_k'
    \cd
    \normal\left(0, \hess^{-1} \Sigma \hess^{-1}\right).
  \end{equation*}
\end{lemma}

By Lemmas~\ref{lemma:corrected-errors-matrices}
and~\ref{lemma:corrected-errors-matrices-lipschitz} and the probabilistic
equivalence~\eqref{eqn:wrong-is-right} of the sequences
$\sqrt{n}\wb{\wt{\error}}_n$ and $\sqrt{n}\wb{\error}_n$,
Lemma~\ref{lemma:corrected-sequence-normal} implies that (under the
conditions of the lemma)
\begin{equation*}
  \frac{1}{\sqrt{n}} \sum_{k = 1}^n \error_k
  \cd
  \normal\left(0, \hess^{-1} \Sigma \hess^{-1}\right).
\end{equation*}
This is the statement of the theorem under
Assumption~\ref{assumption:resid-lipschitz}; the proof of the theorem
will be complete if we can show that under the strong convexity
Assumption~\ref{assumption:resid-strong-convex} we have $\sup_k
\E[\indic{\badratioevent_{k,t}}\norm{\error_k}^2] < \infty$
for all $t \in \R$.

We present a final lemma that gives the boundedness of the sequences
$\E[\indic{\badratioevent_{k,t}}\norm{\error_k}^2]$.
\begin{lemma}
  \label{lemma:smaller-than-past}
  Let Assumption~\ref{assumption:resid-strong-convex} hold and $\epsilon >
  0$, $t \in \R$. There exists some $N = N(\epsilon,t) \in \N$ such that $n
  \ge N$ implies
  \begin{equation*}
    \E[\norms{\wt{\error}_{n+1}}^2 \indic{\badratioevent_{n+1,t}}]
    \le \epsilon^2
    \max_{k \le n} \max\{1, \E[\indic{\badratioevent_{k,t}}
      \norms{\error_k}^2]\}.
  \end{equation*}
\end{lemma}

Now,
Lemma~\ref{lemma:wrong-is-right},
inequality~\eqref{eqn:expectation-single-step-bound}, and the fact that the
stepsize power $\steppow > \npow$ (recall the
interval~\eqref{eqn:delay-moment-interval}) immediately yields that for any
$\epsilon > 0$ and $t \in \R$ there is some $N = N(\epsilon, t) \in \N$ such
that $n \ge N$ implies
\begin{equation*}
  \E\left[\indic{\badratioevent_{n,t}} \norms{\wt{\error}_n - \error_n}^2
    \right]
  \le \epsilon^2 \max_{k < n} \max\left\{1,
  \E\left[\indic{\badratioevent_{k,t}} \norms{\error_k}^2\right]\right\}.
\end{equation*}
We combine these inequalities and Lemma~\ref{lemma:smaller-than-past} to
argue that under Assumption~\ref{assumption:resid-strong-convex}, we have
$\sup_k C_{\error,k,t}^2 = \sup_k \E[\indic{\badratioevent_{k,t}}
  \norm{\error_k}^2] < \infty$. Indeed, choose $\epsilon > 0, t \in \R$ and
$N = N(\epsilon, t)$ such that $n \ge N$ implies
\begin{equation*}
  \E\left[\norms{\wt{\error}_n - \error_n}^2\indic{\badratioevent_{n,t}}\right]
  \le \frac{\epsilon^2}{4} \max_{k < n}
  \max\left\{1, \E\left[\indic{\badratioevent_{k,t}}\norms{\error_k}^2\right]
  \right\}
\end{equation*}
and
\begin{equation*}
  \E\left[\indic{\badratioevent_{n,t}} \norms{\wt{\error}_n}^2\right]
  \le \frac{\epsilon^2}{4}
  \max_{k < n} \max\left\{1,
  \E\left[\indic{\badratioevent_{k,t}}\norms{\error_k}^2\right]\right\}.
\end{equation*}
Then
\begin{align*}
  \E\left[\indic{\badratioevent_{n,t}} \norm{\error_n}^2 \right]
  & \le 2 \E\left[\indic{\badratioevent_{n,t}}
    \norms{\error_n - \wt{\error}_n}^2\right]
  + 2 \E\left[\indic{\badratioevent_{n,t}} \norms{\wt{\error}_n}^2\right] \\
  & \le \epsilon^2
  \max_{k < n} \max\left\{1, \E\left[\indic{\badratioevent_{n,t}}
    \norms{\error_k}^2\right]\right\}.
\end{align*}
Repeating this chain of inequalities for all $k$ such that
$N \le k < n$, we find that
\begin{equation*}
  \E\left[\indic{\badratioevent_{n,t}} \norm{\error_n}^2 \right]
  \le \epsilon^2 \max_{k \le N} \max\left\{1, 
  \E\left[\indic{\badratioevent_{k,t}} \norms{\error_k}^2 \right]\right\}
\end{equation*}
for all $n > N$. As $\max_{k \le N} \E[\norm{\error_k}^2] < \infty$ for any
finite $N$, this shows that there exists a constant $C_{\error,t} < \infty$
such that $\sup_k \E[\indic{\badratioevent_{k,t}} \norm{\error_k}^2] \le
C_\error^2$ whenever Assumption~\ref{assumption:resid-strong-convex} holds
and the stepsizes $\stepsize_k$ are chosen such that $\stepsize_k =
\stepsize k^{-\steppow}$ for $\steppow \in (\half, 1)$ and satisfying $\npow
< \steppow - \half$.  This completes the proof of
Theorem~\ref{theorem:nonlinear}.

\section{Discussion and conclusions}

In this paper, we have analyzed an asynchronous gradient method, based on
\citeauthor{NiuReReWrNi11}'s \hogwild~\cite{NiuReReWrNi11}, for the solution
of stochastic convex optimization and variational equality problems.  Our
work shows particularly that asynchrony introduces essentially negligible
penalty for stochastic optimization problems under standard optimization
assumptions, which can be leveraged in the development of extremely fast
optimization procedures.  Our experimental results in
Section~\ref{sec:experiments} show that there is still work to be done in
terms of a deep understanding of implementation of these methods. In
particular, even without inherent competition for locks or other
synchronization resources in the computer, there can be competition for
other resources, such as memory access. As Table~\ref{table:cache-fun}
demonstrates, even moderately careful control of memory accesses can be
extremely beneficial, and without it, asynchronous methods do not enjoy the
performance benefits made possible by multi-core and multi-processor
systems. It will thus be beneficial, in future work we hope to undertake, to
develop an understanding of memory access and use similar to that now
well-known in the scientific computing literature (see, for example,
\citet{BallardDeHoSc11}). This understanding will greatly improve the
practical effectiveness of stochastic and asynchronous methods.

\appendix


\section{Technical proofs for Theorem~\ref{theorem:nonlinear}}

In this appendix, we collect the technical proofs required for
Theorem~\ref{theorem:nonlinear}. We also state a few additional technical
lemmas.

\begin{lemma}
  \label{lemma:errors-not-so-bad}
  Define $C_{\error,n,t} \defeq \max_{k \le n} \E[\indic{\badratioevent_{n,t}}
    \norm{\error_k}^2]$.
  If Assumption~\ref{assumption:resid-strong-convex} holds,
  there is a constant $c < \infty$ independent of $n$ and $C_{\error,n,t}$ such
  that for any $l \in \R_+$,
  \begin{equation*}
    \E\left[\badratio_n^l \indic{\badratioevent_{n,t}}
      \sum_{k = n - n^\npow}^{n - 1}
      \norm{\resid(x_k) + \noise_k}^2 \right]
    \le c \cdot t^l n^\npow (C_{\error,n-1,t}^2 + 1).
  \end{equation*}
  If Assumption~\ref{assumption:resid-lipschitz} holds,
  there is a constant $c < \infty$ independent of $n$ such that
  \begin{equation*}
    \E\left[\sum_{k = n - n^\npow}^{n - 1}
      \norm{\resid(x_k) + \noise_k}^2 \right]
    \le c n^\npow.
  \end{equation*}
\end{lemma}
\begin{proof}
  We begin with the result under
  Assumption~\ref{assumption:resid-strong-convex}, where we use
  $\E[\norms{\error_k}^2 \indic{\badratioevent_{n,t}}] \le
  C_{\error,n-1,t}^2$ for $k \le n-1$. We also know that
  \begin{align*}
    \norm{\resid(x_k)}^2 \lesssim \norm{\error_k}^2
    ~~~ \mbox{and} ~~~
    \E[\norms{\noise_k}^2 \mid \mc{F}_{k-1}]
    \lesssim \norm{\error_k}^2 + 1,
  \end{align*}
  where we have used Assumption~\ref{assumption:noise-fun}, and thus we have
  $\E[\E[\indic{\badratioevent_{k,t}} \norms{\noise_k}^2 \mid
      \mc{F}_{k-1}]] \le c (C_{\error,n-1,t}^2 + 1)$.  In addition, we have
  $\norm{\resid(x)} = \norm{\resid(x) - \resid(x\opt)} \le L\norm{x -
    x\opt}$ and $\badratio_n^l \indic{\badratioevent_{n,t}} \le t^l
  \indic{\badratioevent_{n,t}}$, so
  \begin{align*}
    \E\left[
      \badratio_n \indic{\badratioevent_{n,t}} \sum_{k = n - n^\npow}^{n - 1}
      \norms{\resid(x_k) + \noise_k}^2\right]
    & \le 
    2 t^l \sum_{k = n - n^\npow}^{n - 1}
    \E\left[\indic{\badratioevent_{n,t}}
      \norm{\resid(x_k)}^2 + \indic{\badratioevent_{n,t}}
      \norms{\noise_k}^2\right] \\
    & \le 2 C t^l n^\npow
    \left(L C_{\error,n-1,t}^2 + c (C_{\error,n-1,t}^2 + 1)\right),
  \end{align*}
  where we have used that $\indic{\badratioevent_{n,t}} \le
  \indic{\badratioevent_{n-1,t}}$ for all $n$.
  Under Assumption~\ref{assumption:resid-lipschitz}, the result is simpler:
  we have simply that $\E[\norms{\noise_k}^2 \mid \mc{F}_{k-1}]
  \le C$ and $\norm{\resid(x_k)} \le C$, giving the result.
\end{proof}

We also give two technical results involving integral convergence.
\begin{lemma}
  \label{lemma:funny-gamma-integral}
  Let $c > 0$ and $\kappa \in (0, 1)$ be constants and $b \ge a > 0$. Then
  \begin{equation*}
    \int_a^b \exp\left(-c (t^\kappa - a^\kappa)\right) dt
    \le \frac{\max\{2^{\frac{1 - \kappa}{\kappa} - 1}, 1\}}{
      \kappa c}
    \left[c^\frac{\kappa - 1}{\kappa}
      \Gamma\Big(\frac{1}{\kappa}\Big)
      + a^{1 - \kappa}\right].
  \end{equation*}
\end{lemma}
See Appendix~\ref{sec:proof-funny-gamma-integral} for a proof of
Lemma~\ref{lemma:funny-gamma-integral}.  The final technical result we use
also gives a bound on (essentially) another gamma integral.
\begin{lemma}
  \label{lemma:funny-zero-sum}
  Let $\steppow \in (\half, 1)$ and $\npow < \steppow - \half$.
  Then
  \begin{equation*}
    \lim_{n \to \infty}\,
    \sum_{k=1}^n k^{\npow - 2 \steppow}
    \exp\left(-c (n^{1 - \steppow} - k^{1 - \steppow})\right)
    = 0.
  \end{equation*}
\end{lemma}
\noindent
See Appendix~\ref{sec:proof-funny-zero-sum} for a proof of
Lemma~\ref{lemma:funny-zero-sum}.

\subsection{Proof of Lemma~\ref{lemma:p1-convergence}}
\label{sec:proof-p1-convergence}

We begin by using the Lipschitz continuity of the gradients of $\lyap$ to
note that
\begin{align*}
  \lefteqn{\lyap(\wt{\error}_{n + 1})
    = \lyap(\wt{\error}_n - \stepsize_n g_n)
    \le \lyap(\wt{\error}_n) - \stepsize_n \<\nabla \lyap(\wt{\error}_n), g_n\>
    + \frac{L \stepsize_n^2}{2} \norm{g_n}^2} \\
  & = \lyap(\wt{\error}_n) - \stepsize_n\<\nabla \lyap(\wt{\error}_n),
  \resid(\wt{x}_n)\>
  - \alpha_n \<\nabla \lyap(\wt{\error}_n),
  \resid(x_n) - \resid(\wt{x}_n)\>
  - \alpha_n \<\nabla \lyap(\wt{\error}_n), \noise_n\>
  + \frac{L \stepsize_n^2}{2} \norm{g_n}^2.
\end{align*}
Taking expectations conditional on $\mc{F}_{n-1}$, we have
$x_n, \wt{x}_n \in \mc{F}_{n-1}$, and $\E[\noise_n \mid \mc{F}_{n-1}] = 0$.
Moreover, we have $g_n = \resid(x_n) + \noise_n$, and thus
\begin{equation}
  \begin{split}
    \E[\lyap(\wt{\error}_{n+1}) \mid \mc{F}_{n-1}]
    & \le \lyap(\wt{\error}_n)
    - \stepsize_n \<\nabla \lyap(\wt{\error}_n), \resid(\wt{x}_n)\>
    + \frac{L \stepsize_n^2}{2} \E[\norm{\resid(x_n) + \noise_n}^2
      \mid \mc{F}_{n-1}] \\
    & \qquad + \alpha_n \norms{\nabla \lyap(\wt{\error}_n)}
    \norm{\resid(x_n) - \resid(\wt{x}_n)}.
  \end{split}
  \label{eqn:nonlinear-one-step}
\end{equation}
Using that
$\norm{\resid(x_n) + \noise_n}^2
\le 2 \norm{\resid(x_n)}^2 + 2\norm{\noise_n}^2$
and
\begin{align*}
  \norm{\resid(x_n) - \resid(\wt{x}_n)}
  & \le L \norm{x_n - \wt{x}_n}
  = L \normbigg{\sum_{k = 1}^{n-1} \stepsize_k (I - \errmat^{nk}) g_k}
  \le L \sum_{k = 1}^{n-1} \stepsize_k \norms{I - \errmat^{nk}}\norm{g_k},
\end{align*}
we find that there is a constant $C$ such that
\begin{align*}
  \E[\lyap(\wt{\error}_{n+1}) \mid \mc{F}_{n-1}] \nonumber
  & \le \lyap(\wt{\error}_n)
  - \stepsize_n \<\nabla \lyap(\wt{\error}_n), \resid(\wt{x}_n)\>
  + C \stepsize_n^2 \norm{\resid(x_n)}^2
  + C \stepsize_n^2 \E[\norm{\noise_n}^2 \mid \mc{F}_{n-1}] \\
  & \qquad + C \stepsize_n \norms{\nabla \lyap(\wt{\error}_n)}
  \sum_{k = 1}^{n - 1} \stepsize_k \norms{I - \errmat^{nk}} \norm{g_k}.
\end{align*}
Recalling the definition~\eqref{eqn:ratio-def} of the random variables
$\badratio_n$ and the selection~\eqref{eqn:delay-moment-interval} of the
power $\npow$, we have $\sum_{k=1}^{n-1} \stepsize_k \norms{I -
  \errmat^{nk}} \norm{g_k} \le \badratio_n \stepsize_n \sum_{k = n -
  n^\npow}^{n - 1} \norm{g_k}$, and using that $g_k = \resid(x_k) +
\noise_k$, we obtain
\begin{subequations}
  \begin{align}
    \E[\lyap(\wt{\error}_{n+1}) \mid \mc{F}_{n-1}] \nonumber
    & \le \lyap(\wt{\error}_n)
    - \stepsize_n \<\nabla \lyap(\wt{\error}_n), \resid(\wt{x}_n)\>
    + C \stepsize_n^2 \norm{\resid(x_n)}^2
    + C \stepsize_n^2 \E[\norm{\noise_n}^2 \mid \mc{F}_{n-1}] \\
    & \qquad + C \stepsize_n^2 \norms{\nabla \lyap(\wt{\error}_n)}
    \badratio_n \sum_{k = n - n^\npow}^{n - 1} \norm{\resid(x_k) + \noise_k}
    \label{eqn:single-step-alternate-lyap} \\
    & \le \lyap(\wt{\error}_n)
    - \stepsize_n \<\nabla \lyap(\wt{\error}_n), \resid(\wt{x}_n)\>
    + C \stepsize_n^2 \norm{\error_n}^2
    + C \stepsize_n^2 \E[\norm{\noise_n}^2 \mid \mc{F}_{n-1}] \nonumber \\
    & \qquad + C \stepsize_n^2 n^\npow \lyap(\wt{\error}_n)
    + C \stepsize_n^2 \badratio_n^2 \sum_{k = n - n^\npow}^{n-1}
    \norm{\resid(x_k) + \noise_k}^2,
    \label{eqn:single-step-alternate-lyap-v2}
  \end{align}
\end{subequations}
the final equality following because $\norm{\resid(x)} = \norm{\resid(x) -
  \resid(x\opt)} \le L \norm{x - x\opt}$, $\norm{\nabla \lyap(\error)} \le L
\norm{\error} \le (L / \sqrt{\lambda}) \sqrt{\lyap(\error)}$, and $ab \le
\half a^2 + \half b^2$ for any $a, b \in \R$.

We now use the technical Lemma~\ref{lemma:errors-not-so-bad}, which allows
us to control the error terms in
inequality~\eqref{eqn:single-step-alternate-lyap-v2}.  Indeed, by
Lemma~\ref{lemma:errors-not-so-bad} and our assumption that $C_{\error,t}^2
= \sup_k \E[\indic{\badratioevent_{k,t}} \norms{\error_k}^2] < \infty$ if
Assumption~\ref{assumption:resid-strong-convex} holds, we obtain
\begin{equation}
  \sum_{n = 1}^\infty \stepsize_n^2 \E\left[
    \badratio_n^2 \indic{\badratioevent_{n,t}} \sum_{k = n - n^\npow}^{n - 1}
    \norms{\resid(x_k) + \noise_k}^2\right]
  \le c t^2 (C_{\error,t}^2 + 1) \sum_{n = 1}^\infty \stepsize_n^2 n^\npow
  \lesssim \int_1^\infty u^{-2 \steppow + \npow} du
  < \infty,
  \label{eqn:sum-bad-ratios}
\end{equation}
the final inequality holding when $2 \steppow - \npow > 1$, or $\npow < 2
\steppow - 1$.  In particular, the Robbins-Siegmund convergence theorem
(Lemma~\ref{lemma:robbins-siegmund}) applies, as we can write (recall
inequality~\eqref{eqn:single-step-alternate-lyap} and that
$\indic{\badratioevent_{n,t}} \le \indic{\badratioevent_{n-1,t}}$)
\begin{equation*}
  \E[\indic{\badratioevent_{n,t}} \lyap(\wt{\error}_{n+1}) \mid \mc{F}_{n-1}]
  \le (1 + \beta_{n-1}) \indic{\badratioevent_{n-1,t}} \lyap(\wt{\error}_n)
  + \kappa_{n-1} - \varepsilon_{n-1},
\end{equation*}
where $\beta_{n-1} = C \alpha_n^2 n^\npow$, $\varepsilon_n = 
\indic{\badratioevent_{n,t}} \alpha_n
\<\nabla \lyap(\wt{\error}_n), \resid(\wt{x}_n)\>$, and
\begin{equation*}
  \kappa_{n-1} = \indic{\badratioevent_{n,t}}
  \left[C \alpha_n^2 \norm{\resid(x_n)}^2
    + C \alpha_n^2 \E[\norms{\noise_n}^2 \mid \mc{F}_{n-1}]
    + C \alpha_n^2 \badratio_n^2
    \sum_{k = n - n^\npow}^{n - 1} \norm{\resid(x_k)
      + \noise_k}^2\right]
\end{equation*}
are all $\mc{F}_{n-1}$-measurable.  Moreover, $\sum_n \beta_n \lesssim
\sum_{n=1}^\infty n^{\npow -2 \steppow} < \infty$ because $\npow < 2
\steppow - 1$, and $\sum_n \E[\kappa_n] < \infty$ by the fact that
$\E[\E[\norm{\noise_n}^2 \mid \mc{F}_{n-1}] \lesssim \E[\norm{\error_n}^2 +
    1]$ (this is Assumption~\ref{assumption:noise-fun}) coupled with
  Lemma~\ref{lemma:errors-not-so-bad} and
  inequality~\eqref{eqn:sum-bad-ratios}.  We thus conclude that
\begin{equation*}
  \indic{\badratioevent_{n,t}} \lyap(\wt{\error}_n) \cas \lyap_t
  ~~~ \mbox{and} ~~~
  \sum_{n = 1}^\infty \stepsize_n \indic{\badratioevent_{n,t}}
  \<\nabla \lyap(\wt{\error}_n),
  \resid(\wt{x}_n)\> < \infty
\end{equation*}
with probability $1$ whenever Assumption~\ref{assumption:resid-strong-convex}
holds in addition to the assumptions of the lemma.

In the somewhat simpler case that
Assumption~\ref{assumption:resid-lipschitz} holds, we may simply remove all
indicator functions $\indic{\badratioevent_{n,t}}$, as
Lemma~\ref{lemma:errors-not-so-bad} shows
that we may replace inequality~\eqref{eqn:sum-bad-ratios}
with
\begin{equation*}
  \sum_{n = 1}^\infty \stepsize_n^2 \sum_{k = n - n^\npow}^{n - 1}
  \E[\norm{\resid(x_k) - \noise_k}^2]
  \lesssim C \sum_{n = 1}^\infty \stepsize_n^2 n^\npow
  \lesssim \int_1^\infty u^{-2 \steppow + \npow} du < \infty,
\end{equation*}
while $\sup_n \badratio_n < \infty$ with probability $1$.

\subsection{Proof of Lemma~\ref{lemma:wrong-is-right}}
\label{sec:proof-wrong-is-right}

We may write the difference
\begin{equation*}
  \wt{\error}_n - \error_n =
  \wt{x}_n - x_n
  = \sum_{k = 1}^{n-1} \stepsize_k (\errmat^{nk} - I) g_k.
\end{equation*}
Recalling the definition~\eqref{eqn:ratio-def} of $\badratio_n$
and our choice~\eqref{eqn:delay-moment-interval} of $\npow$,
this representation guarantees that
\begin{equation}
  \label{eqn:one-term-error}
  \norms{\wt{\error}_n - \error_n}
  \le \badratio_n \stepsize_n \sum_{k = n - n^\npow}^{n-1} \norm{g_k},
\end{equation}
and so we have
\begin{equation}
  \label{eqn:summed-ltwo-distances-ratioed}
  \frac{1}{\sqrt{n}} \sum_{k=1}^n \norms{\wt{\error}_k - \error_k}
  \le \frac{\badratio_n}{\sqrt{n}} \sum_{k = 1}^n
  \stepsize_k \bigg(\sum_{i = k - k^\npow}^{k-1} \norm{g_i}
  \bigg).
\end{equation}

Let us first show that the
quantity~\eqref{eqn:summed-ltwo-distances-ratioed} is well-behaved in the
simpler case of Assumption~\ref{assumption:resid-lipschitz}. Indeed, we have
\begin{equation*}
  \E\left[\frac{1}{\sqrt{n}} \sum_{k = 1}^n \stepsize_k
    \sum_{i = k - k^\npow}^{k-1} \norm{g_i}\right]
  \lesssim n^{-\half} \sum_{k=1}^n \stepsize_k k^\npow
  \lesssim n^{-\half} \int_1^n u^{\npow - \steppow} du
  \asymp n^{\half + \npow - \steppow},
\end{equation*}
which tends to zero if and only if $\steppow > \npow + \half$.  Then
inequality~\eqref{eqn:summed-ltwo-distances-ratioed} implies $n^{-\half}
\sum_{k=1}^n \norms{\wt{\error}_n - \error_n} \le \badratio_n Z_n$, where
$Z_n \clp{1} 0$ and $\sup_n \badratio_n < \infty$ with probability $1$ by
Lemma~\ref{lemma:crazy-ratios}; thus, the
convergence~\eqref{eqn:wrong-is-right} holds under
Assumption~\ref{assumption:resid-lipschitz}.

We turn to the somewhat more challenging case that
Assumption~\ref{assumption:resid-strong-convex} holds and that for our
choice of bound $t$ on $\badratio_n$, there exist constants
$C_{\error,n,t}^2 = \max_{k \le n} \E[\indic{\badratioevent_{k,t}}
  \norm{\error_k}^2]$ such that $C_{\error,t} = \sup_n C_{\error,n,t} <
\infty$.  In this case, inequality~\eqref{eqn:summed-ltwo-distances-ratioed}
and the definition~\eqref{eqn:ratio-event-def} of the event
$\badratioevent_{n,t} = \{ \badratio_{n,t} \le t\}$ imply
\begin{align*}
  \E\bigg[\indic{\badratioevent_{n,t}}
    \frac{1}{\sqrt{n}} \sum_{k=1}^n \norms{\wt{\error}_k - \error_k}
    \bigg]
  & \le \frac{t}{\sqrt{n}} \sum_{k=1}^n \stepsize_k
  \E\bigg[\indic{\badratioevent_{n,t}} \sum_{i = k - k^\npow}^{k-1}
    \norm{g_i}\bigg]
\end{align*}
Now, we note that for $k \le n$, we have $\badratioevent_{k,t} \supset
\badratioevent_{n,t}$ so that $\indic{\badratioevent_{n,t}} \le
\indic{\badratioevent_{k,t}}$, and
\begin{align}
  \E[\norm{g_k}^2 \indic{\badratioevent_{n,t}}]
  \le \E[\norm{g_k}^2 \indic{\badratioevent_{k,t}}]
  & \le 2 \E[\norm{\resid(x_k)}^2 \indic{\badratioevent_{k,t}}]
  + 2 \E[ \norm{\noise_k}^2 \indic{\badratioevent_{k,t}}] \nonumber \\
  & \lesssim C_{\error,n-1,t}^2 + C_{\error,n-1,t}^2 + 1
  \label{eqn:gradients-undelayed-small}
\end{align}
by Assumption~\ref{assumption:noise-fun} and that $\badratioevent_{k,t} \in
\mc{F}_{k-1}$. Thus we have $\E[\norm{g_k}^2 \indic{\badratioevent_{n,t}}]
\lesssim C_{\error,n-1,t} + 1$, and by Jensen's inequality and
inequality~\eqref{eqn:one-term-error}, we have
\begin{align*}
  \E\left[\indic{\badratioevent_{n,t}} \norms{\wt{\error}_n - \error_n}^2
    \right]
  & \le t^2 \stepsize_n^2 n^\npow \sum_{k = n - n^\npow}^{n-1}
  \E\left[\indic{\badratioevent_{n,t}} \norm{g_k}^2\right]
  \le C \, t^2 \stepsize_n^2 n^{2 \npow} (C_{\error,n-1,t}^2 + 1),
\end{align*}
where $C$ is some universal constant, by the
bound~\eqref{eqn:gradients-undelayed-small}.  This gives
statement~\eqref{eqn:expectation-single-step-bound} of the lemma.  To obtain
the convergence guarantee~\eqref{eqn:wrong-is-right} in the case of
Assumption~\ref{assumption:resid-strong-convex} and that $\sup_n
C_{\error,n,t} < \infty$, note that
\begin{align*}
  \sum_{k = 1}^n \stepsize_k \E\bigg[\indic{\badratioevent_{n,t}}
    \sum_{i = k - k^\npow}^{k-1} \norm{g_i}\bigg]
  & \lesssim
  \sum_{k = 1}^n \stepsize_k k^{\npow} (C_{\error,k-1,t} + 1) \\
  & \lesssim (C_{\error,n-1,t} + 1)
  \int_1^n u^{\npow -\steppow} du
  \asymp (C_{\error,n-1,t} + 1)
  n^{1 + \npow - \steppow}.
\end{align*}
In particular, we have
\begin{equation*}
  \E\left[\indic{\badratioevent_{n,t}}
    \frac{1}{\sqrt{n}} \sum_{k=1}^n
    \norms{\wt{\error}_k
      - \error_k}\right]
  \lesssim (C_{\error,n-1,t} + 1) t
  n^{\half + \npow - \steppow},
\end{equation*}
which tends to $0$ if $\steppow > \npow + \half$. Thus, we have shown
that for any $\epsilon > 0$, we have for any $t > 0$ that
\begin{equation}
  \label{eqn:n-half-error-probability}
  \lim_{n \to \infty}
  \P\left(\badratioevent_{n, t} ~ \mbox{and} ~
  n^{-\half} \sum_{k=1}^n \norms{\wt{\error}_k - \error_k}
  > \epsilon \right) = 0.
\end{equation}

We now use expression~\eqref{eqn:n-half-error-probability} to get the
desired convergence result in the lemma. Let $D_n = n^{-\half} \sum_{k=1}^n
\norms{\wt{\error}_k - \error_k}$ be shorthand for our error sum. Fix
$\delta > 0$ and let $t$ be large enough that $\P(\badratioevent_{\infty,
  t}) \ge 1 - \delta$, which we know is possible by
Lemma~\ref{lemma:crazy-ratios}. Then as $\badratioevent_{n,t} \subset
\badratioevent_{\infty, t}$ by definition, we have
\begin{align*}
  \P\left(D_n > \epsilon\right)
  & \le \P\left(\badratioevent_{\infty,t} ~ \mbox{and} ~
  D_n > \epsilon\right)
  + \P(\badratioevent_{\infty,t}^c)
  \le \P\left(\badratioevent_{n,t} ~ \mbox{and} ~
  D_n > \epsilon\right)
  + \delta.
\end{align*}
Taking the limit as $n \to \infty$, we find from
expression~\eqref{eqn:n-half-error-probability} that
\begin{equation*}
  \limsup_{n \to \infty}
  \P\left(\frac{1}{\sqrt{n}}
  \sum_{k=1}^n \norms{\wt{\error}_k - \error_k} > \epsilon \right)
  \le \delta,
\end{equation*}
and as $\delta > 0$ was arbitrary, we have the desired
convergence guarantee~\eqref{eqn:wrong-is-right}.

Lastly, we show that
expression~\eqref{eqn:resid-cas-error-true} holds under the
conditions of the lemma.  We use inequality~\eqref{eqn:one-term-error} and
the Borel-Cantelli lemma for this.  Under
Assumption~\ref{assumption:resid-strong-convex} and the additional condition
that $C_{\error,t}^2 = \sup_n \E[\indic{\badratioevent_{n,t}}
  \norms{\error_n}^2] < \infty$ for all $t$
or Assumption~\ref{assumption:resid-lipschitz},
we have
\begin{align*}
  \P\left(\badratioevent_{n,t}, ~ \norms{\wt{\error}_n - \error_n} >
  \epsilon\right)
  & \le \epsilon^{-2} \E\left[\badratio_n^2 \stepsize_n^2
    \left(\sum_{k = n - n^\npow}^{n-1} \norm{g_k}\right)^2
    \indic{\badratioevent_{n,t}}\right] \\
  & \stackrel{(i)}{\le}
  \frac{t^2 \stepsize_n^2 n^\npow \sum_{k = n - n^\npow}^{n-1}
    \E[\norm{g_k}^2 \indic{\badratioevent_{n,t}}]}{\epsilon^2}
  \stackrel{(ii)}{\lesssim} \frac{t^2 \stepsize_n^2 n^{2 \npow}}{\epsilon^2}.
\end{align*}
Here inequality $(i)$ follows from Jensen's inequality and the fact that
$\badratio_n \indic{\badratioevent_{n,t}} \le t$, and inequality $(ii)$
follows either by the bound~\eqref{eqn:gradients-undelayed-small} (when
Assumption~\ref{assumption:resid-strong-convex} holds) or because
$\E[\norm{g_k}^2] \lesssim 1$ for all $k$ (when
Assumption~\ref{assumption:resid-lipschitz} holds).  In particular, we have
\begin{equation*}
  \sum_{n = 1}^\infty \P\left(\badratioevent_{n,t}, ~ \norms{\wt{\error}_n
    - \error_n} > \epsilon\right) \lesssim
  \sum_{n = 1}^\infty n^{2 \npow - 2 \steppow} < \infty
\end{equation*}
whenever $\steppow > \npow + \half$, which we have already assumed. Thus, we
find that under the assumptions of the lemma, we have
$\indic{\badratioevent_{n,t}} \norms{\wt{\error}_n - \error_n} > \epsilon$
only finitely many times, for any $t \in \R$. By
Lemma~\ref{lemma:crazy-ratios}, with probability $1$ there is some $t <
\infty$ such that $\badratio_\infty \le t$, that is, as
$\badratioevent_{\infty, t} \subset \badratioevent_{n,t}$ it must be the
case that $\badratioevent_{n,t}$ always holds. So we find that
$\norms{\wt{\error}_n - \error_n} > \epsilon$ only finitely many times with
probability $1$. That is, $\norms{\wt{\error}_n - \error_n} \cas 0$, and the
continuity of $\lyap$ gives the almost sure
convergence~\eqref{eqn:resid-cas-error-true} as desired, as we
know that $\wt{\error}_n \cas 0$.

\subsection{Proof of Lemmas~\ref{lemma:corrected-errors-matrices}
  and~\ref{lemma:corrected-errors-matrices-lipschitz}}

\label{sec:proof-corrected-errors-matrices}

If we define $B_l^k = \prod_{i=l}^k (I - \stepsize_i
\hess)$, we have that $\wt{\error}_{n+1}' = B_1^n \error_1 - \sum_{k=1}^n
\stepsize_k B_{k+1}^n \noise^k$, and additionally we have
\begin{equation*}
  \sum_{k=1}^n \wt{\error}_k' = \sum_{k=1}^n B_1^{k-1} \error_1
  - \sum_{k = 1}^n \hess^{-1} \noise^k
  + \sum_{k=1}^n W_k^n \noise^k,
\end{equation*}
where the matrix $W_k^n$ is defined by $W_k^n = \stepsize_k \sum_{l =
  k+1}^{n-1} B_{k+1}^l - \hess^{-1}$. This matrix is well-structured,
as the following lemma shows.
\begin{lemma}[Polyak and Juditsky~\cite{PolyakJu92}, Lemma 1]
  \label{lemma:polyak-matrix-lemma}
  Let $\steppow \in (0, 1)$. Then
  \begin{equation*}
    \sup_{k, n} \norm{W_k^n} < \infty
    ~~~ \mbox{and} ~~~
    \lim_n \frac{1}{n} \sum_{k = 1}^n \norm{W_k^n} = 0.
  \end{equation*}
\end{lemma}
\noindent
Thus---as we show rigorously shortly---the behavior
of $\sum_{k=1}^n \wt{\error}_k'$ is governed almost completely
by $\sum_{k=1}^n \hess^{-1} \noise^k$.

Now, by the
iteration~\eqref{eqn:corrected-iterations}, we have that
\begin{equation*}
  \wt{\error}_{k + 1}
  = \wt{\error}_k - \stepsize_k \left(\resid(x_k) + \noise^k
  \right)
  = (I - \stepsize_k \hess) \wt{\error}_k
  + \stepsize_k\underbrace{(H \wt{\error}_k - \resid(x_k))}_{\eqdef Z_k}
  - \stepsize_k \noise^k,
\end{equation*}
so that by analogy with $\wt{\error}_k'$ we have
\begin{equation*}
  \sum_{k=1}^n \wt{\error}_k' = \sum_{k=1}^n B_1^{k-1} \error_1
  - \sum_{k=1}^n \hess^{-1} \noise^k
  + \sum_{k=1}^n W_k^n \noise^k
  - \sum_{k=1}^n \hess^{-1} Z_k
  + \sum_{k=1}^n W_k^n Z_k.
\end{equation*}
Using the iteration~\eqref{eqn:corrected-iterations} for $\wt{\error}_k'$,
we thus have
\begin{equation}
  \sum_{k=1}^n \left( \wt{\error}_k - \wt{\error}_k' \right)
  = \sum_{k = 1}^n (\hess^{-1} - W_k^n) Z_k
  = \sum_{i = 1}^n (\hess^{-1} - W_k^n)
  (\hess \wt{\error}_k - \resid(x_k)).
  \label{eqn:residual-matrix-sum}
\end{equation}
Thus, to show that $\sqrt{n}(\wb{\wt{\error}}_n - \wb{\wt{\error}}_n')
\cp 0$, it suffices to show that the rightmost sum in
expression~\eqref{eqn:residual-matrix-sum}
is $o_P(\sqrt{n})$.

By Lemma~\ref{lemma:polyak-matrix-lemma}, we know that
$\sup_{k,n} \norm{\hess^{-1} - W_k^n} < \infty$, and thus
\begin{align}
  \label{eqn:hojicha}
  \normbigg{\sum_{k=1}^n (\wt{\error}_k - \error_k')}
  \lesssim \sum_{i=1}^n \norm{\hess \wt{\error}_k - \resid(x_k)}
  \le \sum_{k = 1}^n \left(\norm{\hess \wt{\error}_k - \resid(\wt{x}_k)}
  + \norm{\resid(\wt{x}_k) - \resid(x_k)}\right).
\end{align}
We consider each of the right-hand terms in inequality~\eqref{eqn:hojicha}
in turn, beginning with the second.
In this case, Lemma~\ref{lemma:wrong-is-right} implies that under
the conditions of either Lemma~\ref{lemma:corrected-errors-matrices}
or~\ref{lemma:corrected-errors-matrices-lipschitz}, we have
\begin{equation}
  \label{eqn:residuals-together}
  \frac{1}{\sqrt{n}} \sum_{k=1}^n \norm{\resid(\wt{x}_k) - \resid(x_k)}
  \le \frac{L}{\sqrt{n}} \sum_{k=1}^n \norms{\wt{x}_k - x_k}
  = \frac{L}{\sqrt{n}} \sum_{k=1}^n \norms{\wt{\error}_k - \error_k}
  \cp 0.
\end{equation}

We now turn to the error part $H\wt{\error}_k - \resid(\wt{x}_k)$ of
inequality~\eqref{eqn:hojicha}.  To that end, let $\epsilon > 0$ be
the value in Assumption~\ref{assumption:residual-quadratic} such that
$\norm{H(x - x\opt) - \resid(x)} \le C \norm{x - x\opt}^{1 + \gamma}$ for
$x$ such that $\norm{x - x\opt} \le \epsilon$. Splitting the sum into
two parts, Assumption~\ref{assumption:residual-quadratic} thus implies
\begin{align*}
  \sum_{k = 1}^n \norm{\hess \wt{\error}_k - \resid(\wt{x}_k)}
  & \le \sum_{k = 1}^n \norm{\hess \wt{\error}_k - \resid(\wt{x}_k)}
  \indic{\norms{\wt{\error}_k} > \epsilon}
  + C \sum_{k = 1}^n \norms{\wt{\error}_k}^{1 + \gamma}.
\end{align*}
As we know that $\wt{\error}_k \cas 0$, with probability $1$ over
$\{\noise^1, \noise^2, \cdots\} \cup \{\errmat^{ij}\}_{i \ge j}$ (recall the
convergence guarantee~\eqref{eqn:resid-cas-error} after
Lemma~\ref{lemma:p1-convergence}) there exists some (random)
$N(\epsilon) < \infty$ such that $\norms{\wt{\error}_k} \le \epsilon$ for
all $k \ge N(\epsilon)$, so that
\begin{equation*}
  \lim_{n \to \infty}
  \sum_{k = 1}^n \norm{\hess \wt{\error}_k - \resid(\wt{x}_k)}
  \indic{\norms{\wt{\error}_k} > \epsilon} < \infty
  ~~ \mbox{w.p.}~1,
\end{equation*}
and
\begin{equation*}
  n^{-\half}
  \sum_{k = 1}^n \norm{\hess \wt{\error}_k - \resid(\wt{x}_k)}
  \indic{\norms{\wt{\error}_k} > \epsilon} \cas 0.
\end{equation*}
It thus remains to argue that $n^{-\half} \sum_{k=1}^n
\norms{\wt{\error}_k}^{1 + \gamma} \to 0$ in probability (or otherwise).

Let $\epsilon > 0$ be such that $\<\nabla\lyap (x - x\opt), \resid(x)\>
\ge \lambda_0 \lyap(x - x\opt)$ for all $x$ such that $\norm{x - x\opt}
\le \epsilon$ (such an $\epsilon$ certainly exists under both
Assumptions~\ref{assumption:resid-strong-convex}
and~\ref{assumption:resid-lipschitz}). Define the events
\begin{equation*}
  \event_a^k = \left\{\norms{\wt{\error}_i} \le \epsilon,
  ~ \mbox{all~} i \in \{\ceil{a}, \ldots, k \}\right\}.
\end{equation*}
Dividing the sum into two parts, we have
\begin{align*}
  \sum_{k = 1}^n \norms{\wt{\error}_k}^{1 + \gamma}
  & \le \sum_{k = 1}^n \norms{\wt{\error}_k}^{1 + \gamma}
  \indic{\event_{k/2}^{k-1}}
  + \sum_{k = 1}^n \norms{\wt{\error}_k}^{1 + \gamma}
  \left(1 - \indic{\event_{k/2}^{k-1}}\right).
\end{align*}
By the fact that $\wt{\error}_k \cas 0$, we know that there exists some
(random but finite) $N(\epsilon)$ such that $\norms{\wt{\error}_k} <
\epsilon$ for all $k > N(\epsilon)$; the second term in the preceding
display is thus finite with probability one and
\begin{equation*}
  n^{-\half} \sum_{k = 1}^n \norms{\wt{\error}_k}^{1 + \gamma}
  \left(1 - \indic{\event_{k/2}^{k-1}}\right)
  \cas 0.
\end{equation*}
By combining expressions~\eqref{eqn:residual-matrix-sum},
\eqref{eqn:hojicha}, and \eqref{eqn:residuals-together} with the above
display, we see that to prove Lemma~\ref{lemma:corrected-errors-matrices}
or~\ref{lemma:corrected-errors-matrices-lipschitz}, all that remains to show
is that
\begin{equation}
  \label{eqn:to-show-corrected-matrix-convergence}
  n^{-\half} \sum_{k=1}^n \norms{\wt{\error}_k}^{1 + \gamma}
  \indic{\event_{k/2}^{k-1}} \cp 0.
\end{equation}

\subsubsection{Proof of Lemma~\ref{lemma:corrected-errors-matrices}}
\label{sec:proof-corrected-errors-matrices-a}

We give a single-step bound on $\lyap(\wt{\error}_k)$ that we can use to
give the desired convergence guarantee under the conditions of
Lemma~\ref{lemma:corrected-errors-matrices}.
Recall the definitions~\eqref{eqn:ratio-def} and~\eqref{eqn:ratio-event-def}
of $\badratio_n$ and the associated event $\badratioevent_{n,t} =
\{\badratio_n \le t\}$, and recall also our
assumption~\eqref{eqn:delay-moment-interval} that $\npow \in
(\frac{1}{\delmoment-1}, \steppow - \frac{1}{1 + \gamma}) \subset
(\frac{1}{\delmoment - 1}, \steppow - \half)$, where $\npow$ is the power
used in the definition of $\badratio_n$. We claim that there exist constants
$c > 0$ and $C < \infty$, independent of $C_{\error,t}^2 = \sup_k
\E[\indic{\badratioevent_{k,t}} \norm{\error_k}^2]$, such that for $l \in
\R$ with $l < k$, we have
\begin{equation}
  \begin{split}\E\left[\lyap(\wt{\error}_{k+1}) \indic{\event_{l}^k,
        \badratioevent_{k+1,t}}\right]
    & \le \E\left[(1 - c \stepsize_k + C \stepsize_k^2 k^\npow)
      \lyap(\wt{\error}_k)
      \indic{\event_{l}^{k-1}, \badratioevent_{k,t}}\right] \\
    & \qquad ~ + C t^2 \stepsize_k^2 k^{\npow} (C_{\error,t}^2 + 1).
  \end{split}
  \label{eqn:down-with-the-lyapunovs-a}
\end{equation}
We temporarily defer proof of this claim and show how to use it to
show $\sum_{k=1}^n \norms{\wt{\error}_k}^{1 + \gamma} =
o_P(\sqrt{n})$.

Let $K < \infty$ be large enough that for the constants $c, C$ in
inequality~\eqref{eqn:down-with-the-lyapunovs-a},
there is a constant $c'$ such that for $k \ge K$ and stepsizes
$\stepsize_k = \stepsize k^{-\steppow}$, we have
\begin{equation*}
  (1 - c \stepsize_i + C \stepsize_i^2 i^\npow)
  \le \exp(-c' \stepsize_i)
\end{equation*}
for $i \ge k/2$. This must be possible as we have assumed $\npow < \steppow$.
Then by recursively applying
inequality~\eqref{eqn:down-with-the-lyapunovs-a}, we have for $k \ge K$,
\begin{align}
  \lefteqn{\E\left[\lyap(\wt{\error}_k)
      \indic{\event_{k/2}^{k-1}, \badratioevent_{k,t}}\right]} \nonumber \\
  & \le \exp\left(-c \sum_{i = \ceil{k/2}}^{k-1} \stepsize_i \right)
  \E[\lyap(\wt{\error}_{\ceil{k/2}}) \indic{\badratioevent_{\ceil{k/2},t}}]
  + C (C_{\error,t}^2 + 1) t^2 \sum_{i = \ceil{k/2}}^{k-1}
  \stepsize_i^2 i^{\npow} \exp\left(-c\sum_{j = i + 1}^{k-1}
  \stepsize_j\right) \nonumber \\
  & \le
  \exp\left(-c' k^{1 - \steppow} \right)
  \E[\lyap(\wt{\error}_{\ceil{k/2}}) \indic{\badratioevent_{\ceil{k/2},t}}]
  + C \sum_{i = \ceil{k/2}}^{k-1}
  \stepsize_i^2 i^{\npow} \exp\left(-c' \sum_{j = i + 1}^{k-1}
  \stepsize_j\right),
  \label{eqn:lyapunovs-down-a}
\end{align}
where the second inequality follows because $\sum_{i=l}^k \stepsize_i
\asymp k^{1 - \steppow} - l^{1 - \steppow}$, and $k^{1 - \steppow} -
(k/2)^{1 - \steppow} \ge \frac{1 - \steppow}{2} k^{1 - \steppow}$.

Now, by using the assumption that $\E[\indic{\badratioevent_{k,t}}
  \lyap(\wt{\error}_k)] \lesssim
\E[\indic{\badratioevent_{k,t}} \norms{\wt{\error}_k}^2] \le
C_{\error,t}^2$ for all $k$, we find that
\begin{align*}
  \sum_{k=1}^n & \E\left[\norms{\wt{\error}_k}^{1 + \gamma}
      \indic{\badratioevent_{k,t}, \event_{k/2}^{k-1}}\right]
  \le K C_{\error,t}^{1 + \gamma}
  + \sum_{k \ge K}^n \E\left[\norms{\wt{\error}_k}^2
    \indic{\event_{k/2}^{k-1}, \badratioevent_{k,t}}
    \right]^{\frac{1 + \gamma}{2}} \\
  & \stackrel{(i)}{\le} K C_{\error,t}^{1 + \gamma}
  + \sum_{k \ge K}^n \left(
  \exp\left(-c k^{1 - \steppow}\right)
  C_{\error,t}^2 + C \sum_{i = k/2}^k \stepsize_i^2 i^{\npow}
  \exp\left(-c (k^{1 - \steppow} - i^{1 - \steppow})\right)
  \right)^{\frac{1 + \gamma}{2}} \\
  & \le K C_\error^{1 + \gamma}
  + \sum_{k \ge K}
  \exp\left(-\frac{c(1 + \gamma)}{2} k^{1 - \steppow}\right)
  + C \sum_{k \ge K}
  \left(k^{\npow - 2 \steppow}
  \sum_{i = k/2}^k \exp\left(-c (k^{1 - \steppow}
  - i^{1 - \steppow})\right)\right)^{\frac{1 + \gamma}{2}},
\end{align*}
where step~(i) follows from inequality~\eqref{eqn:lyapunovs-down-a}.
By Lemma~\ref{lemma:funny-gamma-integral}, we know that
\begin{equation*}
  \sum_{i = k/2}^k \exp\left(-c(k^{1 - \steppow} - i^{1 - \steppow})\right)
  \lesssim \Gamma\left(\frac{1}{1 - \steppow}\right)
  + k^\steppow,
\end{equation*}
so that
\begin{equation*}
  \sum_{k=1}^n \E\left[\norms{\wt{\error}_k}^{1 + \gamma}
    \indic{\event_{k/2}^{k-1}, \badratioevent_{k,t}}\right]
  \le K C_\error^{1 + \gamma}
  + C \sum_{k = K}^n
  \exp\left(-c k^{1 - \steppow}\right)
  + C \sum_{k = K}^n
  k^{(\npow -\steppow) \frac{1 + \gamma}{2}}.
\end{equation*}
Noting that
$\sum_{k=1}^n k^{(\npow - \steppow) \frac{1 + \gamma}{2}}
\asymp n^{1 - \frac{(\steppow - \npow)(1 + \gamma)}{2}}$,
we have
\begin{equation}
  \label{eqn:l1-guarantee-loads-indics}
  n^{-\half} \sum_{k=1}^n \norms{\wt{\error}_k}^{1 + \gamma}
  \indic{\event_{k/2}^{k-1}, \badratioevent_{k,t}}
  \clp{1} 0
  ~~~ \mbox{if} ~~~
  (\steppow - \npow) \frac{1 + \gamma}{2} > \half,
  ~~ \mbox{or} ~~
  \steppow - \npow > \frac{1}{1 + \gamma}.
\end{equation}
Our initial choice of $\npow$ satisfied this inequality, so we have that
for any $t \in \R$ the preceding convergence guarantee holds.

Now, let $\delta > 0$ be arbitrary and using Lemma~\ref{lemma:crazy-ratios},
choose $t$ large enough that $\P(\badratioevent_{\infty, t}) \le \delta$.
Then
\begin{align*}
  \P\left(n^{-\half} \sum_{k=1}^n \norms{\wt{\error}_k}^{1 + \gamma}
  \indic{\event_{k/2}^{k-1}} > \epsilon\right)
  & \le \P\left(\badratioevent_{\infty,t}, ~ 
  n^{-\half} \sum_{k=1}^n \norms{\wt{\error}_k}^{1 + \gamma}
  \indic{\event_{k/2}^{k-1}} > \epsilon\right)
  + \P(\badratioevent_{\infty,t}^c) \\
  & \le \underbrace{\P\left(
    n^{-\half} \sum_{k=1}^n \norms{\wt{\error}_k}^{1 + \gamma}
    \indic{\event_{k/2}^{k-1}}\indic{\badratioevent_{k,t}}
    > \epsilon\right)}_{
    \to 0 ~ \mbox{as}~ n \to \infty}
  + \delta,
\end{align*}
the convergence to zero a consequence of
inequality~\eqref{eqn:l1-guarantee-loads-indics}.  As $\delta > 0$ was
arbitrary, we see that $n^{-\half} \sum_{k=1}^n \norms{\wt{\error}_k}^{1 +
  \gamma} \indic{\event_{k/2}^{k-1}} \cp 0$,
and expression~\eqref{eqn:to-show-corrected-matrix-convergence}
gives the lemma.

\paragraph{Proof of inequality~\eqref{eqn:down-with-the-lyapunovs-a}.}

Recall the definition of the event $\event_l^k = \{\norms{\wt{\error}_i} \le
\epsilon, ~ \mbox{all}~ i = \ceil{l}, \ldots, k \}$ and
that $\indic{\badratioevent_{k+1,t}} \le \indic{\badratioevent_{k,t}}$. By
inequality~\eqref{eqn:single-step-alternate-lyap-v2} and the fact
that $\event_i^k \in \mc{F}_{k-1}$ for any $i \le k$
and $\badratioevent_{k,t} \in \mc{F}_{k-1}$, we have
\begin{align*}
  \lefteqn{
    \E\left[\lyap(\wt{\error}_{k+1}) \indic{\event_{l}^{k},
      \badratioevent_{k+1,t}}\right]
    \le \E\left[\E\left[\lyap(\wt{\error}_{k+1})
        \mid \mc{F}_{k-1}\right]
    \indic{\event_{l}^k, \badratioevent_{k,t}}\right]} \nonumber \\
  & \le 
  \E\left[\left(\lyap(\wt{\error}_{k})
    - \stepsize_{k} \<\nabla \lyap(\wt{\error}_{k}),
    \resid(\wt{x}_{k})\>
    + C \stepsize_{k}^2 k^\npow \lyap(\wt{\error}_k)
    \right) \indic{\event_{l}^{k}, \badratioevent_{k,t}}\right] \nonumber \\
  & \qquad ~ + C \E\left[
    \stepsize_k^2 \badratio_k^2 \sum_{i = k - k^\npow}^{k - 1}
    \norms{\resid(x_i) + \noise_i}^2
    \indic{\event_{l}^{k-1}, \badratioevent_{k,t}}
    + \stepsize_k^2 (\norm{\error_k}^2 + 1)\indic{\badratioevent_{k,t},
      \event_l^{k-1}} \right],
\end{align*}
where we have used Assumption~\ref{assumption:noise-fun}
that $\noise_k = \noise_k(0) + \noisier(x_k)$, and
$\E[\sup_k \E[\norm{\noise_k(0)}^2 \mid \mc{F}_{k-1}]] < \infty$.

We now use Lemma~\ref{lemma:errors-not-so-bad} to provide
control of the preceding inequality.  Under
Assumption~\ref{assumption:resid-strong-convex} coupled with $\sup_k
\E[\indic{\badratio_{k,t}} \norm{\error_k}^2] \le C_{\error,t}^2 < \infty$,
Lemma~\ref{lemma:errors-not-so-bad} implies that
\begin{equation*}
  \E\left[
    \stepsize_k^2 \badratio_k^2 \sum_{i = k - k^\npow}^{k - 1}
    \norms{\resid(x_i) + \noise_i}^2
    \indic{\event_{l}^{k-1}, \badratio_{k,t}}\right]
  \lesssim \stepsize_k^2 \, t^2 k^{\npow} (C_{\error,t}^2 + 1).
\end{equation*}
Noting that $\norms{\wt{\error}_k}^2 \le C \lyap(\wt{\error}_k)$ and
$\<\nabla \lyap(\wt{\error}_k), \resid(\wt{x}_k)\> \ge \lambda_0
\lyap(\wt{\error}_k)$ by Assumption~\ref{assumption:resid-strong-convex}, we
obtain inequality~\eqref{eqn:down-with-the-lyapunovs-a} as desired.

\subsubsection{Proof of Lemma~\ref{lemma:corrected-errors-matrices-lipschitz}}
\label{sec:proof-corrected-errors-matrices-b}

As in the proof of Lemma~\ref{lemma:corrected-errors-matrices}, we show that
expression~\eqref{eqn:to-show-corrected-matrix-convergence} holds.  As
before, we have $\npow \in (\frac{1}{\delmoment - 1}, \steppow - \frac{1}{1
  + \gamma})$ as the power used in the definition of $\badratio_n$.  We
begin by considering the progress made by a single step of the iteration
with alternate error terms. We first claim that, similar to
inequality~\eqref{eqn:down-with-the-lyapunovs-a}, that there exist constants
$c, C$ such that for any $l < k$ with $l \in \R$, we have
\begin{align}
  \label{eqn:down-with-the-lyapunovs-b}
  \E\left[\lyap(\wt{\error}_{k+1}) \indic{\event_l^k, \badratioevent_{k+1,t}}
  \right]
  & \le (1 - c \stepsize_k + C \stepsize_k^2 k^\npow)
  \E\left[\lyap(\wt{\error}_k)
    \indic{\event_l^{k-1}, \badratioevent_{k,t}}\right]
  + C \stepsize_k^2 t^2 k^\npow.
\end{align}
Indeed, as in the proof of inequality~\eqref{eqn:down-with-the-lyapunovs-a},
we use inequality~\eqref{eqn:single-step-alternate-lyap}
to obtain
\begin{align*}
  \lefteqn{
    \E\left[\lyap(\wt{\error}_{k+1}) \indic{\event_{l}^{k},
      \badratioevent_{k+1,t}}\right]
    \le \E\left[\E\left[\lyap(\wt{\error}_{k+1})
        \mid \mc{F}_{k-1}\right]
      \indic{\event_{l}^k, \badratioevent_{k,t}}\right]} \nonumber \\
  & \le 
  \E\left[\left(\lyap(\wt{\error}_{k})
    - \stepsize_{k} \<\nabla \lyap(\wt{\error}_{k}),
    \resid(\wt{x}_{k})\>
    + C \stepsize_{k}^2 k^\npow \lyap(\wt{\error}_k)
    \right) \indic{\event_{l}^{k}, \badratioevent_{k,t}}\right] \nonumber \\
  & \qquad ~ + C \E\left[
    \stepsize_k^2 \badratio_k^2 \sum_{i = k - k^\npow}^{k - 1}
    \norms{\resid(x_i) + \noise_i}^2
    \indic{\event_{l}^{k-1}, \badratioevent_{k,t}}
    + \stepsize_k^2 (\norm{\resid(x_k)}^2 + 1)\indic{\badratioevent_{k,t},
      \event_l^{k-1}} \right],
\end{align*}
where we have used Assumption~\ref{assumption:resid-lipschitz} that
$\E[\norm{\noise_k}^2 \mid \mc{F}_{k-1}] \lesssim 1$ for all $k$.  Using our
assumption that on the event $\event_l^k$ we have $\<\nabla
\lyap(\wt{\error}_k, \resid(\wt{x}_k)\> \ge \lambda_0 \lyap(\wt{\error}_k)$
and that $\norm{\resid(x)} \lesssim 1$ for all $x$
(Assumption~\ref{assumption:resid-lipschitz}), we obtain
the desired inequality~\eqref{eqn:down-with-the-lyapunovs-b}.

Again paralleling the proof of Lemma~\ref{lemma:corrected-errors-matrices},
let $K < \infty$ be large enough that for the constants $c, C$ in
inequality~\eqref{eqn:down-with-the-lyapunovs-b}, there is a constant $c'$
such that for $k \ge K$ and stepsizes $\stepsize_k = \stepsize
k^{-\steppow}$, we have
$(1 - c \stepsize_i + C \stepsize_i^2 i^\npow)
\le \exp(-c' \stepsize_i)$
for $i \ge k/2$.  Then by recursively applying
inequality~\eqref{eqn:down-with-the-lyapunovs-b}, we have for $k \ge K$,
\begin{align}
  \lefteqn{
    \E\left[\lyap(\wt{\error}_k)
      \indic{\event_{k/2}^{k-1}, \badratioevent_{k,t}}
    \right]} \nonumber \\
  & \le
  \exp\left(-c' k^{1 - \steppow} \right)
  \E\left[\lyap(\wt{\error}_{\ceil{k/2}})\indic{\badratioevent_{\ceil{k/2},t}}
    \right]
  + C \sum_{i = \ceil{k/2}}^{k-1}
  \stepsize_i^2 i^\npow \exp\left(-c' \sum_{j = i + 1}^{k-1}
  \stepsize_j\right),
  \label{eqn:lyapunovs-down-b}
\end{align}
exactly as in the derivation of inequality~\eqref{eqn:lyapunovs-down-a}.

The remainder of the proof is completely identical to that
of Lemma~\ref{lemma:corrected-errors-matrices}.

\subsection{Proof of Lemma~\ref{lemma:smaller-than-past}}
\label{sec:proof-smaller-than-past}

Fix $n \in \N$ and recall our assumption~\eqref{eqn:delay-moment-interval}
that $\npow \in (\frac{1}{\delmoment - 1}, \steppow - \half)$, where $\npow$
is used in the definition of the ratio $\badratio_n$ and event
$\badratioevent_{n,t}$ (Defs.~\eqref{eqn:ratio-def}
and~\eqref{eqn:ratio-event-def}), and consider $\E[\lyap(\wt{\error}_{n+1})
  \indic{\badratioevent_{n+1,t}}]$. Combining
inequality~\eqref{eqn:single-step-alternate-lyap-v2}, the fact that
$\indic{\badratioevent_{n,t}}$ is non-increasing (because
$\badratioevent_{n+1,t} \subset \badratioevent_{n,t}$), and
$\badratioevent_{n,t} \in \mc{F}_{n-1}$, we see that defining
\begin{equation*}
  C_{\error,n,t}^2 =
  \max_{k \le n} \E[\norm{\error_k}^2 \indic{\badratioevent_{k,t}}],
\end{equation*}
we have
\begin{align*}
  \lefteqn{
    \E\left[\lyap(\wt{\error}_{n+1}) \indic{\badratioevent_{n+1,t}}\right]
    \le \E\left[\lyap(\wt{\error}_{n+1}) \indic{\badratioevent_{n,t}}\right]
  } \\
  & \le \E\left[
    \left(\lyap(\wt{\error}_n) - \stepsize_n\<\nabla\lyap(\wt{\error}_n),
    \resid(\wt{x}_n)\> + C \stepsize_n^2 n^\npow
    \norms{\nabla \lyap(\wt{\error}_n)}^2
    \right) \indic{\badratioevent_{n,t}}
    \right]
  + C t^2 \stepsize_n^2 n^\npow (C_{\error,n,t}^2 + 1).
\end{align*}
For the final inequality we have used
inequality~\eqref{eqn:gradients-undelayed-small}, that is, $\E[\norm{g_k}^2
  \indic{\badratioevent_{k,t}}] \le c (C_{\error,n,t}^2 + 1)$.  By using
that $\norm{\nabla \lyap(x - x\opt)}^2 \le L^2 \norm{x - x\opt}^2 \le
\frac{L^2}{\lambda^2} \lyap(x - x\opt)$ by our assumptions on $\lyap$, we
thus obtain
\begin{align*}
  \lefteqn{
    \E\left[\lyap(\wt{\error}_{n+1}) \indic{\badratioevent_{n+1,t}}\right]} \\
  & \le 
  \E\left[\left(
    \lyap(\wt{\error}_n) - \stepsize_n\<\nabla\lyap(\wt{\error}_n),
    \resid(\wt{x}_n)\> + C \stepsize_n^2 n^\npow \lyap(\wt{\error}_n)\right)
    \indic{\badratioevent_{n,t}}
    \right]
  + C t^2 \stepsize_n^2 n^\npow (C_{\error,n,t}^2 + 1).
\end{align*}
Now, noting that $\langle\nabla
\lyap(\wt{\error}_n), \resid(\wt{x}_n)\rangle \ge \lambda_0
\lyap(\wt{\error}_n)$ by Assumption~\ref{assumption:resid-strong-convex},
we have
\begin{align*}
  \E\left[\lyap(\wt{\error}_{n+1}) \indic{\badratioevent_{n+1,t}} \right]
  & \le
  (1 - \lambda_0 \stepsize_n + C \stepsize_n^2 n^\npow)
  \E\left[\lyap(\wt{\error}_n)\indic{\badratioevent_{n,t}} \right]
  + C t^2 \stepsize_n^2 n^\npow (C_{\error,n,t}^2 + 1)
\end{align*}
By recursively applying this inequality,
we have
\begin{equation}
  \begin{split}
    \lefteqn{
      \E\left[\lyap(\wt{\error}_{n+1}) \indic{\badratioevent_{n+1,t}}
        \right]} \\
    & \le \prod_{k = 1}^n (1 - \lambda_0 \stepsize_k + C \stepsize_k^2 k^\npow)
    \E[\lyap(\error_1)]
    + C t^2 (C_{\error,n,t}^2 + 1) \sum_{k = 1}^n \stepsize_k^2 k^\npow
    \prod_{l=k+1}^n
    (1 - \lambda_0 \stepsize_l + C \stepsize_l^2 l^\npow).
  \end{split}
  \label{eqn:strong-recurse-lyap}
\end{equation}
We state a technical lemma the controls the products above. Let
\begin{equation*}
  b_l^k
  \defeq \prod_{i = l}^k (1 - \lambda_0 \stepsize_i + C \stepsize_i^2 i^\npow).
\end{equation*}
\begin{lemma}
  \label{lemma:b-scalars-small}
  Let the scalar sequence $b_l^k$ be defined as above
  with $l \le k$, $\steppow > \npow$,
  and $\stepsize_i =
  \stepsize i^{-\steppow}$ where $C \ge 4\lambda_0^2$ (so that each term in the
  product is non-negative). There exist constants $c_0, c_1, c_2$ (dependent
  on $\steppow$, $\stepsize$, $\lambda_0$, and $C$) such that
  \begin{equation}
    b_l^k \le c_0 \exp\left(-c_1 \sum_{i = l}^k \stepsize_i
    \right)
    \le c_0 \exp\left(-c_2 (k^{1 - \steppow} - l^{1 - \steppow})\right).
    \label{eqn:b-scalars-small}
  \end{equation}
\end{lemma}
\begin{proof}
  This result is similar to a result of Polyak and Juditsky~\cite[proof of
    Lemma 1, Part 3]{PolyakJu92}, but with some differences for additional
  powers in the sequence. As $\stepsize_i = \stepsize i^{-\steppow}$ and we
  have assumed that $\steppow > \npow$, there exists some $K \in \N$ such
  that for $k \ge K$ we have $2 C \stepsize_k^2 k^\npow \le \lambda_0
  \stepsize_k$, or $\stepsize k^{\npow - \steppow} \le \lambda_0 / (2C)$.
  For any $k \ge K$, we have $(1 - \lambda_0 \stepsize_k + C \stepsize_k^2)
  \le 1 - \lambda_0 \stepsize_k / 2 \le \exp(-\frac{\lambda_0}{2}
  \stepsize_k)$.  We find that
  \begin{align*}
    b_l^k
    & = \prod_{i = l}^k (1 - \lambda_0 \stepsize_i + C \stepsize_i^2 i^\npow)
    = \prod_{i \ge l \wedge K}^k
    (1 - \lambda_0 \stepsize_i + C \stepsize_i^2 i^\npow)
    \prod_{i = l}^{(l \wedge K) - 1}
    (1 - \lambda_0 \stepsize_i + C \stepsize_i^2 i^\npow) \\
    & \le \prod_{i = 1}^K
    \max\{1, 1 - \lambda_0 \stepsize_i + C \stepsize_i^2 i^\npow\}
    \prod_{i \ge l \wedge K}^k \exp\left(-\frac{\lambda_0}{2}
    \stepsize_i\right) \\
    & \le \bigg[\exp\bigg(\frac{\lambda_0}{2}\sum_{i=1}^K
      \stepsize_i\bigg) \prod_{i=1}^K
      \max\{1, 1 - \lambda_0 \stepsize_i + C \stepsize_i^2 i^\npow\}
      \bigg]
    \exp\left(-\frac{\lambda_0}{2} \sum_{i = l}^k \stepsize_i\right).
  \end{align*}
  The term in the braces $[\cdot]$ in the preceding product is the constant
  $c_0$, giving the first inequality of
  expression~\eqref{eqn:b-scalars-small}. For the second, note that
  \begin{align*}
    \sum_{i = l}^k \stepsize_i
    = \stepsize \sum_{i = l}^k i^{-\steppow}
    \ge \stepsize \int_l^k t^{-\steppow} dt
    = \frac{\stepsize}{1 - \steppow}\left[k^{1 - \steppow} - l^{1 - \steppow}
      \right],
  \end{align*}
  which completes the proof.
\end{proof}

Applying Lemma~\ref{lemma:b-scalars-small} in
inequality~\eqref{eqn:strong-recurse-lyap}, we obtain
for constants $c, C$ independent of $C_{\error,n,t}$ that
\begin{equation*}
  \E\left[\lyap(\wt{\error}_{n+1})\indic{\badratioevent_{n+1,t}} \right]
  \le C \exp(-c n^{1 - \steppow}) \E[\lyap(\error_1)]
  + C t^2 (C_{\error,n,t}^2 + 1)
  \sum_{k=1}^n \stepsize_k^2 k^\npow \exp\left(-c (n^{1 - \steppow}
  - k^{1 - \steppow})\right).
\end{equation*}
Now we use the technical Lemma~\ref{lemma:funny-zero-sum}, which shows that
the final sum tends to zero when $\stepsize_k = \stepsize k^{-\steppow}$. In
particular, for any $\epsilon > 0$, there exists some $N(\epsilon,t) <
\infty$, independent of $C_{\error,n,t}$, such that $n \ge N(\epsilon,t)$
implies
\begin{equation*}
  C \exp(-c n^{1 - \steppow}) < \epsilon^2
  ~~~ \mbox{and} ~~~
  C t^2 \sum_{k=1}^n \stepsize^2 k^{\npow - 2\steppow}
  \exp\left(-c (n^{1 - \steppow} - k^{1 - \steppow})\right)
  \le \epsilon^2.
\end{equation*}
That is, we have
\begin{equation*}
  \E\left[\lyap(\wt{\error}_{n+1}) \indic{\badratioevent_{n+1,t}}\right]
  \le \epsilon^2 \left(\E[\lyap(\error_1)]
  + C_{\error,n,t}^2 + 1\right).
\end{equation*}
As $\epsilon > 0$ was arbitrary and there are constants $c, C$ such that $c
\norm{x - x\opt}^2 \le \lyap(x - x\opt) \le C \norm{x - x\opt}^2$ this gives
Lemma~\ref{lemma:smaller-than-past}.

\subsection{Proof of Lemma~\ref{lemma:funny-gamma-integral}}
\label{sec:proof-funny-gamma-integral}

We prove the result via a change of variables. Let
$u = c (t^\kappa - a^\kappa)$, so that
\begin{equation*}
  t = \left(u/c + a^\kappa\right)^{\frac{1}{\kappa}},
  ~~~
  du = \kappa c t^{\kappa - 1} dt
  = \kappa c \left(u/c + a^\kappa\right)^{\frac{\kappa - 1}{\kappa}}
  dt,
  ~~~ \mbox{or} ~~~
  dt = (\kappa c)^{-1}
  \left(u / c + a^\kappa\right)^{\frac{1 - \kappa}{\kappa}}
  du.
\end{equation*}
That is, by our change of variables, we have
\begin{align*}
  \int_a^b \exp\left(-c (t^\kappa - a^\kappa)\right) dt
  & = \frac{1}{\kappa c}\int_0^{c(b^\kappa - a^\kappa)}
  \left(\frac{u}{c} + a^\kappa\right)^\frac{1 - \kappa}{\kappa}
  e^{-u}
  du \\
  & \le \frac{\max\{2^{\frac{1 - \kappa}{\kappa} - 1}, 1\}}{\kappa c}
  \left[
    \int_0^{c(b^\kappa - a^\kappa)}
    \left(\frac{u}{c}\right)^\frac{1 - \kappa}{\kappa} e^{-u} du
    +
    \int_0^{c(b^\kappa - a^\kappa)}
    a^{1 - \kappa} e^{-u} du\right],
\end{align*}
where the final inequality follows by convexity of $t \mapsto t^\frac{1 -
  \kappa}{\kappa}$, for $\kappa < \half$ and
the fact that $(t_1 + t_2)^\frac{1 - \kappa}{\kappa}
\le t_1^\frac{1 - \kappa}{\kappa} + t_2^\frac{1 - \kappa}{\kappa}$
for $\kappa \ge \half$ (or $\frac{1 - \kappa}{\kappa} \le 1$).
Noting that $\int_0^\infty u^\frac{1 - \kappa}{\kappa}e^{-u} du
= \Gamma(\frac{1}{\kappa})$ and
$\int_0^\infty e^{-u} du = 1$,
we obtain our desired result.

\subsection{Proof of Lemma~\ref{lemma:funny-zero-sum}}
\label{sec:proof-funny-zero-sum}

\newcommand{\splitmult}{a}

The quantity in the summation diverges or converges identically
to the integral
\begin{equation}
  \label{eqn:break-two-integrals}
  \int_1^n u^{\npow -2 \steppow}
  \exp\left(-c(n^{1 - \steppow} - u^{1 - \steppow})
  \right) du
  \le \int_1^{\splitmult n} \exp\left(-c(n^{1 - \steppow} - u^{1 - \steppow})\right)
  dt + \int_{\splitmult n}^n u^{\npow - 2 \steppow} dt
\end{equation}
for any $\splitmult \in [0, 1]$. Now, by concavity of $u \mapsto u^{1 -
  \steppow}$ for $\steppow \in (\half, 1)$, we have $u^{1 - \steppow} \le
n^{1 - \steppow} + (1 - \steppow) n^{-\steppow}(u - n)$, or $n^{1 -
  \steppow} - u^{1 - \steppow} \ge (1 - \steppow)(n - u) n^{-\steppow}$.
In particular, the first integral on the right side of the
display~\eqref{eqn:break-two-integrals} has bound
\begin{align}
  \int_1^{\splitmult n}
  \exp\left(-c(n^{1 - \steppow} - u^{1 - \steppow})\right) & du
  \le
  \int_1^{\splitmult n} \exp\left(-\frac{c}{n^\steppow}(n - u)\right) du
  = -\int_{\frac{c}{n^\steppow}(n - 1)}^{\frac{c}{n^\steppow}
    (1 - \splitmult) n}
  \frac{n^\steppow}{c} \exp(-u) du \nonumber \\
  & \le \frac{n^\steppow}{c}
  \int_{c(1 - \splitmult) n^{1 - \steppow}}^{c n^{1 - \steppow}}
  e^{-u} du =
  \frac{n^\steppow}{c}
  \left[\exp\left(-c(1 - \splitmult) n^{1 - \steppow}\right)
    - \exp\left(-c n^{1 - \steppow}\right)\right] \nonumber \\
  & \le \frac{n^\steppow}{c} \exp\left(-c(1 - \splitmult) n^{1 - \steppow}
  \right), \label{eqn:first-integral-funny-zero}
\end{align}
where we made a change of variables. For the second integral,
we have
\begin{equation*}
  \int_{\splitmult n}^n u^{\npow - 2 \steppow} du =
  \frac{1}{1 - 2 \steppow}(n^{1 + \npow - 2 \steppow}
  - (\splitmult n)^{1 + \npow - 2 \steppow}),
\end{equation*}
and combining this with~\eqref{eqn:first-integral-funny-zero} in the
bound~\eqref{eqn:break-two-integrals}, we obtain for any $\splitmult \in (0,
1)$ that for constants $C, c$,
\begin{equation*}
  \sum_{k = 1}^n k^{\npow - 2 \steppow} \exp\left(-c(n^{1 - \steppow}
  - k^{1 - \steppow})\right) 
  \le  
  C\left[ n^\steppow \exp\left(-c(1 - \splitmult) n^{1 - \steppow}\right)
    + \frac{n^{1 + \npow - 2 \steppow}}{2\steppow - 1}
    \left(1 - \splitmult^{1 + \npow - 2 \steppow}\right)\right].
\end{equation*}
By our assumption that $\npow < 1$ and $\splitmult \in (0, 1)$, the first
term above converges to zero;
our assumption that $\npow < \steppow - \half$ guarantees that the
second does as well.

\bibliographystyle{abbrvnat}
\bibliography{bib}

\begin{thebibliography}{31}
\providecommand{\natexlab}[1]{#1}
\providecommand{\url}[1]{\texttt{#1}}
\expandafter\ifx\csname urlstyle\endcsname\relax
  \providecommand{\doi}[1]{doi: #1}\else
  \providecommand{\doi}{doi: \begingroup \urlstyle{rm}\Url}\fi

\bibitem[Agarwal and Duchi(2011)]{AgarwalDu11}
A.~Agarwal and J.~C. Duchi.
\newblock Distributed delayed stochastic optimization.
\newblock In \emph{Advances in Neural Information Processing Systems 24}, 2011.

\bibitem[Agarwal et~al.(2012)Agarwal, Bartlett, Ravikumar, and
  Wainwright]{AgarwalBaRaWa12}
A.~Agarwal, P.~L. Bartlett, P.~Ravikumar, and M.~J. Wainwright.
\newblock Information-theoretic lower bounds on the oracle complexity of convex
  optimization.
\newblock \emph{IEEE Transactions on Information Theory}, 58\penalty0
  (5):\penalty0 3235--3249, 2012.

\bibitem[Baldi et~al.(2014)Baldi, Sadowski, and Whiteson]{BaldiSaWh14}
P.~Baldi, P.~Sadowski, and D.~Whiteson.
\newblock Searching for exotic particles in high-energy physics with deep
  learning.
\newblock \emph{Nature Communications}, 5, July 2014.

\bibitem[Ballard et~al.(2011)Ballard, Demmel, Holtz, and
  Schwartz]{BallardDeHoSc11}
G.~Ballard, J.~Demmel, O.~Holtz, and O.~Schwartz.
\newblock Minimizing communication in numerical linear algebra.
\newblock \emph{SIAM Journal on Matrix Analysis and Applications}, 32\penalty0
  (3):\penalty0 866--901, 2011.

\bibitem[Bertsekas and Tsitsiklis(1989)]{BertsekasTs89}
D.~P. Bertsekas and J.~N. Tsitsiklis.
\newblock \emph{Parallel and Distributed Computation: Numerical Methods}.
\newblock Prentice-Hall, Inc., 1989.

\bibitem[Dean et~al.(2012)Dean, Corrado, Monga, Chen, Devin, Le, Mao, Ranzato,
  Senior, Tucker, Yang, and Ng]{DeanCoMoChDeMaRaSeTuYaNg12}
J.~Dean, G.~S. Corrado, R.~Monga, K.~Chen, M.~Devin, Q.~V. Le, M.~Z. Mao,
  M.~Ranzato, A.~Senior, P.~Tucker, K.~Yang, and A.~Y. Ng.
\newblock Large scale distributed deep networks.
\newblock In \emph{Advances in Neural Information Processing Systems 25}, 2012.

\bibitem[Dekel et~al.(2012)Dekel, Gilad-Bachrach, Shamir, and
  Xiao]{DekelGiShXi12}
O.~Dekel, R.~Gilad-Bachrach, O.~Shamir, and L.~Xiao.
\newblock Optimal distributed online prediction using mini-batches.
\newblock \emph{Journal of Machine Learning Research}, 13:\penalty0 165--202,
  2012.

\bibitem[Duchi et~al.(2011)Duchi, Hazan, and Singer]{DuchiHaSi11}
J.~C. Duchi, E.~Hazan, and Y.~Singer.
\newblock Adaptive subgradient methods for online learning and stochastic
  optimization.
\newblock \emph{Journal of Machine Learning Research}, 12:\penalty0 2121--2159,
  2011.

\bibitem[Duchi et~al.(2012)Duchi, Agarwal, and Wainwright]{DuchiAgWa12}
J.~C. Duchi, A.~Agarwal, and M.~J. Wainwright.
\newblock Dual averaging for distributed optimization: convergence analysis and
  network scaling.
\newblock \emph{IEEE Transactions on Automatic Control}, 57\penalty0
  (3):\penalty0 592--606, 2012.

\bibitem[Duchi et~al.(2013)Duchi, Jordan, and McMahan]{DuchiJoMc13_nips}
J.~C. Duchi, M.~I. Jordan, and H.~B. McMahan.
\newblock Estimation, optimization, and parallelism when data is sparse.
\newblock In \emph{Advances in Neural Information Processing Systems 26}, 2013.

\bibitem[Duchi et~al.(2015)Duchi, Chaturapruek, and R\'e]{DuchiChRe15code}
J.~C. Duchi, S.~Chaturapruek, and C.~R\'e.
\newblock Asynchronous stochastic convex optimization, 2015.
\newblock URL
  \url{https://www.codalab.org/worksheets/0x610bcdb722bf48d3b537a65edf0fe72d/}.
\newblock Code for reproducing experiments.

\bibitem[Ermoliev(1969)]{Ermoliev69}
Y.~M. Ermoliev.
\newblock On the stochastic quasi-gradient method and stochastic quasi-{F}eyer
  sequences.
\newblock \emph{Kibernetika}, 2:\penalty0 72--83, 1969.

\bibitem[Ghadimi and Lan(2012)]{GhadimiLa12}
S.~Ghadimi and G.~Lan.
\newblock Optimal stochastic approximation algorithms for strongly convex
  stochastic composite optimization, {I}: a generic algorithmic framework.
\newblock \emph{SIAM Journal on Optimization}, 22\penalty0 (4):\penalty0
  1469--1492, 2012.

\bibitem[Hazan and Kale(2011)]{HazanKa11}
E.~Hazan and S.~Kale.
\newblock An optimal algorithm for stochastic strongly convex optimization.
\newblock In \emph{Proceedings of the Twenty Fourth Annual Conference on
  Computational Learning Theory}, 2011.

\bibitem[Hiriart-Urruty and Lemar\'echal(1993)]{HiriartUrrutyLe93}
J.~Hiriart-Urruty and C.~Lemar\'echal.
\newblock \emph{Convex Analysis and Minimization Algorithms {I}}.
\newblock Springer, New York, 1993.

\bibitem[Juditsky et~al.(2011)Juditsky, Nemirovski, and Tauvel]{JuditskyNeTa11}
A.~Juditsky, A.~Nemirovski, and C.~Tauvel.
\newblock Solving variational inequalities with the stochastic mirror-prox
  algorithm.
\newblock \emph{Stochastic Systems}, 1\penalty0 (1):\penalty0 17--58, 2011.

\bibitem[{Le Cam} and Yang(2000)]{LeCamYa00}
L.~{Le Cam} and G.~L. Yang.
\newblock \emph{Asymptotics in Statistics: Some Basic Concepts}.
\newblock Springer, 2000.

\bibitem[Lehmann and Casella(1998)]{LehmannCa98}
E.~L. Lehmann and G.~Casella.
\newblock \emph{Theory of Point Estimation, Second Edition}.
\newblock Springer, 1998.

\bibitem[Lewis et~al.(2004)Lewis, Yang, Rose, and Li]{LewisYaRoLi04}
D.~Lewis, Y.~Yang, T.~Rose, and F.~Li.
\newblock {RCV}1: A new benchmark collection for text categorization research.
\newblock \emph{Journal of Machine Learning Research}, 5:\penalty0 361--397,
  2004.

\bibitem[Lichman(2013)]{Lichman13}
M.~Lichman.
\newblock {UCI} machine learning repository, 2013.
\newblock URL \url{http://archive.ics.uci.edu/ml}.

\bibitem[Liu et~al.(2014)Liu, Wright, R\'e, Bittorf, and
  Sridhar]{LiuWrReBiSr14}
J.~Liu, S.~J. Wright, C.~R\'e, V.~Bittorf, and S.~Sridhar.
\newblock An asynchronous parallel stochastic coordinate descent algorithm.
\newblock In \emph{Proceedings of the 31st International Conference on Machine
  Learning}, 2014.

\bibitem[Nemirovski and Yudin(1983)]{NemirovskiYu83}
A.~Nemirovski and D.~Yudin.
\newblock \emph{Problem Complexity and Method Efficiency in Optimization}.
\newblock Wiley, 1983.

\bibitem[Nemirovski et~al.(2009)Nemirovski, Juditsky, Lan, and
  Shapiro]{NemirovskiJuLaSh09}
A.~Nemirovski, A.~Juditsky, G.~Lan, and A.~Shapiro.
\newblock Robust stochastic approximation approach to stochastic programming.
\newblock \emph{SIAM Journal on Optimization}, 19\penalty0 (4):\penalty0
  1574--1609, 2009.

\bibitem[Nesterov(2009)]{Nesterov09}
Y.~Nesterov.
\newblock Primal-dual subgradient methods for convex problems.
\newblock \emph{Mathematical Programming}, 120\penalty0 (1):\penalty0 261--283,
  2009.

\bibitem[Niu et~al.(2011)Niu, Recht, Re, and Wright]{NiuReReWrNi11}
F.~Niu, B.~Recht, C.~Re, and S.~Wright.
\newblock Hogwild: a lock-free approach to parallelizing stochastic gradient
  descent.
\newblock In \emph{Advances in Neural Information Processing Systems 24}, 2011.

\bibitem[Polyak and Juditsky(1992)]{PolyakJu92}
B.~T. Polyak and A.~B. Juditsky.
\newblock Acceleration of stochastic approximation by averaging.
\newblock \emph{SIAM Journal on Control and Optimization}, 30\penalty0
  (4):\penalty0 838--855, 1992.

\bibitem[Recht and R\'{e}(2012)]{RechtRe12}
B.~Recht and C.~R\'{e}.
\newblock Beneath the valley of the noncommutative arithmetic-geometric mean
  inequality: conjectures, case-studies, and consequences.
\newblock In \emph{Proceedings of the Twenty Fifth Annual Conference on
  Computational Learning Theory}, 2012.

\bibitem[Richt\'arik and Tak\'a\v{c}(2015)]{RichtarikTa15}
P.~Richt\'arik and M.~Tak\'a\v{c}.
\newblock Parallel coordinate descent methods for big data optimization.
\newblock \emph{Mathematical Programming}, page Online first, 2015.
\newblock URL \url{http://link.springer.com/article/10.1007/s10107-015-0901-6}.

\bibitem[Robbins and Monro(1951)]{RobbinsMo51}
H.~Robbins and S.~Monro.
\newblock A stochastic approximation method.
\newblock \emph{Annals of Mathematical Statistics}, 22:\penalty0 400--407,
  1951.

\bibitem[Robbins and Siegmund(1971)]{RobbinsSi71}
H.~Robbins and D.~Siegmund.
\newblock A convergence theorem for non-negative almost supermartingales and
  some applications.
\newblock In \emph{Optimizing Methods in Statistics}, pages 233--257. Academic
  Press, New York, 1971.

\bibitem[van~der Vaart(1998)]{VanDerVaart98}
A.~W. van~der Vaart.
\newblock \emph{Asymptotic Statistics}.
\newblock Cambridge Series in Statistical and Probabilistic Mathematics.
  Cambridge University Press, 1998.
\newblock ISBN 0-521-49603-9.

\end{thebibliography}

\end{document}